\documentclass[12pt]{amsart}

\usepackage{amsmath,amssymb,amsthm,enumerate,graphicx}
\usepackage[margin=1in]{geometry}

\numberwithin{equation}{section}

\DeclareMathOperator{\supp}{supp}

\newcommand{\sm}{\setminus}

\newcommand{\C}{\mathbb{C}}
\newcommand{\D}{\mathcal{D}}

\newcommand{\cE}{\mathcal{E}}
\newcommand{\N}{\mathbb{N}}
\newcommand{\Z}{\mathbb{Z}}

\newcommand{\Q}{\mathcal{Q}}
\newcommand{\R}{\mathbb{R}}

\newcommand{\F}{\mathbb{F}}

\newcommand{\E}{\mathbb{E}}
\newcommand{\A}{\mathcal{A}}
\newcommand{\T}{\mathbb{T}}

\newcommand{\ve}{\varepsilon}

\newcommand{\1}{^{-1}}

\newcommand{\f}[2]{\frac{#1}{#2}}

\newcommand{\ul}[1]{\underline{#1}}

\newtheorem{theorem}{Theorem}[section]

\newtheorem{definition}[theorem]{Definition}

\newtheorem{corollary}[theorem]{Corollary}
\newtheorem{proposition}[theorem]{Proposition}
\newtheorem{lemma}[theorem]{Lemma}

\numberwithin{theorem}{section}

\title{Bounds for sets with no polynomial progressions}
\author{Sarah Peluse}
\address{School of Mathematics, Institute for Advanced Study, 1 Einstein Drive, Princeton, NJ 08540, USA}
\address{Department of Mathematics, Princeton University, Fine Hall, Washington Road, Princeton, NJ 08540, USA}
\email{speluse@princeton.edu}

\begin{document}
\begin{abstract}
Let $P_1,\dots,P_m\in\Z[y]$ be polynomials with distinct degrees, each having zero constant term. We show that any subset $A$ of $\{1,\dots,N\}$ with no nontrivial progressions of the form $x,x+P_1(y),\dots,x+P_m(y)$ has size $|A|\ll N/(\log\log{N})^{c_{P_1,\dots,P_m}}$. Along the way, we prove a general result controlling weighted counts of polynomial progressions by Gowers norms.
\end{abstract}

\maketitle

\section{Introduction}\label{sec1}

For any polynomials $P_1,\dots,P_m\in\Z[y]$, let $r_{P_1,\dots,P_m}(N)$ denote the size of the largest subset of $[N]:=\{1,\dots,N\}$ containing no progressions of the form $x,x+P_1(y),\dots,x+P_m(y)$ with $y\neq 0$. Bergelson and Leibman~\cite{BergelsonLeibman96} showed that
\[
  r_{P_1,\dots,P_m}(N)=o_{P_1,\dots,P_m}(N)
\]
whenever $P_1,\dots,P_m\in\Z[y]$ all have zero constant term. This is a polynomial generalization of Szemer\'edi's theorem~\cite{Szemeredi75} on arithmetic progressions, which states that $r_{y,2y,\dots,(k-1)y}(N)=o_k(N)$ for every $k\in\N$. While quantitative bounds in Szemer\'edi's theorem for all $k\in\N$ are known due to work of Gowers~\cite{Gowers98,Gowers01}, no bounds are known in general for the polynomial Szemer\'edi theorem. Thus, Gowers~\cite{Gowers01S} has posed the problem of proving explicit bounds for the quantities $r_{P_1,\dots,P_m}(N)$.

In this paper, we prove quantitative bounds for $r_{P_1,\dots,P_m}(N)$ whenever $P_1,\dots,P_m$ have distinct degrees, giving the first quantitative version of the polynomial Szemer\'edi theorem for this large class of progressions.
\begin{theorem}\label{thm1.1}
Let $P_1,\dots,P_m\in\Z[y]$ be polynomials with distinct degrees, each having zero constant term. There exists a $c_{P_1,\dots,P_m}>0$ such that
\[
r_{P_1,\dots,P_m}(N)\ll\f{N}{(\log\log{N})^{c_{P_1,\dots,P_m}}}.
\]
\end{theorem}

Obviously, any polynomial progression involving only linear polynomials is a subprogression of some arithmetic progression, so that bounds for Szemer\'edi's theorem (such as the current best bounds of Bloom~\cite{Bloom16} for $3$-term progressions, Green and Tao~\cite{GreenTao17} for $4$-term progressions, and Gowers~\cite{Gowers01} for longer progressions) imply bounds in the linear case of the polynomial Szemer\'edi theorem. Until recently, very few cases beyond this were known. Indeed, quantitative versions of the polynomial Szemer\'edi theorem were known in only two other situations: for two-term polynomial progressions~\cite{Sarkozy78I,Sarkozy78II,BalogPelikanPintzSzemeredi94,Slijepcevic03,Lucier06,Rice19}, to which Fourier analytic methods immediately apply, and for arithmetic progressions with common difference equal to a perfect power~\cite{Prendiville17} (and thus all subprogressions of those progressions), to which Gowers's method~\cite{Gowers01} may be adapted to apply.

It was essential for the success of the density increment arguments in~\cite{Gowers98} and~\cite{Gowers01} that $k$-term arithmetic progressions are preserved under translation and dilation, since the inverse theorems for the Gowers norms (both local and global) give a density increment on an arithmetic progression whose common difference can be much larger than the length of the progression. Similarly, $k$-term arithmetic progressions with common difference equal to a perfect $d^{th}$ power are preserved under translation and dilation by a perfect $d^{th}$ power, so that Gowers's local inverse theorem from~\cite{Gowers01} could be applied in~\cite{Prendiville17} with suitable modification to get a density increment on a progression with common difference equal to a perfect $d^{th}$ power. However, the vast majority of polynomial progressions do not behave so nicely under dilation (e.g., $x,x+y,x+y^2$), and so to handle more progressions of length greater than two, new strategies avoiding the use of the inverse theorems for the Gowers norms were needed.

Recently, significant progress has been made on the problem of proving a quantitative version of the polynomial Szemer\'edi theorem in the finite field setting. Similar to above, let $r_{P_1,\dots,P_m}(\F_p)$ denote the size of the largest subset of $\F_p$ containing no nontrivial progressions of the form $x,x+P_1(y),\dots,x+P_m(y)$. Bourgain and Chang~\cite{BourgainChang17} proved that $r_{y,y^2}(\F_p)\ll p^{14/15}$, the author~\cite{Peluse18} proved that $r_{P_1,P_2}(\F_p)\ll p^{23/24}$ whenever $P_1$ and $P_2$ are affine-linearly independent over $\mathbb{Q}$, and then Dong, Li, and Sawin~\cite{DongLiSawin17} very shortly after and independently showed improved bounds, getting $r_{P_1,P_2}(\F_p)\ll_{P_1,P_2} p^{11/12}$. All three of these arguments completely avoided the use of any inverse theorems for the Gowers norms. However, there were serious barriers to generalizing any of the methods of~\cite{BourgainChang17,Peluse18,DongLiSawin17} to the integer setting or to longer progressions in the finite field setting.

Using a different method, the author~\cite{Peluse19} proved that $r_{P_1,\dots,P_m}(\F_p)\ll p^{1-\gamma_{P_1,\dots,P_m}}$ whenever $P_1,\dots,P_m\in\Z[y]$ are affine-linearly independent. Theorem~\ref{thm1.1} thus brings our knowledge of the polynomial Szemer\'edi theorem in the integers more in line with what is known in finite fields. The proof of Theorem~\ref{thm1.1} involves adapting the central idea of~\cite{Peluse19} to the integer setting. Such an adaptation was first done by Prendiville and the author~\cite{PelusePrendiville19} for the special case of the progression $x,x+y,x+y^2$, showing that $r_{y,y^2}(N)\ll N/(\log\log{N})^c$ for some absolute constant $c>0$. It turns out that the assumption that $P_1,\dots,P_m$ have distinct degrees in Theorem~\ref{thm1.1} is the exact condition needed to adapt the argument of~\cite{Peluse19} to the integers in full. We will say more about why this is the case in Section~\ref{sec3}.

We now briefly discuss the proof of Theorem~\ref{thm1.1} in comparison to the arguments in~\cite{Peluse19} and~\cite{PelusePrendiville19}. The proof of Theorem~\ref{thm1.1} proceeds via a density increment argument where, as in~\cite{PelusePrendiville19}, it is shown that any subset of $[N]$ with no nontrivial polynomial progressions has increased density on a long arithmetic progression with very small common difference. This is done by following the strategy for proving quantitative bounds in the polynomial Szemer\'edi theorem originating in~\cite{Peluse19}, which is to first show that the count of polynomial progressions in a set is controlled by some Gowers $U^s$-norm, and then to show that, in certain situations, one can combine this $U^s$-control with understanding of shorter progressions to deduce $U^{s-1}$-control. We refer to this second part of the argument as a ``degree-lowering'' result, and it is here that it is crucial that $P_1,\dots,P_m$ have distinct degrees. A key feature of the proof of the degree-lowering result is that, while the $U^s$-norm plays a role in the argument for arbitrarily large $s$, it bypasses the use of any inverse theorems for uniformity norms of degree greater than $2$. Starting with control by any $U^s$-norm, one can repeatedly apply the degree-lowering result to deduce control in terms of the $U^2$- or $U^1$-norm, which are much easier to deal with than higher degree uniformity norms.

In contrast to the finite field situation of~\cite{Peluse19}, the main challenge in this paper is to first prove that the count of polynomial progressions is controlled by some $U^s$-norm. By using repeated applications of the van der Corput inequality following Bergelson and Leibman's~\cite{BergelsonLeibman96} PET induction scheme, we can prove control in terms of an average of a certain family of Gowers box norms. In~\cite{TaoZiegler18}, Tao and Ziegler use the results of their paper on concatenation~\cite{TaoZiegler16} to prove that such an average is qualitatively controlled by a global $U^s$-norm, but with no quantitative bounds. The results of~\cite{TaoZiegler16} are purely qualitative, and so not suitable for our purposes. In this paper, we prove a new quantitative concatenation result, which we use to control (with polynomial bounds) the averages of Gowers box norms just mentioned by a $U^s$-norm for some $s$ depending only on the degrees of the polynomials involved. In~\cite{PelusePrendiville19}, this was done for the special case of the average of Gowers box norms controlling the progression $x,x+y,x+y^2$, which is the simplest case requiring a nontrivial concatenation argument. In the general situation covered by Theorem~\ref{thm1.1}, these averages of Gowers box norms can become arbitrarily complex, necessitating a new and more general approach. The concatenation theory developed in this paper is significantly stronger than that in~\cite{PelusePrendiville19}, and the bulk of the new ideas in this paper go into proving these concatenation results. We must also be more careful during the PET induction step than in previous works in order to produce an average of Gowers box norms of the particular form that our concatenation result can be applied to. Though the proof of Theorem~\ref{thm1.1} only requires a $U^s$-control result for polynomial progressions involving polynomials with distinct degrees, a result for general polynomial progressions can be proved with a little more work using our methods. Since it may be of independent interest, we record this result in Theorem~\ref{thm6.2}.

In~\cite{PelusePrendiville19}, the author and Prendiville adapted the degree-lowering method of~\cite{Peluse19} to handle the progression $x,x+y,x+y^2$ in the integer setting. The adaptation in that paper quickly breaks down for essentially all other non-linear progressions, however. To prove a degree-lowering result that works in the generality of Theorem~\ref{thm1.1}, we must prove several intermediate degree-lowering results by induction. This induction is intertwined with an induction proving several intermediate ``major arc lemmas''. These lemmas are ingredients in the proofs of the intermediate degree-lowering results whose proofs themselves require other intermediate ``major arc lemmas'' and degree-lowering results, along with the $U^s$-control result mentioned in the previous paragraph. Despite the additional complications of this inductive argument, the proof of each intermediate degree-lowering result (assuming the corresponding major arc lemma) is still based on the proof of the degree-lowering result of~\cite{PelusePrendiville19}.

This paper is organized as follows. In Section~\ref{sec2}, we set notation and recall some facts about the Gowers uniformity and box norms. In Section~\ref{sec3}, we give a detailed outline of the proof of Theorem~\ref{thm1.1}, stating the most important intermediate results needed. In Section~\ref{sec4}, we prove that weighted counts of the polynomial progressions we consider are controlled by an average of a certain family of Gowers box norms. In Section~\ref{sec5}, we prove our main concatenation result, which we combine with the results of Section~\ref{sec4} to deduce control by uniformity norms in Section~\ref{sec6}. In Section~\ref{sec7}, we prove several lemmas needed to carry out the degree-lowering argument, and in Section~\ref{sec8} we prove our general degree-lowering result. We repeatedly combine the degree-lowering result with the $U^s$-control result proven in Section~\ref{sec6} to deduce a local $U^1$-control result in Section~\ref{sec9}. In Section~\ref{sec10}, we use this local $U^1$-control result to carry out the density increment argument, completing the proof of Theorem~\ref{thm1.1}.

\section*{Acknowledgments}
The author thanks Sean Prendiville and Kannan Soundararajan for helpful comments on earlier versions of this paper and the anonymous referees for many useful suggestions that improved the presentation in this paper, including one that simplified the proof of Corollary~\ref{cor3.8}. The author was partially supported by the NSF Graduate Research Fellowship Program under Grant No. DGE-114747 and by the Stanford University Mayfield Graduate Fellowship.

\section{Notation and preliminaries}\label{sec2}

We are interested in the regime where $N\to\infty$, and so we will assume that $N$ is sufficiently large so that, for example, the quantity $\log\log{N}$ is well-defined and positive. The standard asymptotic notation $O$ and $\Omega$, along with $\ll,\gg,$ and $\asymp$, will be used throughout the paper. So, $A=O(B)$, $B=\Omega(A)$, $A\ll B$, and $B\gg A$ all mean that $|A|\leq C|B|$ for some absolute constant $C>0$, and $A\asymp B$ means that $A\ll B$ and $B\ll A$. When $O$, $\Omega$, $\ll$, $\gg$, or $\asymp$ appear with a subscript, this means that the implied constant $C$ may depend on the subscript. We will also use expressions of the form $O(A)$ to denote a quantity that has size at most an absolute constant times $A$, and analogously for $\Omega(A)$. 

For any function $f:\Z^n\to\C$ and finite subset $S\subset\Z^n$, we denote the average of $f$ over $S$ by $\E_{x\in S}f(x):=\f{1}{|S|}\sum_{x\in S}f(x)$, and if $\mu:\Z^n\to[0,\infty)$ is finitely supported, we similarly denote the average of $f$ with respect to $\mu$ by $\E_{x}^\mu f(x):=\sum_{x\in\Z^n}f(x)\mu(x)$. We say that $f$ is \textit{$1$-bounded} if $\|f\|_{L^\infty}\leq 1$.  We normalize the $\ell^p$-norms on the space of functions $\Z^n\to\C$ by setting $\|f\|_{\ell^p}^p:=\sum_{x\in\Z^n}|f(x)|^p$. For any $L>0$, we define the weight $\mu_L:\Z\to[0,1]$ by
\[
\mu_L(h):=\f{\#\{(h_1,h_2)\in[L]^2: h_1-h_2=h\}}{L^2},
\]
so that $\supp \mu_L\subset(-L,L)$, $\|\mu_L\|_{\ell^1}=1$, and $\|\mu_{L}\|_{\ell^2}^2\leq 1/L$. Set $e(x):=e^{2\pi i x}$. When $f:\Z\to\C$ is finitely supported, we define its Fourier transform $\hat{f}:\T\to\C$ by
\[
\hat{f}(\xi):=\sum_{x\in\Z}f(x)e(-\xi x),
\]
and the convolution of $f$ with another finitely supported function $g:\Z\to\C$ by
\[
(f*g)(x):=\sum_{y\in\Z}f(y)g(x-y).
\]
With this choice of normalizations, note that $\widehat{f*g}=\widehat{f}\cdot\widehat{g}$, $f(x)=\int_{\T}\widehat{f}(\xi)e(\xi x)d\xi$ for all $x\in\Z$, and $\sum_{x\in\Z}f(x)\overline{g(x)}=\int_{\T}\widehat{f}(\xi)\overline{\widehat{g}(\xi)}d\xi$.

For any $f:\Z\to\C$ and $h\in\Z$, we define functions $T_hf:\Z\to\C$ and $\Delta_hf:\Z\to\C$ by $T_hf(x)=f(x+h)$ and $\Delta_hf(x):=\overline{f(x+h)}f(x)$, and also define, for $h_1,\dots,h_s$, the function $\Delta_{h_1,\dots,h_s}f:\Z\to\C$ by $\Delta_{h_1,\dots,h_s}f=\Delta_{h_1}\cdots\Delta_{h_s}f$. Note that $\Delta_{h_1}\Delta_{h_2}f=\Delta_{h_2}\Delta_{h_1}f$ for any $h_1,h_2\in\Z$. Thus, for any finite subset $I\subset\Z$, we may unambiguously define $\Delta_{(h_i)_{i\in I}}f$ to equal $\Delta_{h_{i_1},\dots,h_{i_{|I|}}}f$ where $i_1,\dots,i_{|I|}$ is any enumeration of the elements of $I$. In the same vein, we will use the notation $\Delta_{\ul{h}}f$ when $\ul{h}=(h_1,\dots,h_k)$ to denote the function $\Delta_{h_1,\dots,h_k}f$. Finally, for any $(h_1,h_1')\in\Z^2$ we similarly define $\Delta_{(h_1,h_1')}'f:\Z\to\C$ by $\Delta_{(h_1,h_1')}'f(x):=\overline{f(x+h_1)}f(x+h_1')$, and also define $\Delta'_{(h_1,h_1'),\dots,(h_s,h'_s)}f$ and $\Delta'_{(h_i,h'_i)_{i\in I}}f$ analogously to $\Delta_{h_1,\dots,h_s}f$ and $\Delta_{(h_i)_{i\in I}}f$.

We can now define the Gowers box and uniformity norms.

\begin{definition}\label{def2.1}
Let $d\in\N$, $Q_1,\dots,Q_d\subset\Z$ be finite subsets, and $f:\Z\to\C$ be a function supported on a finite subset $S\subset\Z$. We define the (normalized) Gowers box norm of $f$ with respect to $Q_1,\dots,Q_d$ by
\[
\|f\|^{2^d}_{\square^d_{Q_1,\dots,Q_d}(S)}:=\f{1}{|S|}\sum_{x\in\Z}\E_{\substack{h_i,h_i'\in Q_i \\ i=1,\dots,d}}\Delta'_{(h_1,h_1'),\dots,(h_d,h_d')}f(x).
\]
When $Q\subset\Z$ is any finite subset, we define the Gowers $U^s$-norm of $f$ with respect to $Q$ by
\[
\|f\|_{U_Q^s(S)}:=\|f\|_{\square_{Q,\dots,Q}^s(S)}.
\]
\end{definition}

We will occasionally use the Gowers--Cauchy--Schwarz inequalities, which we now recall. The following two results are standard (see Lemma~B.2 of~\cite{GreenTao10}, for example).
\begin{lemma}\label{lem2.2}
  Let $X_1,\dots,X_s$ be finite sets, $f:\prod_{i=1}^sX_i\to\C$, and, for each $i\in[s]$, $g_i:\prod_{i=1}^sX_i\to\C$ be a $1$-bounded function such that the value of $g_i(x_1,\dots,x_s)$ does not depend on $x_i$. We have
  \[
    \left|\E_{\substack{x_i\in X_i \\ i=1,\dots,s}}f(x_1,\dots,x_s)\prod_{i=1}^sg_i(x_1,\dots,x_s)\right|^{2^s}\leq\E_{\substack{x_i^0,x_i^1\in X_i \\ i=1,\dots,s}}\prod_{\omega\in\{0,1\}^s}\mathcal{C}^{|\omega|}f(x_1^{\omega_1},\cdots,x_s^{\omega_s}).
  \]
\end{lemma}

\begin{lemma}\label{lem2.3}
Let $Q_1,\dots,Q_d\subset\Z$ be finite subsets and, for each $\omega\in\{0,1\}^d$, $f_\omega:\Z\to\C$ be a function supported on a finite subset $S\subset\Z$. We have
\[
\left|\f{1}{|S|}\sum_{x\in\Z}\E_{\substack{h_i,h_i'\in Q_i \\ i=1,\dots,d}}\prod_{\omega\in\{0,1\}^d}\mathcal{C}^{|\omega|}f_\omega(x+\ul{h}\cdot\omega+\ul{h}'\cdot(\ul{1}-\omega))\right|\leq\prod_{\omega\in\{0,1\}^d}\|f_\omega\|_{\square^d_{Q_1,\dots,Q_d}(S)}.
\]
\end{lemma}
In the above lemmas and elsewhere in the paper, $\mathcal{C}:\C\to\C$ denotes the complex conjugation operator and $\ul{1}$ denotes the tuple with entries all equal to $1$, whose dimensions will be clear from context. Similarly, $\ul{0}$ denotes the tuple with entries all equal to $0$.

Finally, we will need an inverse theorem for $U^2$-norms of the form $\|\cdot\|_{U^2_{[\delta' L]}([L])}$. This is the only inverse result for uniformity norms used in the proof of Theorem~\ref{thm1.1}.
\begin{lemma}\label{lem2.4}
Let $L>0$. If $f:\Z\to\C$ is $1$-bounded, supported on the interval $[L]$, and satisfies
\[
\|f\|_{U^2_{[\delta' L]}([L])}\geq\delta,
\]
then there exists a $\beta\in\T$ such that
\[
\left|\E_{x\in [L]}f(x)e(\beta x)\right|\gg(\delta\delta')^{O(1)}.
\]
\end{lemma}
\begin{proof}
  By making the change of variables $x\mapsto x-h_1'-h_2'$ in the definition of $\|\cdot\|_{U^2_{[\delta'L]}([L])}$, we have
  \[
\f{1}{L}\sum_{x,h_1,h_2\in\Z}\Delta_{h_1,h_2}f(x)\mu_{\delta'L}(h_1)\mu_{\delta'L}(h_2)\geq\delta^4.
  \]
 By Fourier inversion, it follows that
  \[
\left(\int_{\T}|\widehat{\mu_{\delta'L}}(\xi)|d\xi\right)^2\cdot\max_{\xi_1,\xi_2\in\T}\left|\f{1}{L}\sum_{x,h_1,h_2\in\Z}\Delta_{h_1,h_2}f(x)e(\xi_1h_1)e(\xi_2h_2)\right|\geq\delta^4.
  \]
 Note that
  \[
\int_{\T}|\widehat{\mu_{\delta'L}}(\xi)|d\xi=\int_{\T}\f{|\widehat{1_{[\delta' L]}}(\xi)|^2}{(\delta'L)^2}d\xi=\f{\|1_{[\delta'L]}\|_{\ell^2}^2}{(\delta'L)^2}=\f{1}{\delta'L},
\]
 since $\mu_{\delta'L}=(1_{[\delta'L]}*1_{-[\delta'L]})/(\delta'L)^2$. Thus,
  \[
\left|\f{1}{L^3}\sum_{x,h_1,h_2\in\Z}f(x)e((\xi_1+\xi_2)x)\overline{f(x+h_1)e(\xi_1(x+h_1))f(x+h_2)e(\xi_2(x+h_2))}f(x+h_1+h_2)\right|
  \]
is at least $(\delta')^2\delta^4$ for some $\xi_1,\xi_2\in\T$. The result now follows by applying the Gowers--Cauchy--Schwarz inequality and $U^2$-inverse theorem in $\Z/5\lceil L\rceil\Z$ (see~\cite{Tao12}, for example, for these standard results).
\end{proof}

\section{Outline of the proof of Theorem~\ref{thm1.1}}\label{sec3}
To hopefully aid the reader, Figure~\ref{fig1} below shows the logical dependencies between the key intermediate results stated in this section, as well as Theorem~\ref{thm1.1}.
\begin{center}
	\begin{figure}[h]
\includegraphics{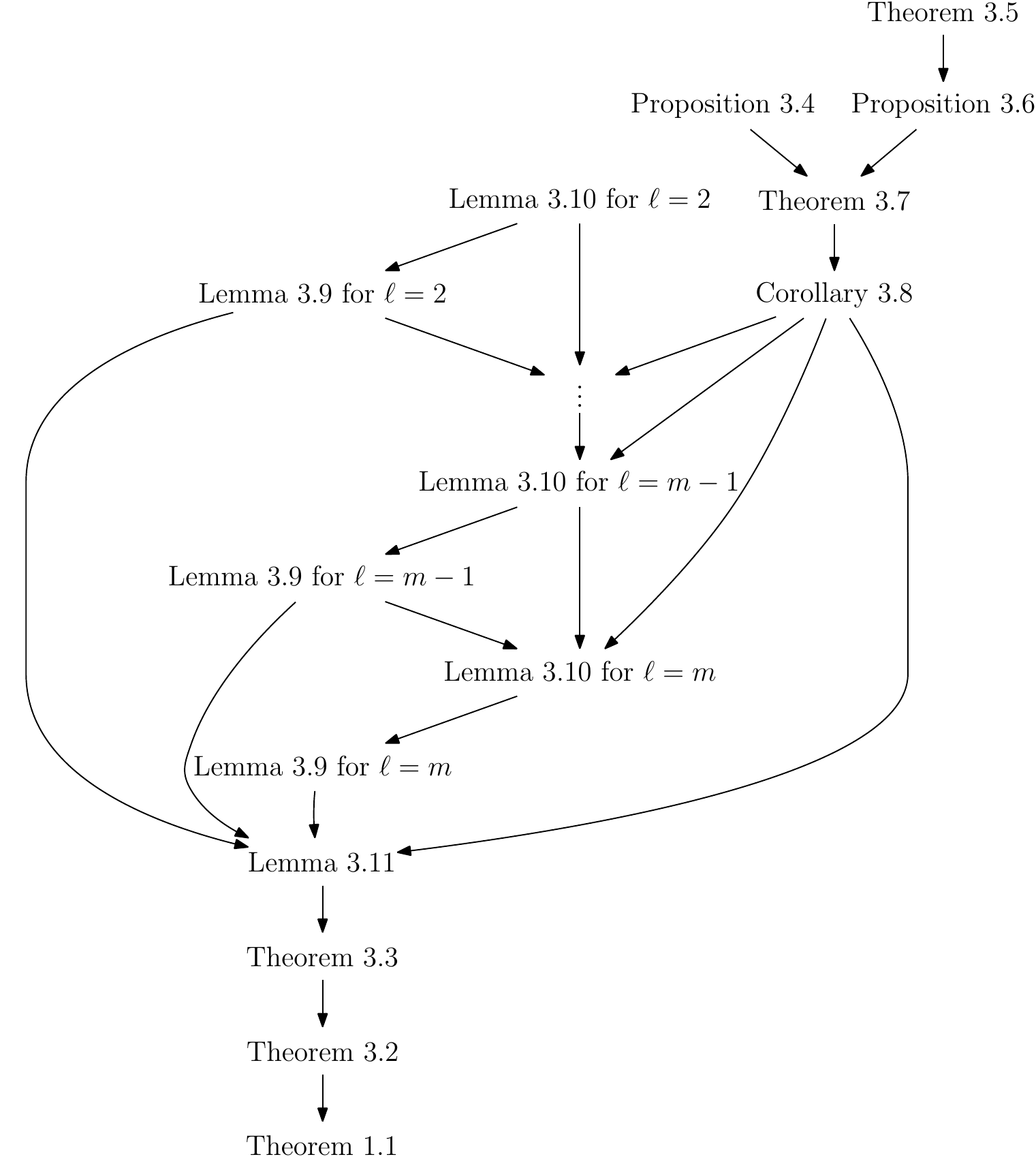}          \caption{Logical dependencies between key results}
                \label{fig1}
	\end{figure}
\end{center}

As was mentioned in the introduction, Theorem~\ref{thm1.1} is proved using a density increment argument. Let $P_1,\dots,P_m\in\Z[y]$ be polynomials with distinct degrees, each having zero constant term. We show that if $A\subset[N]$ has density $\alpha$ and contains no nontrivial progressions of the form $x,x+P_1(y),\dots,x+P_m(y)$, then there exists an arithmetic progression $a+q[N']\subset[N]$ with $N'\asymp_{P_1,\dots,P_m}N^{\Omega_{P_1,\dots,P_m}(1)}$ and $q\ll_{P_1,\dots,P_m}\alpha^{-O_{P_1,\dots,P_m}(1)}$ such that
\[
  \f{|A\cap (a+q[N'])|}{N'}\geq\alpha+\Omega_{P_1,\dots,P_m}(\alpha^{O_{P_1,\dots,P_m}}).
\]
Note that if $A\subset[N]$ contains no nontrivial progressions of the form $x,x+P_1(y),\dots,x+P_m(y)$, then the rescaled set $A':=\{n\in [N']:a+qn\in A\cap (a+q[N'])\}$ contains no nontrivial progressions of the form
\begin{equation}\label{eq3.1}
x,x+\f{P_1(qy)}{q},\dots,x+\f{P_m(qy)}{q},
\end{equation}
and the polynomials $P_i^{(q)}(y):=\f{P_i(qy)}{q}$ for $i=1,\dots,m$ all have integer coefficients and zero constant term.

To continue the density increment argument, we must prove that $A'$ also has increased density on a long arithmetic progression with small common difference. To ensure that our density increment iteration terminates, we want the size of the density increment for $A'$ to depend only on the original polynomials $P_1,\dots,P_m$, and not on $q$. For this reason, we make the following useful definition.
\begin{definition}\label{def3.1}
A polynomial $P=a_dy^d+\dots+a_1y$ has \textit{$(C,q)$-coefficients} if $|a_i|\leq C|a_d|$ for all $i=1,\dots,d-1$ and  $a_d=a_d'q^{d-1}$ with $0<|a_d'|\leq C$.
\end{definition}
Note that any polynomial with $(C,q)$-coefficients has zero constant term by definition, and that any polynomial with zero constant term trivially has $(C,1)$-coefficients for some $C>0$. The usefulness of this definition comes from the fact that if $P_1,\dots,P_m$ all have $(C,r)$-coefficients, then $P_1^{(q)},\dots,P_m^{(q)}$ all have $(C,qr)$-coefficients.

Now we can state our density increment result.
\begin{theorem}\label{thm3.2}
Let $N>0$, $q\in\N$, and $P_1,\dots,P_m\in\Z[y]$ be polynomials with $(C,q)$-coefficients such that $\deg P_1<\dots<\deg P_{m}$. If $A\subset[N]$ has density $\alpha:=|A|/N$ and contains no nontrivial progressions of the form $x,x+P_1(y),\dots,x+P_m(y)$, then there exist positive integers $q'$ and $N'$ satisfying $q'\ll_{C,\deg{P_m}}\alpha^{-O_{\deg{P_m}}(1)}$ and 
\[
N^{\deg{P_1}/\deg{P_m}}\geq N'\gg_{C,\deg{P_m}}N^{\deg{P_1}/\deg{P_m}}(\alpha/q)^{O_{\deg{P_m}}(1)}
\]
such that
\[
\f{|A\cap Q|}{N'}\geq\alpha+\Omega_{C,\deg{P_m}}(\alpha^{O_{\deg{P_m}}(1)})
\]
for some arithmetic progression $Q\subset[N]$ of the form $Q=a+q'q^{b}[N']$ with $b\ll_{\deg{P_m}}1$, provided that $N\gg_{C,\deg{P_m}}(q/\alpha)^{O_{\deg{P_m}}(1)}$.
\end{theorem}
Note that, while the length of the progression on which $A$ has increased density in Theorem~\ref{thm3.2} may depend on $q$, the lower bound $\Omega_{C,\deg{P_m}}(\alpha^{O_{\deg{P_m}}(1)})$ on the density increment is unchanged when $P_1,\dots,P_m$ are replaced by $P_1^{(q)},\dots,P_m^{(q)}$. We are thus guaranteed that our density increment argument will terminate, yielding the bound in Theorem~\ref{thm1.1}.

We prove Theorem~\ref{thm3.2} by studying, for functions $f_0,\dots,f_\ell:\Z\to\C$ supported in the interval $[N]$ and characters $\psi_{\ell+1},\dots,\psi_m:\Z\to S^1$, the following general multilinear average:
\begin{align*}
\Lambda_{P_1,\dots,P_m}^{N,M}&(f_0,\dots,f_\ell;\psi_{\ell+1},\dots,\psi_{m}):= \\
&\f{1}{NM}\sum_{x\in\Z}\sum_{y\in[M]}f_0(x)f_1(x+P_1(y))\cdots f_\ell(x+P_\ell(y))\psi_{\ell+1}(P_{\ell+1}(y))\cdots\psi_{m}(P_m(y)).
\end{align*}
When $m=\ell$ and $f_0=\dots=f_m=f$, we denote $\Lambda_{P_1,\dots,P_m}^{N,M}(f_0,\dots,f_m)$ by $\Lambda_{P_1,\dots,P_m}^{N,M}(f)$. Note that for any $A\subset[N]$ and $M$ sufficiently large, the quantity $\Lambda_{P_1,\dots,P_m}^{N,M}(1_A)$ is $1/NM$ times the number of nontrivial progressions $x,x+P_1(y),\dots,x+P_m(y)$ in $A$. It is necessary for us to study the more general averages $\Lambda_{P_1,\dots,P_m}^{N,M}(f_0,\dots,f_\ell;\psi_{\ell+1},\dots,\psi_{m})$ in order to run some of the inductive arguments within the proof Theorem~\ref{thm1.1}.

Theorem~\ref{thm3.2} is a consequence of the following result, whose proof takes up the bulk of this paper.
\begin{theorem}\label{thm3.3}
Let $N>0$, $q\in\N$, and $P_1,\dots,P_m\in\Z[y]$ be polynomials with $(C,q)$-coefficients such that $\deg{P_1}<\dots<\deg{P_m}$. Set $M:=(N/q^{\deg{P_m}-1})^{1/\deg{P_m}}$. If $f_0,\dots,f_m:\Z\to\C$ are $1$-bounded functions supported on the interval $[N]$ and
\[
\left|\Lambda_{P_1,\dots,P_m}^{N,M}(f_0,\dots,f_m)\right|\geq\delta,
\]
then there exist positive integers $q',b,$ and $N'$ satisfying $q'\ll_{C,\deg{P_m}}\delta^{-O_{\deg{P_m}}(1)}$, $b\ll_{\deg{P_m}}1$, and
\[
M^{\deg{P_1}}\geq N'\gg_{C,\deg{P_m}}M^{\deg{P_1}}(\delta/q)^{O_{\deg{m}}(1)}
\]
such that
\[
\f{1}{N}\sum_{x\in\Z}\left|\E_{y\in[N']}f_1(x+q'q^{b}y)\right|\gg_{C,\deg{P_m}}\delta^{O_{\deg{P_m}}(1)},
\]
provided $N\gg_{C,\deg{P_m}}(q/\delta)^{O_{\deg{P_m}}(1)}$.
\end{theorem}

As was discussed in the introduction, to prove Theorem~\ref{thm3.3} we must show that the average $\Lambda_{P_1,\dots,P_m}^{N,M}(f_0,\dots,f_\ell,\psi_{\ell+1},\dots,\psi_{m})$ is controlled by some $U^s$-norm of the form $\|\cdot\|_{U^s_{c[\delta'L]}([L])}$. We do this by first showing that $\Lambda_{P_1,\dots,P_m}^{N,M}(f_0,\dots,f_\ell,\psi_{\ell+1},\dots,\psi_{m})$ is controlled by an average of a family of Gowers box norms of a special form, and then proving the main concatenation result of Section~\ref{sec5} and repeatedly applying it to averages of such Gowers box norms.

We now describe the special form of the families of Gowers box norms just mentioned. Let $\ell$ and $c$ be nonzero integers with $\ell>0$. For each $j=0,\dots,\ell-1$, we define a sequence of finite sets $I_j=I_j((k_i)_{i\in I_{j-1}})$, which depend on the choice of $k_i\in\N$ for each $i\in I_{j-1}$ when $j\geq 1$, and sets of polynomials $\A_j=\A_j(\ell,c;(k_i)_{i\in I_{j-1}})=\{p_i:i\in I_j\}$, which are indexed by $I_j$, recursively as follows:
\begin{enumerate}
\item $I_0=\{0\}$, $I_1(k_0)=\{0,1\}^{k_0}\sm\{\ul{0}\}$, and
  \[
I_j((k_i)_{i\in I_{j-1}}):=\{0,1\}^{\{(i,r):i\in I_{j-1},r\in[k_i]\}}\sm\{\ul{0}\}
  \]
  for $j=2,\dots,\ell-1$, and
\item $\A_0(\ell,c):=\{c\}$, $\A_1(\ell,c;k_0):=\{(\ell ca_{0,1}^{(1)},\dots,\ell ca_{0,k_0}^{(1)})\cdot\omega:\omega\in I_1(k_0)\}$, and
  \[
\A_j(\ell,c;(k_i)_{i\in I_{j-1}}):=\{((\ell-(j-1))p_ia_{i,r}^{(j)})_{i\in I_{j-1},r\in[k_i]}\cdot\omega:\omega\in I_j\}
  \]
  for $j=2,\dots,\ell-1$.
\end{enumerate}
For example, when $\ell=3$, $c=1$, $k_0=2$, $k_{(0,1)}=k_{(1,1)}=1$, and $k_{(1,0)}=2$, we have $I_0=\{0\}$, $I_1(k_0)=\{(0,1),(1,0),(1,1)\}$,
\[
I_2(k_{(0,1)},k_{(1,0)},k_{(1,1)})=\{0,1\}^{\{((0,1),1),((1,0),1),((1,0),2),((1,1),1)\}}\setminus\{\ul{0}\},
\]
$\A_0(\ell,c)=\{1\}$, $\A_1(\ell,c;k_0)=\{3a_{0,1}^{(1)},3a_{0,2}^{(1)},3a_{0,1}^{(1)}+3a_{0,2}^{(1)}\}$, and
\begin{align*}
  \A_2(\ell,c;(k_{(0,1)},k_{(1,0)},k_{(1,1)}))=\{(6a_{0,1}^{(1)}a^{(2)}_{(1,0),1},6a_{0,2}^{(1)}a^{(2)}_{(0,1),1},6a_{0,2}^{(1)}a^{(2)}_{(0,1),2},6(a_{0,1}^{(1)}+a_{0,2}^{(1)})a_{(1,1),1}^{(2)})\cdot\omega:& \\
  \omega\in\{0,1\}^4\sm\{\ul{0}\}\}&.
\end{align*}

We will show that $\Lambda_{P_1,\dots,P_m}^{N,M}(f_0,\dots,f_\ell;\psi_{\ell+1},\dots,\psi_m)$ is controlled by an average of Gowers box norms of the form $\|\cdot\|_{\square^{|I_{\ell-1}|}_{(Q_i(\ul{a}))_{i\in I_{\ell-1}}}([N])}$, where $Q_i(\ul{a})=p_i(\ul{a})[\delta' M]$ for suitable $0<\delta'<1$. Note that it suffices to prove such a result in the case when $\deg{P_i}=i$ for each $i=1,\dots,m$, for any polynomial progression considered in Theorem~\ref{thm1.1} is a subprogression of such a progression. One may also assume that $\psi_{\ell+1}=\dots=\psi_m=1$, for the general case follows from this special case by the Cauchy--Schwarz inequality. We thus restrict to this situation in the following proposition for ease of notation.

\begin{proposition}\label{prop3.4}
Let $N,M>0$, $q\in\N$, and $P_1,\dots,P_\ell\in\Z[y]$ be polynomials with $(C,q)$-coefficients such that $\deg{P_i}=i$ for $i=1,\dots,\ell$ and $P_\ell$ has leading coefficient $c_\ell$. There exist positive integers $k_i\ll_{\ell}1$ for each $i\in I_{j}$ and $j=0,\dots,\ell-2$ such that the following holds. If $1/C\leq q^{\ell-1}M^{\ell}/N\leq C$, $f_0,\dots,f_\ell:\Z\to\C$ are $1$-bounded functions supported on the interval $[N]$,
\[
\left|\Lambda_{P_1,\dots,P_\ell}^{N,M}(f_0,\dots,f_\ell)\right|\geq\delta,
\]
and $\delta'\ll_{C,\ell}\delta^{O_\ell(1)}$, then we have
\[
\E_{\ul{a}\in A}\|f_\ell\|_{\square^{|I_{\ell-1}|}_{(p(\ul{a})[\delta'M])_{p\in \A_{\ell-1}}}([N])}\gg_{C,\ell}\delta^{O_\ell(1)},
\]
where $I_{\ell-1}:=I_{\ell-1}(\{k_i:i\in I_{\ell-2}\})$ and $\A_{\ell-1}:=\A_{\ell-1}(\ell,c_\ell;(k_i)_{i\in I_{\ell-2}})$ are defined as above and $A:=((-\delta' M,\delta' M)\cap\Z)^{\sum_{j=0}^{\ell-2}\sum_{i\in I_{j}}k_i}$.
\end{proposition}

In Section~\ref{sec5}, we prove that the averages of Gowers box norms appearing in Proposition~\ref{prop3.4} are controlled by some $U^s$-norm with $s\ll_\ell 1$. The most important ingredient of this proof is the following theorem, which is our main concatenation result.
\begin{theorem}\label{thm3.5}
  Let $N,M_1,M_2>0$ with $M_2\leq M_1$ and $M_1M_2\leq N/|c|$ and $b_1,\dots,b_{s}\in[-CN/M_1,CN/M_1]\cap\Z$. If $f:\Z\to\C$ is a $1$-bounded function supported on the interval $[N]$ such that
\begin{equation}\label{eq3.2}
\E_{a\in[M_2]}\|f\|_{\square^s_{((ca+b_i)[M_1])_{i=1}^s}([N])}\geq\delta,
\end{equation}
and $\delta'\ll_{C,s}\delta^{O_s(1)}$, then there exists an $s'\ll_s 1$ such that
\[
\|f\|_{U^{s'}_{c[\delta'M_1M_2]}([N])}\gg_{C,s}\delta^{O_s(1)},
\]
provided that $M_1M_2\gg_{C,s}(\delta\delta')^{-O_s(1)}$.
\end{theorem}
Many averages of Gowers box norms appearing naturally can be controlled by global Gowers uniformity norms through repeated applications of Theorem~\ref{thm3.5}, so we expect that this result could be of independent interest. Another general concatenation result appearing later that may also be of independent interest is Lemma~\ref{lem5.1}.

In the special case when $M_1=M_2=N^{1/2}$, $c=1$, and $b_1,\dots,b_s=0$, after an application of Lemma~\ref{lem2.2}, Theorem~\ref{thm3.5} implies that the average $\E_{a\in[N^{1/2}]}\E_{x\in[N]}\E_{h_1,\dots,h_s\in[N^{1/2}]}\Delta_{ah_1,\dots,ah_s}f(x)$ of ``local Gowers uniformity norms'' (as defined in~\cite{TaoZiegler08}) is controlled by some $U^s$-norm, with polynomial bounds. This thus gives a quantitative version of Proposition~1.26 of~\cite{TaoZiegler16} for arbitrary $s$, though with a worse dependence of $s'$ on $s$.

We take advantage of the special structure of $\A_{\ell-1}$ to prove the following proposition using repeated applications of Theorem~\ref{thm3.5}, showing that averages of Gowers box norms of the form appearing in Proposition~\ref{prop3.4} are controlled by $U^s$-norms.
\begin{proposition}\label{prop3.6}
Let $N,M>0$, $q\in\N$, and $P_1,\dots,P_\ell\in\Z[y]$ be polynomials with $(C,q)$-coefficients such that $\deg{P_i}=i$ for $i=1,\dots,\ell$ and $P_\ell$ has leading coefficient $c_\ell$. There exists an $s\ll_{\ell}1$ such that the following holds. Let $I_{\ell-1}$, $\A_{\ell-1}$, and $A$ be as in Proposition~\ref{prop3.4}. If $1/C\leq q^{\ell-1}M^{\ell}/N\leq C$, $f:\Z\to\C$ is a $1$-bounded function supported on the interval $[N]$,
\[
\E_{\ul{a}\in A}\|f\|_{\square^{|I_{\ell-1}|}_{(p(\ul{a})[\delta' M])_{p\in \A_{\ell-1}}}([N])}\geq\delta, 
\]
and $\delta'\ll_{C,\ell}\delta^{O_\ell(1)}$, then we have
\[
\|f\|_{U^s_{\ell! c_\ell[\delta' M^\ell]}([N])}\gg_{C,\ell}\delta^{O_\ell(1)},
\]
provided that $N\gg_{C,\ell}(q/\delta\delta')^{O_\ell(1)}$.
\end{proposition}
Combining Propositions~\ref{prop3.4} and~\ref{prop3.6}, we thus deduce using the Cauchy--Schwarz inequality that $\Lambda_{P_1,\dots,P_m}^{N,M}(f_0,\dots,f_\ell;\psi_{\ell+1},\dots,\psi_m)$ is controlled by an average of $U^s$-norms.
\begin{theorem}\label{thm3.7}
Let $N,M>0$, $1\leq \ell\leq m$, and $P_1,\dots,P_m\in\Z[y]$ be polynomials such that $P_1,\dots,P_\ell$ have $(C,q)$-coefficients, $\deg{P_1}<\dots<\deg{P_m}$, and $P_\ell$ has leading coefficient $c_\ell$. There exists an $s\ll_{\deg{P_\ell}}1$ such that the following holds. If $1/C\leq q^{\deg{P_\ell}-1}M^{\deg{P_\ell}}/N\leq C$, $f_0,\dots,f_\ell:\Z\to\C$ are $1$-bounded functions supported on the interval $[N]$, $\psi_{\ell+1},\dots,\psi_m:\Z\to S^1$ are characters,
\[
\left|\Lambda_{P_1,\dots,P_m}^{N,M}(f_0,\dots,f_\ell;\psi_{\ell+1},\dots,\psi_m)\right|\geq\delta,
\]
and $\delta'\ll_{C,\deg{P_\ell}}\delta^{O_{\deg{P_\ell}}(1)}$, then we have
\[
\|f_\ell\|_{U^s_{(\deg{P_\ell})! c_\ell[\delta' M^{\deg{P_\ell}}]}([N])}\gg_{C,\deg{P_\ell}}\delta^{O_{\deg{P_\ell}}(1)},
\]
provided that $N\gg_{C,\deg{P_\ell}}(q/\delta\delta')^{O_{\deg{P_\ell}}(1)}$.
\end{theorem}
We will next use the Cauchy--Schwarz inequality to deduce from Theorem~\ref{thm3.7} control of $\Lambda_{P_1,\dots,P_m}^{N,M}(f_0,\dots,f_\ell;\psi_{\ell+1},\dots,\psi_m)$ in terms of an average of $U^s$-norms of dual functions.
\begin{corollary}\label{cor3.8}
Let $N,M>0$, $q\in\N$, $1\leq \ell\leq m$, and $P_1,\dots,P_m\in\Z[y]$ be polynomials such that $P_1,\dots,P_\ell$ have $(C,q)$-coefficients, $\deg{P_1}<\dots<\deg{P_m}$, and $P_\ell$ has leading coefficient $c_\ell$. There exists an $s\ll_{\deg{P_\ell}}1$ such that the following holds. If $1/C\leq q^{\deg{P_\ell}-1}M^{\deg{P_\ell}}/N\leq C$, $f_0,\dots,f_\ell:\Z\to\C$ are $1$-bounded functions supported on the interval $[N]$ and $\psi_{\ell+1},\dots,\psi_m:\Z\to S^1$ are characters,
\[
\left|\Lambda_{P_1,\dots,P_m}^{N,M}(f_0,\dots,f_\ell;\psi_{\ell+1},\dots,\psi_m)\right|\geq\delta
\]
and $\delta'\ll_{\deg{P_\ell},C}\delta^{O_{\deg{P_\ell}}(1)}$, then we have
\[
\|F_\ell\|_{U^s_{(\deg{P_\ell})! c_\ell[\delta' M^{\deg{P_\ell}}]}([O_{\deg{P_\ell}}(CN)])}\gg_{\deg{P_\ell},C}\delta^{O_{\deg{P_\ell}}(1)},
\]
provided that $N\gg_{\deg{P_\ell},C}(q/\delta\delta')^{O_{\deg{P_\ell}}(1)}$, where $F_\ell$ is the dual function
\[
F_\ell(x):=\E_{y\in[M]}f_0(x-P_\ell(y))\cdots f_{\ell-1}(x+P_{\ell-1}(y)-P_\ell(y))\psi_{\ell+1}(P_{\ell+1}(y))\cdots \psi_{m}(P_m(y)).
\]
\end{corollary}

The next step of the proof of Theorem~\ref{thm1.1} is to show our general degree-lowering result. 

\begin{lemma}[Degree lowering for $\ell$]\label{lem3.9}
  Let $N,M>0$, $q\in\N$, $2\leq \ell\leq m$, $P_1,\dots,P_m\in\Z[y]$ be polynomials such that $P_1,\dots,P_\ell$ have $(C,q)$-coefficients, $\deg{P_1}<\dots<\deg{P_m}$, and $P_\ell$ has leading coefficient $c_\ell$ satisfying $1/C\leq |c_{\ell}/c|\leq C$, $f_0,\dots,f_\ell:\Z\to\C$ be $1$-bounded functions supported on the interval $[N]$, and $\psi_{\ell+1},\dots,\psi_m:\Z\to S^1$ be characters. Let $F_\ell$ be as in Corollary~\ref{cor3.8}. If $s\geq 3$, $1/C\leq |c|M^{\deg{P_\ell}}/N\leq C$, $0<\delta'\leq 1$,  and
\[
\|F_\ell\|_{U^s_{c[\delta'M^{\deg{P_\ell}}]}([CN])}\geq\delta,
\]
then
\[
\|F_\ell\|_{U^{s-1}_{c[\delta'M^{\deg{P_\ell}}]}([CN])}\gg_{C,\deg{P_{\ell}},s}(\delta\delta')^{O_{\deg{P_\ell},s}(1)},
\]
provided that $N\gg_{C,\deg{P_{\ell}},s}(q/\delta\delta')^{O_{\deg{P_\ell,s}}(1)}$.
\end{lemma}

Lemma~\ref{lem3.9} is labeled as ``Degree lowering for $\ell$'' because it is proved by induction on $\ell$ using the following lemma.

\begin{lemma}[Major arc lemma for $\ell$]\label{lem3.10}
Let $N,M>0$, $q\in\N$, $2\leq\ell\leq m$, $P_1,\dots,P_{m}\in\Z[y]$ be polynomials such that $P_1,\dots,P_\ell$ have $(C,q)$-coefficients, $\deg{P_1}<\dots<\deg{P_m}$, and $P_i$ has leading coefficient $c_i$ for $i=1,\dots,m$, and $\psi_{\ell},\dots,\psi_m:\Z\to S^1$ be characters with $\psi_i(x)=e(\alpha_i x)$ with $\alpha_i\in\T$ for $i=\ell,\dots,m$. Assume further that $1/C\leq |c|M^{\deg{P_\ell}}/N\leq C$. If there exist $1$-bounded functions $f_0,\dots,f_{\ell-1}:\Z\to\C$ supported on the interval $[N]$ such that
\[
\left|\f{1}{N/c}\sum_{x\in\Z}F_\ell(cx)\psi_\ell(cx)\right|\geq\delta,
\]
where $F_\ell$ is as in Corollary~\ref{cor3.8}, then there exists a positive integer $t\ll_{C,\deg{P_m}}\delta^{-O_{\deg{P_m}}(1)}$ and a $c'\ll_C(|cc_m|)^{O_{\deg{P_m}}(1)}$ such that
\[
\|tc'c_m\alpha_m\|\ll_{C,\deg{P_m}}\f{\delta^{-O_{\deg{P_m}}(1)}}{M^{\deg{P_m}}/c'},
\]
provided that $N\gg_{C,\deg{P_m}}(q/\delta)^{O_{\deg{P_m}}(1)}$.
\end{lemma}
The proof of Lemma~\ref{lem3.10} for each $\ell$ is itself part of the inductive proof of Lemma~\ref{lem3.9}. We first prove Lemma~\ref{lem3.10} in the $\ell=2$ case, then show that Lemma~\ref{lem3.9} for $\ell\geq 2$ follows from Lemma~\ref{lem3.10} for $\ell$, and finally show that Lemma~\ref{lem3.10} for $\ell\geq 3$ follows from Lemmas~\ref{lem3.9} and~\ref{lem3.10} for $\ell-1$. Taken together, this shows that Lemmas~\ref{lem3.9} and~\ref{lem3.10} hold for each $\ell$.

As promised in the introduction, we now discuss why we must assume that $P_1,\dots,P_m$ have distinct degrees in Theorem~\ref{thm1.1}, instead of just requiring them to be linearly independent over $\mathbb{Q}$ as in~\cite{Peluse19}. The proof of the degree-lowering result in~\cite{Peluse19} is made simpler by the fact that there is only ever one ``major arc'' in the finite field setting (the trivial character) and a character of $\F_p$ is either equal to the trivial character or it is not. In contrast, the notion of major arc in the integer setting is more flexible. For the proof of Lemma~\ref{lem3.9}, we need the full strength of the conclusion of Lemma~\ref{lem3.10}: that $\alpha_m$ is within some factor of $M^{-\deg{P_m}}$ of a rational with small denominator. But if we relax the hypotheses of Lemma~\ref{lem3.10} to allow $P_1,\dots,P_m$ to be merely linearly independent, then one can only show that $\alpha_m$ is major arc in a quantitatively weaker sense: that $\alpha_m$ is within some factor of $M^{-\deg{P_1}}$ of a rational with small denominator. This is not strong enough to prove a corresponding degree-lowering result. Of course, if $P_1,\dots,P_m$ are not even linearly independent, the degree lowering phenomenon certainly does not occur even in the finite field setting.

For the final stage of the proof of Theorem~\ref{thm3.2}, we combine Corollary~\ref{cor3.8} with repeated applications of Lemma~\ref{lem2.4} and Lemma~\ref{lem3.9} for each $\ell\leq m$ to show that, when $\Lambda^{N,M}_{P_1,\dots,P_m}(f_0,\dots,f_m)$ is large, averages of related multilinear averages with successive $f_i$'s replaced by characters are also large. This is captured in the following lemma.

\begin{lemma}\label{lem3.11}
Let $N,M>0$, $q\in\N$, $2\leq\ell\leq m$, $P_1,\dots,P_m\in\Z[y]$ be polynomials such that $P_1,\dots,P_\ell$ have $(C,q)$-coefficients, $\deg{P_1}<\dots<\deg{P_m}$, and $P_\ell$ has leading coefficient $c_\ell$, $f_0,\dots,f_\ell:\Z\to\C$ be $1$-bounded functions supported on the interval $[N]$, and $\psi_{\ell+1},\dots,\psi_{m}:\Z\to S^1$ be characters. If $1/C\leq q^{\deg{P_\ell}-1}M^{\deg{P_\ell}}/N\leq C$ and
\[
\left|\Lambda_{P_1,\dots,P_m}^{N,M}(f_0,\dots,f_\ell;\psi_{\ell+1},\dots,\psi_m)\right|\geq\delta,
\]
then
\[
\E_{\substack{u,h=0,\dots,|c'|-1 \\ 0\leq w<(N/|c'|)/C'N'}}\left|\Lambda_{P_1^h,\dots,P_m^h}^{C'N',M'}(f_0^{u,h,w},\dots,f_{\ell-1}^{u,h,w};\psi_{\ell,u},\psi_{\ell+1},\dots,\psi_m)\right|\gg_{C,\deg{P_\ell}}\delta^{O_{\deg{P_\ell}}(1)}
\]
for some characters $\psi_{\ell,u}:\Z\to S^1$, where $C'\asymp_{\deg{P_\ell}}C$, $c':=(\deg{P_\ell})!c_\ell$, $M':=M/|c'|$, $N':=(M')^{\deg{P_{\ell-1}}}(q|c'|)^{\deg{P_{\ell-1}}-1}$,
\[
P_i^h(z):=\begin{cases}
\f{P_i(c'z+h)-P_i(h)}{c'} & i=1,\dots,\ell-1 \\
P_i(c'z+h)-P_i(h) & i=\ell,\dots,m
\end{cases},
\]
and
\[
f_i^{u,h,w}(x):=\begin{cases}
T_{c'C'N'w}T_{-P_\ell(h)}T_{-u}(f_0\psi_{\ell,u})(c'x)\cdot 1_{[C'N']}(x) & i=0 \\
T_{c'C'N'w}T_{P_i(h)-P_\ell(h)}T_{-u}f_i(c'x)\cdot 1_{[C'N']}(x) & i=1,\dots,\ell-1
\end{cases},
\]
provided that $N\gg_{C,\deg{P_\ell}}(q/\delta)^{O_{\deg{P_\ell}}(1)}$.
\end{lemma}

Note that if $P_1,\dots,P_{\ell-1}\in\Z[y]$ have $(C,q)$-coefficients, then $P_1^h,\dots,P_{\ell-1}^h\in\Z[y]$, as defined in Lemma~\ref{lem3.11}, have $(O_{\deg{P_\ell}}(C),c'q)$-coefficients for each $h\in[c']$. To prove Theorem~\ref{thm3.3}, we repeatedly apply Lemma~\ref{lem3.11} and van der Corput's inequality to deduce that if $|\Lambda^{N,M}_{P_1,\dots,P_m}(f_0,\dots,f_m)|\geq\delta$, then an average of multilinear averages of the form $\Lambda_{Q_1,\dots,Q_m}^{N',M'}(g_0,g_1;\psi_2,\dots,\psi_m)$ is large as well, where $g_1$ equals various shifts and scalings of $f_1$ and $\deg{Q_i}=\deg{P_i}-(\deg{P_1}-1)$. It is not hard to show that, usually, the phases $\psi_2,\dots,\psi_m$ must all be major arc, so that after passing to sufficiently short subprogressions modulo an integer of the form $q'q^{b}$ for some $q'\ll_{C,\deg{P_m}}\delta^{-O_{\deg{P_m}}(1)}$ and $b\ll_{\deg{P_1}}1$ and unraveling the definition of $g_1$, we are left with an average of the form appearing in Theorem~\ref{thm3.3}.

\section{Control by an average of Gowers box norms}\label{sec4}

As in previous work on the polynomial Szemer\'edi theorem, we will frequently use van der Corput's inequality, which we now recall. See, for example,~\cite{Montgomery94}.
\begin{lemma}[van der Corput's inequality]\label{lem4.1}
Let $M>H>0$ and $g:\Z\to\C$. We have
\[
\left|\E_{y\in[M]}g(y)\right|^2\leq\f{M+H}{M}\sum_{h\in\Z}\mu_{H}(h)\left[\f{1}{M}\sum_{y\in[M]\cap([M]-h)}\overline{g(y+h)}g(y)\right].
\]
\end{lemma}

As was mentioned in Section~\ref{sec3}, we will use repeated applications of the Cauchy--Schwarz and van der Corput inequalities to control $\Lambda_{P_1,\dots,P_m}^{N,M}$ by an average of Gowers box norms of the form appearing in Proposition~\ref{prop3.4}. To do this, we follow Bergelson and Leibman's PET induction scheme~\cite{BergelsonLeibman96}. Tao and Ziegler~\cite{TaoZiegler08,TaoZiegler18} have also used PET induction to prove that counts of polynomial progressions are controlled by averages of Gowers box norms in their work on polynomial progressions in the primes. Our argument differs in that we care about the precise structure of the average of Gowers box norms so that we can apply Theorem~\ref{thm3.5}. Thus, we will have to make more careful choices at certain points of the PET induction argument, and also keep track of more information.

We first record, for the sake of convenience, the most common way in which the Cauchy--Schwarz and van der Corput inequalities are combined in this section. Like Lemmas~\ref{lem4.4},~\ref{lem4.5}, and~\ref{lem4.6} to follow, the statement of Lemma~\ref{lem4.2} is long because of the amount of information we will want to keep track of, but its proof is short.
\begin{lemma}\label{lem4.2}
Let $N,M>0$, $I$ and $A\subset\Z^n$ be finite sets, $i_0\in I$, $\mu:\Z^n\to[0,\infty)$ be supported on $A$ with $\|\mu\|_{\ell^1}\leq 1$, $Q_i\in\Z[a_1,\dots,a_n][y]$ for each $i\in I$, and $f_{\ul{a}},f_i:\Z\to\C$ be $1$-bounded functions supported on the interval $[N]$ for each $\ul{a}\in A$ and $i\in I$. Assume that
\begin{equation}\label{eq4.1}
\min_{i\in I}\max_{\ul{a}\in A}\max_{y\in[M]}|Q_i(\ul{a},y)|\leq CN.
\end{equation}
If
\begin{equation}\label{eq4.2}
\E_{\ul{a}\in A}^\mu\left|\f{1}{N}\sum_{x\in\Z}\E_{y\in[M]}f_{\ul{a}}(x)\prod_{i\in I}f_i(x+Q_i(\ul{a},y))\right|^2\geq\gamma,
\end{equation}
then for all $\gamma'\ll_C\gamma$, we have
\[
\E_{\ul{a}'\in A'}^{\mu'}\f{1}{N}\sum_{x\in\Z}\E_{y\in[M]}f_{i_0}(x)\prod_{i'\in I'}g_{i'}(x+Q_{i'}'(\ul{a}',y))\gg \gamma,
\]
where
\begin{enumerate}
\item $I'=(I\times\{0,1\})\sm\{(i_0,0)\}$,
\item $A'=A\times((-\gamma' M,\gamma'M)\cap\Z)$,
\item $\mu'(\ul{a}')=\mu(a_1,\dots,a_n)\mu_{\gamma'M}(a_{n+1})$,
\item for each $i'=(i,\epsilon)\in I'$, we have
\[
Q_{i'}'(\ul{a}',y)=
Q_i(a_1,\dots,a_n,y+\epsilon a_{n+1})-Q_{i_0}(a_1,\dots,a_n,y),
\]
\item and for each $i'=(i,\epsilon)\in I'$, we have
\[
g_{i'}=\begin{cases}
f_{i} & \epsilon=0 \\
\overline{f_{i}} & \epsilon=1
\end{cases}.
\]
\end{enumerate}
\end{lemma}
\begin{proof}
For each $\ul{a}\in A$, we first apply the Cauchy--Schwarz inequality in the $x$ variable and use that $f_{\ul{a}}$ is $1$-bounded and supported on $[N]$ to bound the left-hand side of~\eqref{eq4.2} by
\[
\E_{\ul{a}\in A}^\mu\f{1}{N}\sum_{x\in\Z}\left|\E_{y\in[M]}\prod_{i\in I}f_i(x+Q_i(\ul{a},y))\right|^2.
\]
Applying van der Corput's inequality with $g_{x,\ul{a}}(y):=\prod_{i\in I}f_i(x+Q_{i}(\ul{a},y))$ and $H=\gamma' M$ for $0<\gamma'<1$ bounds the above by
\[
\ll\E_{\ul{a}\in A}^\mu\f{1}{N}\sum_{x\in\Z}\left[\sum_{a_{n+1}\in\Z}\mu_{\gamma' M}(a_{n+1})\f{1}{M}\sum_{y\in[M]\cap([M]-a_{n+1})}\overline{g_{x,\ul{a}}(y+a_{n+1})}g_{x,\ul{a}}(y)\right],
\]
where we have used the fact that $M+H=(1+\gamma')M\ll M$.

Now, note that $g_{x,\ul{a}}$ is $1$-bounded because the $f_i$'s are $1$-bounded and, for each $\ul{a}\in A$, $g_{x,\ul{a}}$ is identically zero for all $x\in\Z$ outside of a set of size $\ll CN$ by the assumption~\eqref{eq4.1} since each $f_i$ is supported on the interval $[N]$. Thus, recalling that $\mu_{\gamma' M}$ is supported on $(-\gamma' M,\gamma'M)$ and $\|\mu_{\gamma' M}\|_{\ell^1}\leq 1$, for each $a_{n+1}\in(-\gamma' M,\gamma' M)\cap\Z$ we may extend the sum over $y\in[M]\cap([M]-a_{n+1})$ to a sum over all of $[M]$ at the cost of an error of $O(C\gamma')$. Thus, as long as $\gamma'\ll C\gamma$, we have
\[
\E_{\ul{a}\in A'}^{\mu'}\f{1}{N}\sum_{x\in\Z}\E_{y\in[M]}\prod_{i\in I}\overline{f_i(x+Q_i(a_1,\dots,a_n,y+a_{n+1}))}f_i(x+Q_i(a_1,\dots,a_n,y))\gg\gamma.
\]
To conclude, we make the change of variables $x\mapsto x-Q_{i_0}(\ul{a},y)$.
\end{proof}

To describe the PET induction scheme, we need the notion of a weight vector. This is the $1$-dimensional case of the weight matrix of Bergelson and Leibman~\cite{BergelsonLeibman96}, who also consider more general multidimensional polynomial configurations.
\begin{definition}\label{def4.3}
Let $n\in\N$, $I$ be a finite set, and $Q_i\in\Z[a_1,\dots,a_n][y]$ for each $i\in I$. Set $\mathcal{Q}:=(Q_i)_{i\in I}$, and let $L(Q_i)$ denote the leading coefficient of $Q_i$ for each $i\in I$. The \textit{weight vector} of $\mathcal{Q}$ is defined to be
\[
V(\mathcal{Q}):=(\#\{L(Q_i):\deg{Q_i}=j,i\in I\})_{j=1}^\infty.
\]
We also define the \textit{degree} of $\Q$ to be $\max_{i\in I}\deg{Q_i}$.
\end{definition}

Clearly, the weight vector of any finite set of polynomials has only finitely many nonzero entries. One can define an ordering $\prec$ on the set of weight vectors by saying that $V(\mathcal{Q})\prec V(\mathcal{Q}')$ if there exists a $d\in\N$ such that $\#\{L(Q):\deg{Q}=d,Q\in\mathcal{Q}\}<\#\{L(Q'):\deg{Q'}=d,Q'\in\mathcal{Q}'\}$ and $\#\{L(Q):\deg{Q}=e,Q\in\mathcal{Q}\}=\#\{L(Q'):\deg{Q'}=e,Q'\in\mathcal{Q}'\}$ for all $e>d$. It is easy to see that $\prec$ is a well-ordering on the set of weight vectors. PET induction is simply an induction on the weight vector of collections of polynomials using the ordering $\prec$, with collections of linear polynomials forming the base case of the induction. This method is based on the fact that one can use the Cauchy--Schwarz and van der Corput inequalities to control an average over the polynomial configuration $(x+Q(y))_{Q\in\mathcal{Q}\cup\{0\}}$ by an average over a polynomial configuration $(x+Q'(y))_{Q'\in\mathcal{Q}'\cup\{0\}}$ with $V(\mathcal{Q}')\prec V(\mathcal{Q})$.

As was mentioned in Section~\ref{sec3}, if one can control $\Lambda_{P_1,\dots,P_\ell}^{N,M}(f_1,\dots,f_\ell)$ by an average of $U^s$-norms, then one can also control $\Lambda_{P_1,\dots,P_m}^{N,M}(f_1,\dots,f_\ell;\psi_{\ell+1},\dots,\psi_{m})$ by an average of $U^{s+1}$-norms for any characters $\psi_{\ell+1},\dots,\psi_m:\Z\to S^1$ by using the Cauchy--Schwarz inequality. The first goal of this section is to control $\Lambda_{P_1,\dots,P_\ell}^{N,M}(f_0,\dots,f_\ell)$ in terms of an average of averages over the linear configuration $(x+p(\ul{a})y)_{p\in\A_{\ell-1}\cup\{0\}}$, with $\A_{\ell-1}$ as in Proposition~\ref{prop3.4}. In order to verify that the linear configuration we get at the end of the PET induction argument has this particular form, it will be necessary to keep track of additional details besides the weight vector. In particular, we will keep track of the set of leading coefficients of polynomials of highest degree $d$ and the coefficients of their degree $d-1$ terms.

We will now state three basic lemmas on controlling averages over general progressions $(x+Q(y))_{Q\in\mathcal{Q}\cup\{0\}}$, which apply in different situations depending on the weight vector of $\mathcal{Q}$. These lemmas have long statements, but each proof is just an application of the Cauchy--Schwarz and van der Corput inequalities followed by a change of variables.

\begin{lemma}\label{lem4.4}
Let $N,M>0$, $I$ and $A\subset\Z^n$ be finite sets, $i_0\in I$, $\mu:\Z^n\to[0,\infty)$ be supported on $A$ with $\|\mu\|_{\ell^1}\leq 1$ and $\|\mu\|_{\ell^2}^2\leq C\f{1}{|A|}$, $Q_i\in\Z[a_1,\dots,a_n][y]$ for each $i\in I$, and $f_{\ul{a}},f_i:\Z\to\C$ be $1$-bounded functions supported on the interval $[N]$ for each $\ul{a}\in A$ and $i\in I$. Set $\Q:=(Q_i)_{i\in I}$ and let $d$ be the degree of $\Q$, $r= V(\Q)_d$, $\mathcal{C}$ denote the set of leading coefficients of degree $d$ polynomials in $\Q$, $c_{i_0}$ be the leading coefficient of $Q_{i_0}$, and $d'$ be the smallest index such that $V(\mathcal{Q})_{d'}\neq 0$. Assume further that
\begin{enumerate}
\item $1\leq d'<d$,
\item there exists an $s\in\N$ such that, for all $c\in\mathcal{C}$, there are $s$ degree $d$ polynomials $Q$ in $\Q$ with leading coefficient $c$, each having the form
\[
c(a_1,\dots,a_n)y^{d}+c'_{Q}(a_1,\dots,a_n)y^{d-1}+\text{lower degree terms},
\]
where the coefficients $c_{Q}'(a_1,\dots,a_n)$ are all distinct,
\item $\deg{Q_{i_0}}=d'$,
\item and
\[
\max_{i\in I}\max_{\ul{a}\in A}\max_{y\in[M]}|Q_i(\ul{a},y)|\leq C'N.
\]
\end{enumerate}
If
\begin{equation}\label{eq4.3}
\left|\E_{\ul{a}\in A}^\mu\f{1}{N}\sum_{x\in\Z}\E_{y\in[M]}f_{\ul{a}}(x)\prod_{i\in I}f_i(x+Q_i(\ul{a},y))\right|\geq\gamma,
\end{equation}
then for all $\gamma'\ll_{C,C'}\gamma^2$, we have
\[
\E_{\ul{a}'\in A'}^{\mu'}\f{1}{N}\sum_{x\in\Z}\E_{y\in[M]}f_{i_0}(x)\prod_{i'\in I'}g_{i'}(x+Q_{i'}'(\ul{a},y))\gg_C\gamma^2,
\]
where
\begin{enumerate}
\item $I'=(I\times\{0,1\})\sm\{(i_0,0)\}$,
\item $A'=A\times((-\gamma' M,\gamma'M)\cap\Z)$,
\item $\mu'(\ul{a}')=\f{1_{A}(a_1,\dots,a_n)}{|A|}\mu_{\gamma'M}(a_{n+1})$,
\item for $i'\in I'$, we have $Q'_{i'}(\ul{a}',y)=Q_i(\ul{a},y+\epsilon a_{n+1})-Q_{i_0}(\ul{a},y)$,
\item the set of leading coefficients of degree $d$ polynomials in $\Q':=(Q'_{i'})_{i'\in I'}$ is $\mathcal{C}$,
\item for all $c\in\mathcal{C}$, there are $2s$ degree $d$ polynomials in $\Q'$ with leading coefficient $c$, and for each $i'=(i,\epsilon)\in I'$ with $\deg{Q_{i}}=d$ and $Q_i$ having leading coefficient $c$, the polynomial $Q'_{i'}$ has the form
\begin{align*}
c(a_1,\dots,a_n)y^d+[c'_{Q_{i}}(a_1,\dots,a_n)+\epsilon d c(a_1,\dots,a_n)a_{n+1}-&1_{d'=d-1}c_{i_0}(a_1,\dots,a_n)]y^{d-1} \\
&+\text{lower degree terms},
\end{align*}
so that the coefficients of the degree $d-1$ terms of these polynomials are still distinct,
\item we have
\[
V(\Q')=(n_1,\dots,n_{d'-1},V(\Q)_{d'}-1,V(\Q)_{d'+1},\dots,V(\Q)_d,0,\dots),
\]
where $n_1+\dots+n_{d'-1}<|I'|=2|I|-1$,
\item and, for $i'=(i,\epsilon)\in I'$, we have
\[
g_{i'}=\begin{cases}
f_i & \epsilon=0 \\
\overline{f_i} & \epsilon=1
\end{cases}.
\]
\end{enumerate}
\end{lemma}
\begin{proof}
We expand the definition of $\E^\mu$ to write the left-hand side of~\eqref{eq4.3} as
\[
\left|\sum_{\ul{a}\in A}\mu(\ul{a})\left[\f{1}{N}\sum_{x\in\Z}\E_{y\in[M]}f_{\ul{a}}(x)\prod_{i\in I}f_i(x+Q_i(\ul{a},y))\right]\right|\geq\gamma,
\]
and apply the Cauchy--Schwarz inequality in the $\ul{a}$ variable to deduce that
\[
\E_{\ul{a}\in A}\left|\f{1}{N}\sum_{x\in\Z}f_{\ul{a}}(x)\E_{y\in[M]}\prod_{i\in I}f_i(x+Q_i(\ul{a},y))\right|^2\gg_C\gamma^2,
\]
using the assumption $\|\mu\|_{\ell^2}^2\leq C\f{1}{|A|}$.

We now apply Lemma~\ref{lem4.2} to conclude. Indeed, if $Q_i$ has degree $d$ and leading coefficient $c$, then, by the binomial theorem, $Q_i(a_1,\dots,a_n,y+\epsilon a_{n+1})$ equals
\[
c(a_1,\dots,a_n)y^d+[c_{Q_i}'(a_1,\dots,a_n)+\epsilon d c(a_1,\dots,a_n)a_{n+1}]y^{d-1}+\text{lower degree terms}.
\]
In addition, if $Q_{i}$ has degree $>d'$, then $Q_{(i,\epsilon)}$ (as defined in Lemma~\ref{lem4.2}) has the same degree and leading coefficient as $Q_{i}$, if $Q_i$ has degree $d'$ and leading coefficient equal to $c_{i_0}$, then $Q_{(i,\epsilon)}$ has degree $\leq d'-1$, and if $Q_{i}$ has degree $d'$ and leading coefficient $c_{i}\neq c_{i_0}$, then $Q_{(i,\epsilon)}$ also has degree $d'$ and has leading coefficient $c_{i}-c_{i_0}$, thus confirming conclusion~{(7)} of the lemma.
\end{proof}


\begin{lemma}\label{lem4.5}
Let $N,M>0$, $I$ and $A\subset\Z^n$ be finite sets, $i_0\in I$, $\mu:\Z^n\to[0,\infty)$ be supported on $A$ with $\|\mu\|_{\ell^1}\leq 1$ and $\|\mu\|_{\ell^2}^2\leq C\f{1}{|A|}$, $Q_i\in\Z[a_1,\dots,a_n][y]$ for each $i\in I$, and $f_{\ul{a}},f_i:\Z\to\C$ be $1$-bounded functions supported on the interval $[N]$ for each $\ul{a}\in A$ and $i\in I$. Set $\Q:=(Q_i)_{i\in I}$, and let $d$ be the degree of $\Q$ and $r=V(\Q)_d$. Assume further that
\begin{enumerate}
\item $d>1$ and $r=1$,
\item $V(\Q)_{d'}=0$ for all $d'<d$,
\item the polynomials $Q\in\Q$ each have the form
\[
c(a_1,\dots,a_n)y^d+c_{Q}'(a_1,\dots,a_n)y^{d-1}+\text{lower degree terms},
\]
where the coefficients $c_Q'(a_1,\dots,a_n)$ are all distinct,
\item and
\[
\max_{i\in I}\max_{\ul{a}\in A}\max_{y\in[M]}|Q_i(\ul{a},y)|\leq C'N.
\]
\end{enumerate}
If
\[
\left|\E_{\ul{a}\in A}^\mu\f{1}{N}\sum_{x\in\Z}\E_{y\in[M]}f_{\ul{a}}(x)\prod_{i\in I}f_i(x+Q_{i}(\ul{a},y))\right|\geq\gamma,
\]
then for all $\gamma'\ll_{C,C'}\gamma^2$, we have
\[
\E_{\ul{a}'\in A'}^{\mu'}\f{1}{N}\sum_{x\in\Z}\E_{y\in[M]}f_{i_0}(x)\prod_{i'\in I'}g_{i'}(x+Q_{i'}'(\ul{a},y))\gg_C\gamma^2,
\]
where
\begin{enumerate}
\item $I'=(I\times\{0,1\})\sm\{(i_0,0)\}$,
\item $A'=A\times((-\gamma' M,\gamma'M)\cap\Z)$,
\item $\mu'(\ul{a}')=\f{1_A(a_1,\dots,a_n)}{|A|}\mu_{\gamma'M}(a_{n+1})$,
\item for $i'=(i,\epsilon)\in I'$, we have $Q'_{i'}(\ul{a}',y)=Q_i(\ul{a},y+\epsilon a_{n+1})-Q_{i_0}(\ul{a},y)$,
\item the set $\Q':=(Q_{i'}')_{i'\in I'}$ consists of $2|I|-1$ degree $d-1$ polynomials, each with distinct leading coefficient, and the set of such coefficients is 
\[
\{c'_{Q_i}(a_1,\dots,a_n)+\epsilon d c(a_1,\dots,a_n)a_{n+1}-c_{Q_{i_0}}'(a_1,\dots,a_n): (i,\epsilon)\in I'\},
\]
\item we have
\[
V(\Q')=(\overbrace{0,\dots,0}^{d-2},2|I|-1,0,\dots),
\]
\item and for $i'=(i,\epsilon)\in I'$, we have
\[
g_{i'}=\begin{cases}
f_i & \epsilon=0 \\
\overline{f_i} & \epsilon=1
\end{cases}.
\]
\end{enumerate}
\end{lemma}
\begin{proof}
Apply the Cauchy--Schwarz inequality and Lemma~\ref{lem4.2} in exactly the same manner as in the proof of Lemma~\ref{lem4.4}.
\end{proof}


\begin{lemma}\label{lem4.6}
Let $N,M>0$, $I$ and $A\subset\Z^n$ be finite sets, $i_0\in I$, $\mu:\Z^n\to[0,\infty)$ be supported on $A$ with $\|\mu\|_{\ell^1}\leq 1$ and $\|\mu\|_{\ell^2}^2\leq C\f{1}{|A|}$, $Q_i\in\Z[a_1,\dots,a_n][y]$ for each $i\in I$, and $f_{\ul{a}},f_i:\Z\to\C$ be $1$-bounded functions supported on the interval $[N]$ for each $\ul{a}\in A$ and $i\in I$. Set $\Q:=(Q_i)_{i\in I}$ and let $d$ be the degree of $\Q$, $r= V(\Q)_d$, $\mathcal{C}$ denote the set of leading coefficients of degree $d$ polynomials in $\Q$, and $c_{i_0}$ be the leading coefficient of $Q_{i_0}$. Assume further that
\begin{enumerate}
\item $d>1$ and $r>1$,
\item $V(\Q)_{d'}=0$ for all $d'<d$,
\item there exists an $s\in\N$ such that, for all $c\in\mathcal{C}$, there are $s$ degree $d$ polynomials $Q$ in $\Q$ with leading coefficient $c$, each having the form
\[
c(a_1,\dots,a_n)y^{d}+c'_{Q}(a_1,\dots,a_n)y^{d-1}+\text{lower degree terms},
\]
where the coefficients $c_{Q}'(a_1,\dots,a_n)$ are all distinct,
\item and
\[
\max_{i\in I}\max_{\ul{a}\in A}\max_{y\in[M]}|Q_i(\ul{a},y)|\leq C'N.
\]
\end{enumerate}
If
\[
\left|\E_{\ul{a}\in A}^\mu\f{1}{N}\sum_{x\in\Z}\E_{y\in[M]}f_{\ul{a}}(x)\prod_{i\in I}f_i(x+Q_i(\ul{a},y))\right|\geq\gamma,
\]
then for all $\gamma'\ll_{C,C'}\gamma^2$, we have
\[
\E_{\ul{a}'\in A'}^{\mu'}\f{1}{N}\sum_{x\in\Z}\E_{y\in[M]}f_{i_0}(x)\prod_{i'\in I'}g_{i'}(x+Q_{i'}'(\ul{a},y))\gg_C\gamma^2,
\]
where
\begin{enumerate}
\item $I'=(I\times\{0,1\})\sm\{(i_0,0)\}$,
\item $A'=A\times((-\gamma' M,\gamma'M)\cap\Z)$,
\item $\mu'(\ul{a}')=\f{1_A(a_1,\dots,a_n)}{|A|}\mu_{\gamma'M}(a_{n+1})$,
\item for $i'=(i,\epsilon)\in I'$, we have $Q'_{i'}(\ul{a}',y)=Q_i(\ul{a},y+\epsilon a_{n+1})-Q_{i_0}(\ul{a},y)$,
\item the set of leading coefficients of degree $d$ polynomials in $\Q':=(Q'_{i'})_{i'\in I'}$ is $\{c-c_{i_0}:c\in\mathcal{C}\}\sm\{0\}$,
\item for each $c\in\mathcal{C}\sm\{c_{i_0}\}$ there are $2s$ degree $d$ polynomials in $\Q'$ with leading coefficient $c-c_{i_0}$, and for $i'=(i,\epsilon)\in I'$ with $\deg{Q_{i}}=d$ and $Q_i$ having leading coefficient $c$, the polynomial $Q'_{i'}(\ul{a}',y)$ has the form
\begin{align*}
  (c-c_{i_0})(a_1,\dots,a_n)y^d+[c'_{Q_{i}}(a_1,\dots,a_n)+\epsilon d c(a_1,\dots,a_n)&a_{n+1}-c_{Q_{i_0}}'(a_1,\dots,a_n)]y^{d-1}\\
  &+\text{lower degree terms},
\end{align*}
so that the coefficients of the degree $d-1$ terms of these polynomials are still distinct,
\item we have
\[
V(\Q')=(n_1,\dots,n_{d-1},V(\Q)_d-1,0,\dots),
\]
where $n_1+\dots+n_{d'-1}<|I'|=2|I|-1$,
\item and for $i'=(i,\epsilon)\in I'$, we have
\[
g_{i'}=\begin{cases}
f_i & \epsilon=0 \\
\overline{f_i} & \epsilon=1
\end{cases}.
\]
\end{enumerate}
\end{lemma}
\begin{proof}
As with the previous lemma, the proof is the same as that of Lemma~\ref{lem4.4}.
\end{proof}

 The next two lemmas are proved by many applications of the previous three lemmas, with the choice of $i_0$ in many uses of these lemmas being particularly important. Recall that the set $\A_{\ell-1}$ was defined recursively. Correspondingly, the proof that the average $\Lambda_{P_1,\dots,P_\ell}^{N,M}(f_1,\dots,f_\ell)$ is controlled by an average of averages over the linear progression $(x+p(\ul{a})y)_{p\in\A_{\ell-1}\cup\{0\}}$  proceeds iteratively. Lemma~\ref{lem4.7} produces the initial situation that we will apply Lemma~\ref{lem4.8} to repeatedly.

\begin{lemma}\label{lem4.7}
Let $N,M>0$ and $P_1,\dots,P_\ell\in\Z[y]$ be polynomials with $(C,q)$-coefficients such that $\deg{P_i}=i$ for $i=1,\dots,\ell$ and $P_\ell$ has leading coefficient $c_\ell$. If $1/C\leq q^{\ell-1}M^\ell/N\leq C$, $f_0,\dots,f_\ell:\Z\to\C$ are $1$-bounded functions supported on the interval $[N]$,
\[
\left|\Lambda_{P_1,\dots,P_\ell}^{N,M}(f_0,\dots,f_\ell)\right|\geq\gamma,
\]
and $\gamma'\ll_{C,\ell}\gamma^{O_\ell(1)}$, then we have
\[
\E_{\ul{a}\in A}^{\mu}\f{1}{N}\sum_{x\in\Z}\E_{y\in[M]}f_\ell(x)\prod_{i\in I}f_{i}'(x+Q_{i}(\ul{a},y))\gg_{C,\ell}\gamma^{O_\ell(1)},
\]
where
\begin{enumerate}
\item $I=\{0,1\}^t\sm\{\ul{0}\}$ for some $t\ll_\ell 1$,
\item $A=((-\gamma' M,\gamma' M)\cap\Z)^t$,
\item $\mu(a_1,\dots,a_t)=\f{1_A(a_1,\dots,a_t)}{(2\lfloor \gamma' M\rfloor+1)^{t-1}}\mu_{\gamma' M}(a_t)$,
\item the collection $\Q:=(Q_i)_{i\in I}$ consists only of polynomials of degree $\ell-1$, each of which has distinct leading coefficient, and the set of such leading coefficients is
\[
\{(\ell c_\ell a_1,\dots,\ell c_\ell a_t)\cdot\omega : \omega\in I\},
\]
\item we have
\[
\max_{i\in I}\max_{\ul{a}\in A}\max_{y\in[M]}|Q_i|(\ul{a},y)\ll_{C,\ell}N,
\]
\item and $f_i'$ equals either $f_\ell$ or $\overline{f_\ell}$ for all $i\in I$.
\end{enumerate}
\end{lemma}
In this lemma and those to follow, if $Q=a_dy^d+\dots+a_0\in\C[y]$ is any polynomial, then $|Q|$ denotes the polynomial $|a_d|y^d+\dots+|a_0|$.
\begin{proof}
The proof proceeds by applying Lemma~\ref{lem4.4} some number of times depending on $\ell$, and then Lemma~\ref{lem4.5} once. Suppose that $P_\ell$ has degree $\ell-1$ coefficient $c'_{\ell}$ and $P_{\ell-1}$ has leading coefficient $c_{\ell-1}$. Set $J_0=[\ell]$, $A_0=\{0\}$, $\mu_0=1_{\{0\}}$, $\Q_0=\{P_1,\dots,P_\ell\}$, $\mathcal{C}'_0=\{c'_\ell\}$, $i_{0,0}=1$, and $g_{j,0}=f_j$ for $j=1,\dots,\ell$. We apply Lemma~\ref{lem4.4} repeatedly to produce a sequence of $t-1\ll_\ell 1$ finite sets $J_k$ and $A_k$, measures $\mu_k$, collections of polynomials $\Q_k\subset\Z[a_1,\dots,a_k][y]$, sets $\mathcal{C}_k'\subset\Z[a_1,\dots,a_k]$ of coefficients of the degree $\ell-1$ term of degree $\ell$ polynomials in $\Q_k$, elements $i_{0,k}\in J_k$, and $1$-bounded functions $g_{j,k}$ for each $j\in J_k$ satisfying
\begin{enumerate}
\item $J_k=((J_{k-1}\sm\{j\in J_{k-1}:\deg{Q_j}=0\})\times\{0,1\})\sm\{(i_{0,k-1},0)\}$ for $k=1,\dots,t-1$,
\item $A_k=((-\gamma'M,\gamma'M)\cap\Z)^k$ for $k=1,\dots,t-1$,
\item $\mu_k(a_1,\dots,a_k)=\f{1_{A_{k-1}}(a_1,\dots,a_{k-1})}{(2\lfloor\gamma' M\rfloor+1)^{k-1}}\mu_{\gamma'M}(a_k)$ for $k=1,\dots,t-1$,
\item $\Q_k=(Q_j)_{j\in J_k}$ for $k=1,\dots,t-1$, where, for $j=(j',\epsilon)\in J_k$, we have
\[
Q_j(a_1,\dots,a_k,y)=Q_{j'}(a_1,\dots,a_{k-1},y+\epsilon a_k)-Q_{i_{0,k-1}}(a_1,\dots,a_{k-1},y),
\]
\item $\mathcal{C}_k'=\{c'_\ell-\epsilon(k)c_{\ell-1}+\ell c_\ell(a_1,\dots,a_k)\cdot\omega:\omega\in\{0,1\}^k\}$ for $k=1,\dots,t-1$, where $\epsilon(k)=1$ if $1\ll_\ell k\leq t-1$ and $\epsilon(k)=0$ otherwise,
\item for $j=(j',\epsilon)\in J_k$, we have $g_{j,k}$ equal to either $g_{j',k-1}$ or $\overline{g_{j',k-1}}$,
\item $i_{0,k}\in J_k$ is the index of any nonconstant (in $y$) polynomial of smallest degree in $\Q_k$ for $k=1,\dots,t-1$, and $i_{0,t-1}\in J_{t-1}$ is the index $(\ell,\ul{0})$,
\item and
\[
V(\Q_{t-1})=(\overbrace{0,\dots,0}^{\ell-1},1,0,\dots),
\]
\end{enumerate}
such that
\[
\E_{\ul{a}\in A_k}^{\mu_k}\f{1}{N}\sum_{x\in\Z}\E_{y\in[M]}f_{\ul{a},k}(x)\prod_{\substack{j\in J_k \\ \deg{Q_j}\neq 0}}g_{j,k}(x+Q_j(a_1,\dots,a_k,y))\gg_{k}\gamma^{O_k(1)},
\]
where
\[
f_{\ul{a},k}(x)=g_{i_{0,k-1},k-1}(x)\prod_{\substack{j\in J_k \\ \deg{Q_j}= 0}}g_{j,k}(x+Q_j(a_1,\dots,a_k,y))
\]
for all $k=1,\dots,t-1$, provided that $\gamma'\ll_{C,\ell}\gamma^{O_\ell(1)}$. Indeed, we have that $\|\mu_k\|_{\ell^2}^2\leq\f{1}{|A_{k-1}|\gamma'M}\leq\f{3}{|A_k|}$ for each $k=1,\dots,t-1$, and to check that the condition
\begin{equation}\label{eq4.4}
\max_{j\in J_k}\max_{\ul{a}\in A_k}\max_{y\in[M]}|Q_j(\ul{a},y)|\ll_{C,\ell}N
\end{equation}
holds for each application of Lemma~\ref{lem4.4}, note that
\[
\max_{i=1,\dots,\ell}\sup_{y\in[-cM,cM]}|P_i(y)|\leq \ell c^\ell C^3N
\]
for any $c\in\N$ by the assumptions that $P_1,\dots,P_\ell$ have $(C,q)$-coefficients, $\deg{P_i}=i$ for $i=1,\dots,\ell$, and $q^{\ell-1}M^\ell\leq CN$, which implies that~\eqref{eq4.4} holds by the recursive definition of the $Q_j$'s and the triangle inequality.

Note that $\Q_{t-1}$ consists only of constant polynomials (in $y$) and polynomials of degree $\ell$ (in $y$), we have $J_{t-1}\sm\{j\in J_{t-1}:\deg{Q_j}=0\}=\{\ell\}\times\{0,1\}^{t-1}$, $i_{0,t-1}$ is the index of the degree $\ell$ polynomial in $\Q_{t-1}$ whose degree $\ell-1$ term has coefficient $c'_\ell-c_{\ell-1}$, and $g_{j,t-1}$ equals either $f_\ell$ or $\overline{f_\ell}$ for every $j\in J_k$ such that $\deg{Q_j}=\ell$. We may thus apply Lemma~\ref{lem4.5} with $J_{t-1}\sm\{j\in J_{t-1}:\deg{Q_j}=0\}$, $A_{t-1}$, $\mu_{t-1}$, $i_{0,t-1}$, $f_{\ul{a},t-1}$, and $f_j=g_{j,t-1}$ for each $j\in J_{t-1}\sm\{j\in J_{t-1}:\deg{Q_j}=0\}$, again assuming that $\gamma'\ll_{C,\ell}\gamma^{O_\ell(1)}$. The conclusion of the lemma then follows after relabeling indices in $[\ell]\times\{0,1\}^{t}\sm\{(\ell,\ul{0})\}$ by the corresponding elements of $\{0,1\}^t\sm\{\ul{0}\}$. The bound on $|Q_i|(\ul{a},y)$ follows in the same manner as~\eqref{eq4.4} using the triangle inequality.
\end{proof}

Lemma~\ref{lem4.7} may be used, for example, to control the progression $x,x+y,x+y^3$ in terms of averages over the progression $x,x+3a_1y^2+3a_1^2y,x+3a_2y^2+3a_2^2y,x+3(a_1+a_2)y^2+3(a_1^2+a_2^2+2a_1a_2)y$, where we have absorbed the constant (in $y$) terms into the definitions of the $f_{\ul{a}}$'s for the sake of simplicity.

\begin{lemma}\label{lem4.8}
Let $N,M>0$, $I$ and $A\subset([-M,M]\cap\Z)^n$ be finite sets, $\mu:\Z^n\to[0,\infty)$ be supported on $A$ with $\|\mu\|_{\ell^1}\leq 1$ and $\|\mu\|_{\ell^2}^2\leq C\f{1}{|A|}$, $Q_i\in\Z[a_1,\dots,a_n][y]$ be degree $d\geq 2$ polynomials for each $i\in I$, $\mathcal{C}$ be the set of leading coefficients of polynomials in $\Q:=(Q_i)_{i\in I}$ with $m:=|\mathcal{C}|$, and $f,f_i:\Z\to\C$ be $1$-bounded functions supported on the interval $[N]$ for each $i\in I$. Assume further that
\begin{enumerate}
\item $I$ and $\mathcal{C}$ have the form $I=\{0,1\}^J\sm\{\ul{0}\}$ and
\begin{equation}\label{eq4.5}
\mathcal{C}=\{(c_j^{0}(a_1,\dots,a_n))_{j\in J}\cdot\omega:\omega\in I\}
\end{equation}
for some finite set $J$ and polynomials $c_j^0\in\Z[a_1,\dots,a_n]$,
\item $m=|I|$, so that the leading coefficients of elements of $\Q$ are all distinct,
\item we have
\[
  \max_{i\in I}\max_{\ul{a}\in A}\max_{y\in[M]}|Q_i|(\ul{a},y)\leq CN,
\]
\item and $f_i$ equals either $f$ or $\overline{f}$ for each $i\in I$.
\end{enumerate}
If
\[
\left|\E_{\ul{a}\in A}^\mu\f{1}{N}\sum_{x\in\Z}\E_{y\in[M]}f(x)\prod_{i\in I}f_i(x+Q_i(\ul{a},y))\right|\geq\gamma
\]
and $\gamma'\ll_{C,d,m}\gamma^{O_{d,m}(1)}$, then we have
\[
\E_{\ul{a}\in A'}^{\mu'}\f{1}{N}\sum_{x\in\Z}\E_{y\in[M]}f(x)\prod_{i'\in I'}f_{i'}'(x+Q_i'(\ul{a}',y))\gg_{C,d,m}\gamma^{O_{d,m}(1)},
\]
where
\begin{enumerate}
\item $I'=\{0,1\}^{\{(i,r):i\in I,r\in[k_i]\}}\sm\{\ul{0}\}$ for some $k_i\ll_{d,m} 1$ for each $i\in I$,
\item $A'=A\times ((-\gamma' M,\gamma'M)\cap\Z)^{\sum_{i\in I}k_i}$,
\item $\mu'(\ul{a},(a_{i,r})_{i\in I,r\in[k_i]})=\f{1_{A'}(\ul{a},(a_{i,r})_{i\in I,r\in[k_i]})}{|A|(2\lfloor\gamma' M\rfloor+1)^{\sum_{i\in I}k_i-1}}\mu_{\gamma'M}(a_{j,k_{j}})$ for some $j\in I$,
\item $\Q':=(Q_{i'}')_{i'\in I'}$ consists only of polynomials of degree $d-1$, each of which has distinct leading coefficient, and the set of such leading coefficients is
\[
\{(dc_i(a_1,\dots,a_n)a_{i,r})_{i\in I,r\in[k_i]}\cdot\omega:\omega\in I'\},
\]
\item we have
\[
\max_{i'\in I'}\max_{\ul{a}'\in A'}\max_{y\in[M]}|Q_{i'}'|(\ul{a}',y)\ll_{C,d,m} N,
\]
\item and $f_{i'}'$ equals either $f$ or $\overline{f}$ for every $i'\in I'$.
\end{enumerate}
\end{lemma}
\begin{proof}
The proof proceeds by applying Lemma~\ref{lem4.5} once after repeating the following $m-1$ times: apply Lemma~\ref{lem4.6} once, and then Lemma~\ref{lem4.4} as many times as necessary with careful choices of distinguished index $i_0$ to produce a bound in terms of an average over a polynomial progression involving only polynomials of degree $d$. Each repetition of this procedure reduces the number of distinct leading coefficients of polynomials of degree $d$ by one.

We first enumerate the elements $c_1,\dots,c_m$ of $\mathcal{C}$ by picking any ordering such that if $k\leq k'$, then $c_k(\ul{a})=(c_j^0(\ul{a}))_{j\in J}\cdot\omega$ and $c_{k'}(\ul{a})=(c_j^0(\ul{a}))_{j\in J}\cdot\omega'$ with $|\omega|\leq |\omega'|$. This means that $c_m(\ul{a})=\sum_{j\in J}c_j^0(\ul{a})$. Enumerate the elements $Q_1,\dots,Q_m$ of $\Q$ similarly, so that $Q_i$ has leading coefficient $c_i(\ul{a})$, and let $c_i'(\ul{a})$ denote the coefficient of the degree $d-1$ term of $Q_i$ for each $i=1,\dots,m$. Set $c_0(\ul{a}):=0$.

Let $I_0=[m]$, $A_0=A$, $\mu_0=\mu$, $\Q_0=\Q$, $\mathcal{C}_0=\mathcal{C}$, $\mathcal{C}_0^{(k)}=\{c_k'\}$ for each $k=1,\dots,m$, and $i_{0,0}=1$. We will show that applying Lemma~\ref{lem4.6} and then Lemma~\ref{lem4.4} repeatedly produces a sequence of $m-1$ finite sets $I_j$ and $A_j$, measures $\mu_j$ supported on $A_j$, sets $\Q_j=(Q_{i,j})_{i\in I_j}$ of degree $d$ polynomials with set of leading coefficients $\mathcal{C}_j$, sets $\mathcal{C}_j^{(k)}$ of the coefficients of the degree $d-1$ terms of polynomials in $\Q_j$ with leading coefficient $c_k-c_j$ for each $k=j+1,\dots,m$, and elements $i_{0,j}\in I_j$ satisfying
\begin{enumerate}
\item $I_j=\{j+1,\dots,m\}\times\{0,1\}^{\{(s,r):0\leq s\leq j,r\in[k_{s,j}]\}}$ for some $k_{s,j}\ll_{d,m}1$ for each $0\leq s\leq j$ and $j=1,\dots,m-1$, where $k_{0,j}=1$,
\item $A_j=A_{j-1}\times((-\gamma' M,\gamma'M)\cap\Z)^{k_{j,j}+1}$ for $j=1,\dots,m-1$,
\item $\mu_j(\ul{a},(a_{s,r})_{0\leq s\leq j,r\in[k_{s,j}]})=\f{1_{A_j}(\ul{a},(a_{s,r})_{0\leq s\leq j,r\in[k_{s,j}]})}{|A_{j-1}|(2\lfloor \gamma' M\rfloor+1)^{k_{j,j}}}\mu_{\gamma'M}(a_{j,k_{j,j}})$ for $j=1,\dots,m-1$,
\item $\mathcal{C}_j=\{c_{j+1}-c_j,\dots,c_m-c_j\}$ for $j=1,\dots,m-1$ and, for $i=(s,\omega)\in I_j$, the polynomial $Q_{i,j}\in\Q_j$ has leading coefficient $c_s-c_j$,
\item $\mathcal{C}^{(k)}_j=\{(c_k'-c_j')(\ul{a})+(d(c_k-c_s)(\ul{a})a_{s,r})_{0\leq s\leq j,r\in[k_{s,j}]}\cdot\omega:\omega\in\{0,1\}^{\{(s,r):0\leq s\leq j,r\in[k_{s,j}]\}}\}$ for each $k=j+1,\dots,m$ and $j=1,\dots,m-1$,
\item we have
\[
\max_{i\in I_j}\max_{\ul{a}_j\in A_j}\max_{y\in[M]}|Q_{i,j}|(\ul{a}_j,y)\ll_{C,d,j}N
\]
for $j=1,\dots,m-1$,
\item and $i_{0,j}\in I_j$ equals the index such that $Q_{i_{0,j},j}$ has leading coefficient $c_{j+1}-c_j$ and degree $d-1$ coefficient
\[
c_{j,0}'(\ul{a},(a_{s,r})_{0\leq s\leq j,r\in[k_{s,j}]}):=(c_{j+1}'-c_j')(\ul{a})+d\sum_{\substack{0\leq s\leq j \\ r\in[k_{s,j}]}}(c_{j+1}-c_s)(\ul{a})a_{s,r}
\]
for $j=1,\dots,m-2$, and $i_{0,m-1}\in I_{m-1}$ equals the index such that $Q_{i_{0,m-1},m-1}$ has degree $d-1$ coefficient $(c'_{m}-c'_{m-1})(\ul{a})$
\end{enumerate}
such that
\[
\E_{\ul{a}\in A_j}^{\mu_j}\f{1}{N}\sum_{x\in\Z}\E_{y\in[M]}f_{\ul{a},j}(x)\prod_{i\in I_j}f'_{i}(x+Q_{i}(\ul{a},y))\gg_{C,d,j}\gamma^{O_{d,j}(1)},
\]
where $f_{\ul{a},j}$ is $1$-bounded for each $\ul{a}\in A_j$ and $f'_i$ equals $f$ or $\overline{f}$ for each $i\in I_j$, provided that $\gamma'\ll_{C,d,m}\gamma^{O_{d,m}(1)}$. Before showing that such a sequence of sets, measures, and elements exist, note that if $\gamma'\ll_{C,d,m}\gamma^{O_{d,m}(1)}$, then the conclusion of the lemma follows from one application of Lemma~\ref{lem4.5} when $j=m-1$, for as $s$ ranges over $0\leq s\leq m-1$, the polynomials $c_{m}-c_{s}$ range over all of the $c_i$'s by the assumption~\eqref{eq4.5} and our choice of enumeration $c_1,\dots,c_m$.

It remains to prove that the above sequence exists. As was mentioned earlier, for each $j=1,\dots,m-1$ this will follow from one application of Lemma~\ref{lem4.6} and then repeated applications of Lemma~\ref{lem4.4}, as in the proof of Lemma~\ref{lem4.7}. Let us assume then that $I_j,A_j,\mu_j,\Q_j,\mathcal{C}_j,\mathcal{C}_j^{(k)}$ for $k=j+1,\dots,m$, and $i_{0,j}$ satisfying the above conditions exist for some $j=0,\dots,m-2$. We first apply Lemma~\ref{lem4.6}, which we may do assuming that $\gamma'\ll_{C,d,m}\gamma^{O_{d,m}(1)}$, to get that
\[
\E_{\ul{a}\in A_{j,0}}^{\mu_{j,0}}\f{1}{N}\sum_{x\in\Z}\E_{y\in[M]}f(x)\prod_{i\in I_{j,0}}f_i(x+Q_{i,j,0}(\ul{a},y))\gg_{C,d,j}\gamma^{O_{d,j}(1)},
\]
where
\begin{enumerate}
\item $I_{j,0}=(I_j\times\{0,1\})\sm\{(i_{0,j},0)\}$,
\item $A_{j,0}=A_j\times((-\gamma'M,\gamma'M)\cap\Z)$,
\item $\mu_{j,0}(\ul{a})=\f{1_{A_{j,0}}(\ul{a})}{|A_j|}\mu_{\gamma'M}(a_{j,0})$,
\item $\Q_{j,0}:=(Q_{i,j,0})_{i\in I_{j,0}}$ has set of leading coefficients of degree $d$ polynomials, $\mathcal{C}_{j+1}$
\item $\mathcal{C}_{j,0}^{(k)}$, the set of coefficients of the degree $d-1$ terms of the degree $d$ polynomials in $\Q_{j,0}$ with leading coefficient $c_k-c_{j+1}$, equals
\begin{align*}
\bigg\{(c_k'-c_{j+1}')(\ul{a})-d\sum_{\substack{0\leq s\leq j \\ r\in[k_{s,j}]}}(c_{j+1}-c_s)(\ul{a})a_{s,r}+(d(c_k-c_s)&(\ul{a})a_{s,r})_{0\leq s\leq j,r\in[k_{s,j,0}]}\cdot\omega \\
&:\omega\in\{0,1\}^{\{(s,r):0\leq s\leq j,r\in[k_{s,j,0}]\}}\bigg\}
\end{align*}
for all $k=j+2,\dots,m$, where $k_{s,j,0}=k_{s,j}$ when $s<j$ and $k_{j,j,0}=k_{j,j}+1$,
\item we have
\[
\max_{i\in I_{j,0}}\max_{\ul{a}\in A_{j,0}}\max_{y\in[M]}|Q_{i,j,0}|(\ul{a},y)\ll_{C,d,j}N,
\]
\item and $f_i$ equals either $f$ or $\overline{f}$ for all $i\in I_{j,0}$.
\end{enumerate}
Let $\Q_{j,0}'$ denote the subset of $\Q_{j,0}$ consisting of polynomials of degree $d-1$. By our assumptions on $\Q_j$, the set of leading coefficients of elements of $\Q_{j,0}'$ is
\begin{align*}
\mathcal{C}_{j,0}':=&\{ c-c_{j,0}': c\in \mathcal{C}_j^{(j+1)}\}\sm\{0\} \\
=&\{(d(c_{j+1}-c_s)(\ul{a})a_{s,r})_{0\leq s\leq j,r\in[k_{s,j}]}\cdot(\omega-\ul{1}):\omega\in\{0,1\}^{\{(s,r):0\leq s\leq j,r\in[k_{s,j}]\}}\sm\{\ul{1}\}\}.
\end{align*}
Note that if $Q_i\in\Q_{j,0}'$, then $i$ has the form $i=(j+1,\omega)\in I_{j,0}$.

Next, we set $m':=|\mathcal{C}_{j}^{(j+1)}\sm\{c_{j,0}'\}|$ and enumerate the elements $c_{j,1}',\dots,c_{j,m'}'$ of $\mathcal{C}_{j}^{(j+1)}\sm\{c_{j,0}'\}$ by picking any ordering such that if $k\leq k'$, then
\[
c'_{j,k}(\ul{a},(a_{s,r})_{0\leq s\leq j,r\in[k_{s,j}]})=(c_{j+1}'-c_j')(\ul{a})+(d(c_{j+1}-c_s)(\ul{a})a_{s,r})_{0\leq s\leq j,r\in[k_{s,j}]}\cdot\omega
\]
and
\[
c'_{j,k'}(\ul{a},(a_{s,r})_{0\leq s\leq j,r\in[k_{s,j}]})=(c_{j+1}'-c_j')(\ul{a})+(d(c_{j+1}-c_s)(\ul{a})a_{s,r})_{0\leq s\leq j,r\in[k_{s,j}]}\cdot\omega'
\]
with $|\omega|\geq|\omega'|$ (note that this inequality goes in the opposite direction of the one used for the enumeration of elements of $\mathcal{C}$). This means that $c'_{j,m'}=c_{j+1}'-c_j'$.

Finally, to verify that we can indeed apply Lemma~\ref{lem4.4} repeatedly as in the proof of Lemma~\ref{lem4.7}, we note that if $K$ is any finite set, $B=((-\gamma' M,\gamma' M)\cap\Z)^u$ with $u\in\N$ and $0<\gamma'\leq 1$, $P_k\in\Z[b_1,\dots,b_u][y]$ for each $k\in K$ is a polynomial of degree at most $d$,
\[
\max_{k\in K}\max_{\ul{b}\in B}\max_{y\in [M]}|P_k|(\ul{b},y)\leq DN,
\]
and $k_0\in K$, then
\[
\max_{(k,\epsilon)\in(K\times\{0,1\})\sm\{(k_0,0)\}}\max_{\ul{b}\in B\times ((-\gamma'M,\gamma' M)\cap\Z)}\max_{y\in[M]}|P'_{k,\epsilon}|(\ul{b},y)\ll_d DN,
\]
where $P_{k,\epsilon}'(\ul{b},y):=P_k(b_1,\dots,b_u,y+\epsilon b_{u+1})-P_{k_0}(b_1,\dots,b_u,y)$. To see this, just note that
\[
|P'_{k,\epsilon}|(\ul{b},y)\leq |P_k|(b_1,\dots,b_u,y+\epsilon b_{u+1})+|P_{k_0}|(b_1,\dots,b_u,y)\leq |P_k|(b_1,\dots,b_u,y+\epsilon b_{u+1})+DN
\]
and
\[
|P_k|(b_1,\dots,b_u,y+\epsilon b_{u+1})\leq|P_k|(b_1,\dots,b_u,2M)\leq 2^dDN
\]
for all $\ul{b}\in B\times((-\gamma'M,\gamma'M)\cap\Z)$ and $y\in[M]$.

We now assume that $\gamma'\ll_{C,d,m}\gamma^{O_{d,m}(1)}$ and apply Lemma~\ref{lem4.4} repeatedly ($t_{j'}\ll_{d,m}1$ times for each $j'$) to produce a sequence of $m'$ finite sets $I_{j,j'}$ and $A_{j,j'}$, measures $\mu_{j,j'}$ supported on $A_{j,j'}$, and sets of polynomials $\Q_{j,j'}$ and $\Q_{j,j'}'$ satisfying
\begin{enumerate}
\item $I_{j,j'}=(I_{j,j'-1}\sm \{i\in I_{j,j'-1}:Q_{i,j,j'-1}\in\Q_{j,j'-1}\text{ and }Q_{i,j,j'-1}\text{ has leading coefficient }c_{j,j'}'-c_{j,j'-1}'\})\times\{0,1\}^{t_{j'}}$ for some $t_{j'}\ll_{d,m}1$ for $j'=1,\dots,m'$, 
\item $A_{j,j'}=A_{j,j'-1}\times ((-\gamma'M,\gamma'M)\cap\Z)^{t_{j'}}$ for $j'=1,\dots,m'$,
\item $\mu_{j,j'}(\ul{a},(a_{s,r})_{0\leq s\leq j+1,r\in[k_{s,j,j'}]})=\f{1_{A_{j,j'}}(\ul{a},(a_{s,r})_{0\leq s\leq j+1,r\in[k_{s,j,j'}]})}{|A_{j,j'-1}|(2\lfloor\gamma' M\rfloor+1)^{t_{j'}-1}}\mu_{\gamma'M}(a_{j+1,k_{s,j,j'}})$, where $k_{s,j,j'}=k_{s,j}$ for $s<j$, $k_{j,j,j'}=k_{j,j}+1$, and $k_{j+1,j,j'}=k_{j+1,j,j'-1}+O_{d,m}(1)$ for $j'=1,\dots,m'$,
\item  $\Q_{j,j'}'$ consists of all degree $d-1$ polynomials in $\Q_{j,j'}$
\item the set of leading coefficients of degree $d$ polynomials in $\Q_{j,j'}$ is $\mathcal{C}_{j+1}$,
\item $\Q_{j,j'}'$ has set of leading coefficients $\mathcal{C}_{j,j'}'$,
\item $\Q_{j,j'}$ has set of coefficients of degree $d-1$ terms of polynomials of degree $d$ with leading coefficient $c_k-c_{j+1}$ equal to $\mathcal{C}_{j,j'}^{(k)}$ for each $k=j+2,\dots,m$,
\item $\mathcal{C}_{j,j'}^{(k)}$ is equal to
\[
\{c'_{k}-c_{j}'-c'_{j,j'}+(d(c_k-c_s)(a_1,\dots,a_n)a_{s,r})_{0\leq s\leq j+1,r\in[k_{s,j,j'}]}\cdot\omega:\omega\in\{0,1\}^{\{(s,r):0\leq s\leq j+1,r\in[k_{s,j,j'}]\}}\}
\]
for all $k=j+2,\dots,m$ and $j'=1,\dots,m'$,
\item and
\[
\max_{i\in I_{j,j'}}\max_{\ul{a}\in A_{j,j'}}\max_{y\in[M]}|Q_{i,j,j'}|(\ul{a},y)\ll_{C,d,j+1} N
\]
\end{enumerate}
such that
\[
\E_{\ul{a}\in A_{j,j'}}^{\mu_{j,j'}}\f{1}{N}\sum_{x\in\Z}\E_{y\in[M]}f_{\ul{a},j,j'}(x)\prod_{i\in I_{j,j'}}f_i(x+Q_{i,j,j'}(\ul{a},y))\gg_{C,d,j+1}\gamma^{O_{d,j+1}(1)},
\]
where $f_{\ul{a},j,j'}$ is $1$-bounded for every $\ul{a}\in A_{j,j'}$ and $f_i$ equals either $f$ or $\overline{f}$ for every $i\in I_{j,j'}$, by picking $i_0$ corresponding to elements of $\Q'_{j,j'-1}$ with leading coefficient equal to $c_{j',j}'-c_{j'-1,j}'$ for each application of Lemma~\ref{lem4.4}. We then take $I_{j+1}=I_{j,m'}$, $A_{j+1}=A_{j,m'}$, $\mu_{j+1}=\mu_{j,m'}$, and $\Q_{j+1}=\Q_{j,m'}$.
\end{proof}

Continuing the example from after Lemma~\ref{lem4.7}, Lemma~\ref{lem4.8} may be used to control an average over the progression $x,x+3a_1y^2+3a_1^2y,x+3a_2y^2+3a_2^2y,x+3(a_1+a_2)y^2+3(a_1^2+a_2^2+2a_1a_2)y$ in terms of an average over progressions of the form
\begin{equation}\label{eq4.6}
(x+[(6(a_1+a_2)b_1,6a_1b_2,6a_1b_3,6a_2b_4,\cdots,6a_2b_{11},)\cdot\omega]y)_{\omega\in\{0,1\}^{11}}.
\end{equation}

Lemmas~\ref{lem4.7} and~\ref{lem4.8} combined show that $\Lambda_{P_1,\dots,P_\ell}^{N,M}(f_0,\dots,f_\ell)$ is controlled by an average of averages over the linear progression $(x+p(\ul{a})y)_{p\in\A_{\ell-1}\cup\{0\}}$.
\begin{lemma}\label{lem4.9}
Let $N,M>0$ and $P_1,\dots,P_\ell\in\Z[y]$ be polynomials with $(C,q)$-coefficients such that $\deg{P_i}=i$ for $i=1,\dots,\ell$ and $P_\ell$ has leading coefficient $c_\ell$. Let $I_j$ and $\A_j$ for $j=0,\dots,\ell-1$ be defined as in Section~\ref{sec3} with $c_\ell$ playing the role of $c$. There exist $k_i\ll_{\ell}1$ for all $i\in I_j$ and $j=0,\dots,\ell-2$ such that the following holds. If $1/C\leq q^{\ell-1}M^\ell/N\leq C$, $f_0,\dots,f_\ell:\Z\to\C$ are $1$-bounded functions supported on $[N]$,
\[
\left|\Lambda_{P_1,\dots,P_\ell}^{N,M}(f_0,\dots,f_\ell)\right|\geq\gamma,
\]
and $\gamma'\ll_{C,\ell}\gamma^{O_\ell(1)}$, then we have
\[
\E_{\ul{a}\in A}^\mu\f{1}{N}\sum_{x\in\Z}\E_{y\in[M]}f_\ell(x)\prod_{i\in I_{\ell-1}}f_i'(x+L_i(\ul{a},y))\gg_{C,\ell}\gamma^{O_\ell(1)},
\]
where
\begin{enumerate}
\item $A=((-\gamma' M,\gamma' M)\cap\Z)^{\sum_{j=1}^{\ell-2}\sum_{i\in I_{j}}k_i}$,
\item $\mu((a_{i,r}^{(j)})_{0\leq j\leq \ell,i\in I_{j-1},r\in[k_i]})=\f{1_{A}(\ul{a})}{(2\lfloor\gamma' M\rfloor+1)^{-1+\sum_{j=0}^{\ell-2}\sum_{i\in I_{j}}k_i}}\mu_{\gamma'M}(a_{i,k_{i}}^{(\ell-1)})$ for some $i\in I_{\ell-2}$,
\item $L_i\in\Z[\ul{a}][y]$ is a linear (in $y$) polynomial with leading coefficient equal to $p_i(\ul{a})\in \A_{\ell-1}$ for all $i\in I_{\ell-1}$,
\item we have
\[
\max_{i\in I_{\ell-1}}\max_{\ul{a}\in A}\max_{y\in[M]}|L_i|(\ul{a},y)\ll_{C,\ell}N,
\]
\item and $f_i'$ equals $f_\ell$ or $\overline{f_\ell}$ for all $i\in I_{\ell-1}$.
\end{enumerate}
\end{lemma}
\begin{proof}
Apply Lemma~\ref{lem4.7} once and then Lemma~\ref{lem4.8} $(\ell-2)$ times.
\end{proof}

Controlling the averages of linear progressions appearing in Lemma~\ref{lem4.9} by Gowers box norms is standard, and just requires $|I_{\ell-1}|-1$ more applications of the Cauchy--Schwarz and van der Corput inequalities.
\begin{lemma}\label{lem4.10}
Let $N,M>0$, $L_1,\dots,L_m\in\Z[y]$ be linear polynomials with zero constant term such that $L_i$ has leading coefficient $c_i$, and $f_0,\dots,f_m:\Z\to\C$ be $1$-bounded functions supported on the interval $[N]$. Assume further that
\[
\max_{i=1,\dots,m}\max_{y\in[M]}|L_i|(y)\leq CN.
\]
If
\[
\left|\Lambda_{L_1,\dots,L_m}^{N,M}(f_0,\dots,f_m)\right|\geq\gamma,
\]
and $\gamma'\ll_{C,m}\gamma^{O_m(1)}$, then we have
\[
\|f_m\|_{\square^m_{Q_0,\dots,Q_{m-1}}([N])}\gg_m\gamma^{O_m(1)},
\]
where $Q_0=c_m[\gamma' M]$ and $Q_i=(c_m-c_i)[\gamma'M]$ for $i=1,\dots,m-1$.
\end{lemma}
\begin{proof}
This will follow from $m-1$ applications of Lemma~\ref{lem4.2}, but applied in a slightly different manner than in the proofs of the other lemmas in this section. When $\gamma'\ll_C\gamma^2$ we have, by Lemma~\ref{lem4.2}, that
\[
\E_{h_0,h_0'\in [\gamma'M]}\f{1}{N}\sum_{x}\E_{y\in[M]}\Delta_{c_1h_0,c_1h_0'}'f_1(x)\prod_{i=2}^m\Delta_{c_ih_0,c_ih_0'}'f_i(x+(L_i-L_1)(y))\gg\gamma^2
\]
by unraveling the definition of $\mu_{\gamma' M}$ and making the change of variables $y\mapsto y+h_0'$. Next, we apply Lemma~\ref{lem4.2} again to the quantity inside of the average $\E_{h_0,h_0'\in [\gamma'M]}$ above and then use the Cauchy--Schwarz inequality (instead of applying Lemma~\ref{lem4.2} to the entire quantity in the left-hand side above, as we did before). Repeating this $m-2$ more times yields the conclusion of the lemma, since $L_i-L_j$ has leading coefficient $c_i-c_j$ for all $i,j\in[m]$.
\end{proof}
Finishing our example, we see that Lemma~\ref{lem4.10} can be used to control~\eqref{eq4.6}, and thus the progression $x,x+y,x+y^3$, in terms of an average over $a_1,a_2,b_1,\dots,b_{11}$ of the norm $\|\cdot\|_{\square^{2^{11}-1}_{(Q_\omega)_{\ul{0}\neq\omega\in\{0,1\}^{11}}}([N])}$, where
\[
Q_\omega=((6(a_1+a_2)b_1,6a_1b_2,6a_1b_3,6a_2b_4,\cdots,6a_2b_{11},)\cdot\omega)[\gamma'M]
\]
for each nonzero $\omega\in\{0,1\}^{11}$.

Now we can prove Proposition~\ref{prop3.4}.
\begin{proof}[Proof of Proposition~\ref{prop3.4}]
By Lemma~\ref{lem4.9}, we have that
\[
\E_{\ul{a}\in A}^\mu\f{1}{N}\sum_{x\in\Z}\E_{y\in[M]}f_\ell(x)\prod_{i\in I_{\ell-1}}f_i'(x+L_i(\ul{a},y))\gg_{\ell,C}\delta^{O_\ell(1)}
\]
when $\delta'\ll_{C,\ell}\delta^{O_\ell(1)}$, where $A$, $I_{\ell-1}$, $\A_{\ell-1}$, $f_i'$ for $i\in I_{\ell-1}$, and $L_i$ for $i\in I_{\ell-1}$ are as in the conclusion of Lemma~\ref{lem4.9}.

Set $m:=|I_{\ell-1}|$ and enumerate the elements $p_1,\dots,p_{m}$ of $\A_{\ell-1}$ by picking any ordering such that if $k\leq k'$, then $p_k=(p_i(\ul{a})a_{i,r}^{(\ell-1)})_{i\in I_{\ell-2},r\in[k_i]}\cdot\omega$ and $p_{k'}=(p_i(\ul{a})a_{i,r}^{(\ell-1)})_{i\in I_{\ell-2},r\in[k_i]}\cdot\omega'$ with $|\omega|\leq |\omega'|$. This means that $p_{m}=\sum_{i\in I_{\ell-2},r\in[k_i]}p_i(\ul{a})a_{i,r}^{(\ell-1)}$. Enumerate the $L_k$'s in the same manner, so that $L_k$ has leading coefficient $p_k$. Denote the constant term of $L_k$ by $p'_k$ for each $k\in[m]$ as well.

We now apply Lemma~\ref{lem4.2} once to deduce that
\[
\E_{\substack{\ul{a}\in A \\ h_0,h_0'\in [\delta' M]}}\Lambda_{p_{2}(\ul{a})y,\dots,p_{m}(\ul{a})y}^{N,M}(T_{p_1'(\ul{a})}\Delta_{p_1(\ul{a})(h_0,h_0')}'f_1'(x),\dots,T_{p_{m}'(\ul{a})}\Delta_{p_m(\ul{a})(h_0,h_0')}'f_m'(x))\gg_{C,\ell}\delta^{O_\ell(1)},
\]
assuming that $\delta'\ll_{C,\ell}\delta^{O_\ell(1)}$. We now apply, for each fixed $\ul{a}\in A$ and $(h_0,h_0')\in[\delta'M]^2$, Lemma~\ref{lem4.10} to $\Lambda_{p_{2}(\ul{a})y,\dots,p_{m}(\ul{a})y}^{N,M}(T_{p_1'(\ul{a})}\Delta_{p_1(\ul{a})(h_0,h_0')}'f_1'(x),\dots,T_{p_{m}'(\ul{a})}\Delta_{p_m(\ul{a})(h_0,h_0')}'f_m'(x))$ to get that
\[
\E_{\ul{a}\in A}\|T_{p_{m}'(\ul{a})}f_\ell\|_{\square^{|I_{\ell-1}|}_{(p(\ul{a})[\delta' M])_{p\in\A_{\ell-1}}}([N])}\gg_{C,\ell}\delta^{O_\ell(1)},
\]
again assuming that $\delta'\ll_{C,\ell}\delta^{O_\ell(1)}$ and recalling our choice of enumeration of elements of $\A_{\ell-1}$. To conclude, we note that $\|T_{p_{m}'(\ul{a})}f_\ell\|_{\square^{|I_{\ell-1}|}_{(p(\ul{a})[\delta' M])_{p\in\A_{\ell-1}}}([N])}=\|f_\ell\|_{\square^{|I_{\ell-1}|}_{(p(\ul{a})[\delta' M])_{p\in\A_{\ell-1}}}([N])}$ for each $\ul{a}\in A$ by making the change of variables $x\mapsto x-p_m'(\ul{a})$ inside of the definition of the Gowers box norm.
\end{proof}

\section{Concatenation}\label{sec5}

The main ingredient in the proof of Theorem~\ref{thm3.5} is the following result, whose proof will occupy the first part of this section.
\begin{lemma}\label{lem5.1}
Let $N,M_1,M_2>0$ with $M_2\leq M_1$ and $M_1M_2\leq N/|c|$, $b_1,\dots,b_s\in\Z$, and $f:\Z\to\C$ be a $1$-bounded function supported on the interval $[N]$. If $\gcd(a+b_i,a+b_j)\ll_s1/\gamma''$ for all distinct $i,j\in[s]$ and $|a+b_1|\geq\gamma''M_1$ for all but a $O_s(\gamma'')$ proportion of $a\in[M_1]$,
\[
\E_{a\in[M_1]}\|f\|_{\square^s_{(c(a+b_i)[M_2])_{i=1}^{s}}([N])}\geq\gamma,
\]
and $\gamma',\gamma''\ll_{s}\gamma^{O_s(1)}$, then there exists an $s'\ll_s 1$ such that
\[
\|f\|_{U^{s'}_{c[\gamma'M_1M_2]}([N])}\gg_s\gamma^{O_s(1)},
\]
provided that $M_1M_2\gg_{s}(\gamma\gamma')^{-O_s(1)}$.
\end{lemma}
Before beginning the proof of Lemma~\ref{lem5.1}, we record a couple of lemmas.
\begin{lemma}\label{lem5.2}
Let $M>0$. For all but a $O_s(\gamma)$-proportion of $s$-tuples $(a_1,\dots,a_s)\in[M]^s$, we have that
\[
\gcd((a_1,\dots,a_s)\cdot\omega,(a_1,\dots,a_s)\cdot\omega')<\gamma\1
\]
for all distinct $\omega,\omega'\in\{0,1\}^s\sm\{\ul{0}\}$, and for all but a $O_s(\gamma)$-proportion of pairs of $s$-tuples $(a_1,\dots,a_s,b_1,\dots,b_s)\in[M]^{2s}$, we have that
\[
\gcd((a_1-b_1,\dots,a_s-b_s)\cdot\omega,(a_1-b_1,\dots,a_s-b_s)\cdot\omega')<\gamma\1
\]
for all distinct $\omega,\omega'\in\{0,1\}^s\sm\{\ul{0}\}$.
\end{lemma}
\begin{proof}
These statements follow easily from the union bound and the fact that $\gcd(a,a')<\ve\1$ for all but a $O(\ve)$-proportion of $a,a'\in[M]$. Indeed, for each pair of distinct $\omega,\omega'\in\{0,1\}^s\sm\{\ul{0}\}$, the pair $((a_1,\dots,a_s)\cdot\omega,(a_1,\dots,a_s)\cdot\omega')$ ranges over a subset of $[sM]^2$ of density $\geq 1/s^2$ as $a_1,\dots,a_s$ ranges over $[M]$, and this pair hits each point in its range with multiplicity at most $M^{s-2}$. Thus, the total number of $s$-tuples $(a_1,\dots,a_s)\in[M]^s$ for which $\gcd((a_1,\dots,a_s)\cdot\omega,(a_1,\dots,a_s)\cdot\omega')\geq\gamma\1$ is $\ll\gamma s^2 M^s$. We conclude the first statement by taking the union bound over all $\ll_s 1$ pairs of distinct $\omega,\omega'\in\{0,1\}^s\sm\{\ul{0}\}$. The proof of the second statement is essentially the same.
\end{proof}

As in~\cite{PelusePrendiville19}, we will also need an inverse theorem for certain two-dimensional Gowers box norms. The one we prove next holds in greater generality than the inverse theorem in~\cite{PelusePrendiville19}, at the cost of a slightly weaker conclusion.
\begin{lemma}\label{lem5.3}
Let $N,M_1,M_2>0$ with $M_2\leq M_1$ and $M_1M_2\leq N/m$ and suppose that $|c|,|d|\in[M_1]$ with $|c|\geq\gamma_1M_1m$ and $\gcd(c,d)=m$. Let $f:\Z\to\C$ be a $1$-bounded function supported on the interval $[N]$. If
\[
\|f\|_{\square^2_{c[\gamma_2M_2],d[\gamma_2 M_2]}([N])}\geq\gamma
\]
and $0<\gamma_3<\gamma_2\leq\gamma_1\leq 1$, then there exist $1$-bounded functions $l,r:\Z\to\C$ satisfying
\[
\#\{x\in[N]:l(x)\neq l(x+dz)\text{ for some }z\in[\gamma_3 M_2]\}\ll\f{\gamma_3}{\gamma_2}N
\]
and
\[
\#\{x\in[N]:r(x)\neq r(x+cy)\text{ for some }y\in[\gamma_3 M_2]\}\ll\f{\gamma_3}{\gamma_2}N
\]
such that

\[
\left|\f{1}{N}\sum_{x\in\Z}f(x)l(x)r(x)\right|\geq\gamma^4.
\]
\end{lemma}
\begin{proof}
  By splitting $\Z$ up into progressions modulo $m$ and arguing as in the proof of Corollary~5.4 of~\cite{PelusePrendiville19}, it suffices to prove the $m=1$ case of the lemma. So, we assume for the remainder of the proof that $m=1$.

Since $c$ and $d$ are relatively prime, every $x\in\Z$ can be expressed uniquely as $x=cy+dz$ with $y\in\Z$ and $z\in[|c|]$. Thus, $\|f\|_{\square^2_{c[\gamma_2M_2],d[\gamma_2M_2]}([N])}^4$ can be written as
  \begin{align*}
    \f{1}{N}\sum_{\substack{u\in\Z \\ v\in[c]}}\E_{y,y',z,z'\in[\gamma_2M_2]}[&f(c(y+u)+d(z+v))\overline{f(c(y'+u)+d(z+v))}\\
    &\overline{f(c(y+u)+d(z'+v))}f(c(y'+u)+d(z'+v))].
  \end{align*}
  We split $\Z$ and $[|c|]$ up into intervals of length $\gamma_2M_2$ to write the above as
    \begin{align*}
      \f{1}{N/\gamma_2^2M_2^2}\sum_{\substack{u''\in\Z \\ 0\leq v''<|c|/\gamma_2M_2}}\E_{y,y',z,z',u',v'\in[\gamma_2M_2]}[&f(c(y+u'+\gamma_2M_2u'')+d(z+v'+\gamma_2M_2v''))\\
      &\overline{f(c(y'+u'+\gamma_2M_2u'')+d(z+v'+\gamma_2M_2v''))}\\
                                                                                                                      &\overline{f(c(y+u'+\gamma_2M_2u'')+d(z'+v'+\gamma_2M_2v''))}\\
      &f(c(y'+u'+\gamma_2M_2u'')+d(z'+v'+\gamma_2M_2v''))],
    \end{align*}
    using the fact that $|c|\geq\gamma_2M_2$. By the pigeonhole principle, there thus exist $y',z',u',v'\in[\gamma_2M_2]$ such that
\begin{align*}
\gamma^4\leq      \f{1}{N/\gamma_2^2M_2^2}\sum_{\substack{u''\in\Z \\ 0\leq v''<|c|/\gamma_2M_2}}\E_{y,z\in[\gamma_2M_2]}[&T_{c(u'+\gamma_2M_2u'')+d(v'+\gamma_2M_2v'')}f(cy+dz)\\
      &\overline{T_{c(u'+\gamma_2M_2u'')+d(v'+\gamma_2M_2v'')}f(cy'+dz)}\\
                                                                                                                      &\overline{T_{c(u'+\gamma_2M_2u'')+d(v'+\gamma_2M_2v'')}f(cy+dz')}\\
      &T_{c(u'+\gamma_2M_2u'')+d(v'+\gamma_2M_2v'')}f(cy'+dz')].
\end{align*}

Fix such $y',z',u',$ and $v'$. For each pair of integers $u''$ and $0\leq v''<|c|/\gamma_2M_2$, we define $1$-bounded functions $L_{u'',v''},R_{u'',v''}:[\gamma_2M_2]\to\C$ by setting
\[
  L_{u'',v''}(y):=\overline{T_{c(u'+\gamma_2M_2u'')+d(v'+\gamma_2M_2v'')}f(cy+dz')}
\]
and
\[
  R_{u'',v''}(z):=\overline{T_{c(u'+\gamma_2M_2u'')+d(v'+\gamma_2M_2v'')}f(cy'+dz)}\cdot T_{c(u'+\gamma_2M_2u'')+d(v'+\gamma_2M_2v'')}f(cy'+dz').
\]
We can then define $l_0,r_0:\Z\to\C$ by setting, for each $x\in\Z$ with $x=c(y+\gamma_2M_2y'')+d(z+\gamma_2M_2z'')$ for $y,z\in[\gamma_2M_2]$, $y''\in\Z$, and $0\leq z''<c/\gamma_2M_2$ an integer, $l_0(x):=L_{y'',z''}(y)$ and $r_0(x)=R_{y'',z''}(z)$. Then the above tells us that
\begin{equation}\label{eq5.1}
  \f{1}{N}\sum_{x\in\Z}f(x+cu'+dv')l_0(x)r_0(x)\geq\gamma^4.
\end{equation}

Next, we will show that
\[
\#\{x\in(-2N,2N)\cap\Z:l_0(x)\neq l_0(x+dw)\text{ for some }w\in[\gamma_3M_2]\}\ll\f{\gamma_3}{\gamma_2}N.
\]
By our definition of $l_0$, the left-hand side of the above is exactly the number of $x\in(-2N,2N)\cap\Z$ that can be written as $x=c(y+\gamma_2M_2y'')+d(z+\gamma_2M_2z'')$ with $y\in[\gamma_2M_2]$, $z\in[(\gamma_2-\gamma_3)M_2,\gamma_2M_2]$, $y''\in\Z$, and $0\leq z''<|c|/\gamma_2M_2$ an integer. The number of possible choices for $(y,z)$ is bounded by $\gamma_2\gamma_3M_2^2$. To count the number of possible choices for $(y'',z'')$ for each fixed pair $(y,z)$, note that since $|cy+dz|\ll\gamma_2N$ and the map $\Z\times([0,|c|/\gamma_2M_2)\cap\Z)\ni(y'',z'')\mapsto cy''+dz''$ is injective, the number of possible choices is bounded by the number of integers $0\leq z''<|c|/\gamma_2M_2$ and $w''\in[-O(N/\gamma_2M_2),O(N/\gamma_2M_2)]$ such that $dz''-w''$ is divisible by $c$. This quantity is bounded by $\ll (|c|/\gamma_2M_2)(N/\gamma_2M_2c)$, so that the number of possible $(y'',z'')$ is $\ll N/(\gamma_2M_2)^2$. We conclude that the number of such possible $(y,z,y'',z'')$ is $\ll\f{\gamma_3}{\gamma_2}N$. The same argument shows the corresponding bound for $r_0$.

To conclude, we make the change of variables $x\mapsto x-(cu'+dv')$ in~\eqref{eq5.1} and set $l(x):=l_0(x-(cu'+dv'))$ and $r(x):=r_0(x-(cu'+dv'))$, and note that since $|cu'+dv'|\ll N$, $x-(cu'+dv')\in(-2N,2N)$ whenever $x\in[N]$.
\end{proof}

The proof of Lemma~\ref{lem5.1} proceeds by induction on $s$. We first prove the $s=1$ and $s=2$ cases as separate lemmas.
\begin{lemma}[$s=1$ case of Lemma~\ref{lem5.1}]\label{lem5.4}
Let $N,M_1,M_2>0$ with $M_2\leq M_1$, $b\in\Z$, and $f:\Z\to\C$ be a $1$-bounded function supported on the interval $[N]$. If
\[
\E_{a\in[M_1]}\|f\|_{\square^1_{c(a+b)[M_2]}([N])}\geq\gamma
\]
and $0<\gamma'\leq 1$, then
\[
\|f\|_{U^2_{c[\gamma'M_1M_2]}([N])}\gg\gamma^{O(1)},
\]
provided that $M_1M_2\gg\gamma^{-O(1)}$.
\end{lemma}
\begin{proof}
Applying the Cauchy--Schwarz inequality to the average over $a\in[M_1]$ and expanding the definition of $\|f\|_{\square^1_{c(a+b)[M_2]}([N])}^2$, we have that
\[
\E_{a\in[M_1]}\f{1}{N}\sum_{x\in\Z}\E_{h,h'\in[M_2]}f(x+c(a+b)h)\overline{f(x+c(a+b)h')}\geq\gamma^2.
\]
Making the change of variables $x\mapsto x-c(a+b)h$ and swapping the order of summation, we get from the above that
\[
\f{1}{N}\sum_{x\in\Z}f(x)\left(\E_{a\in[M_1]}\E_{h,h'\in[M_2]}\overline{f(x+c(a+b)[h'-h])}\right)\geq\gamma^2.
\]
Since $f$ is $1$-bounded and supported on $[N]$, we have by another application of the Cauchy--Schwarz inequality and change of variables that
\[
\f{1}{N}\sum_{x\in\Z}\E_{a,a'\in[M_1]}\E_{h,h',h'',h'''\in[M_2]}f(x)\overline{f(x+c(a+b)[h'-h]-c(a'+b)[h'''-h''])}\geq\gamma^4,
\]
and then, by one more application of the Cauchy--Schwarz inequality and a change of variables, that
\[
\f{1}{N}\sum_{x\in\Z}\E_{a,a',a'',a'''\in[M_1]}\E_{h,h',h'',h'''\in[M_2]}f(x)\overline{f(x+c(a''-a)[h'-h]-c(a'''-a')[h'''-h''])}\geq\gamma^{8}.
\]

Note that $|h'-h|,|h'''-h''|>\gamma^{9}M_2$ for all but a $O(\gamma^9)$ proportion of $(h,h',h'',h''')\in[M_2]^4$ and, by Lemma~\ref{lem5.2}, we have $\gcd(h'-h,h'''-h'')<\gamma^{-9}$ for all but a $O(\gamma^9)$ proportion of $(h,h',h'',h''')\in[M_2]^4$. Thus, it follows from the above that
\[
\f{1}{N}\sum_{x\in\Z}\E_{a,a',a'',a'''\in[M_1]}\E_{\substack{h,h',h'',h'''\in[M_2] \\|h'-h|,|h'''-h''|>\gamma^9M_2 \\\gcd(h'-h,h'''-h'')<\gamma^{-9}}}f(x)\overline{f(x+c(a''-a)[h'-h]-c(a'''-a')[h'''-h''])}
\]
is $\gg\gamma^{8}$. We can write this as
\[
\f{1}{N}\sum_{x\in\Z}\sum_{w\in\Z}f(x)\overline{f(x+cw)}\mu(w)\gg\gamma^8,
\]
where
\[
\mu(w):=\E_{a,a',a'',a'''\in[M_1]}\E_{\substack{h,h',h'',h'''\in[M_2] \\|h'-h|,|h'''-h''|>\gamma^9M_2 \\\gcd(h'-h,h'''-h'')<\gamma^{-9}}}1_{w=(a''-a)[h'-h]-(a'''-a')[h'''-h'']}.
\]
Note that $\mu$ is supported on the interval $[-2M_1M_2,2M_1M_2]\cap\Z$.

By Fourier inversion, we have
\[
\int_{\T}\widehat{\mu}(\xi)\left(\f{1}{N}\sum_{x\in\Z}\sum_{|w|\leq 2M_1M_2}f(x)\overline{f(x+cw)}e(\xi w)\right)d\xi\gg\gamma^8,
\]
so that
\[
\left(\int_\T|\widehat{\mu}(\xi)|d\xi\right)\cdot\left(\max_{\xi\in\T}\left|\f{1}{N}\sum_{x\in\Z}\sum_{|w|\leq 2M_1M_2}f(x)\overline{f(x+cw)}e(\xi w)\right|\right)\gg\gamma^8.
\]
Now, note that
\[
  \mu=\E_{\substack{h,h',h'',h'''\in[M_2] \\|h'-h|,|h'''-h''|>\gamma^9M_2 \\\gcd(h'-h,h'''-h'')<\gamma^{-9}}}\nu_{\ul{h}}*\tilde{\nu}_{\ul{h}},
\]
where $\nu_{\ul{h}}(w)=\E_{a,a'\in[M_1]}1_{w=a[h'-h]-a'[h'''-h'']}$ and $\tilde{\nu}_{\ul{h}}(w)=\nu_{\ul{h}}(-w)$. Thus we have
\[
  \int_\T|\widehat{\mu}(\xi)|d\xi= \E_{\substack{h,h',h'',h'''\in[M_2] \\|h'-h|,|h'''-h''|>\gamma^9M_2 \\\gcd(h'-h,h'''-h'')<\gamma^{-9}}}\int_{\T}|\widehat{\nu_{\ul{h}}}(\xi)|^2d\xi= \E_{\substack{h,h',h'',h'''\in[M_2] \\|h'-h|,|h'''-h''|>\gamma^9M_2 \\\gcd(h'-h,h'''-h'')<\gamma^{-9}}}\sum_{w\in\Z}|\nu_{\ul{h}}(\xi)|^2,
\]
by Parseval's identity. Expanding the definition of $\nu_{\ul{h}}$, the above equals
\[
\E_{\substack{h,h',h'',h'''\in[M_2] \\|h'-h|,|h'''-h''|>\gamma^9M_2 \\\gcd(h'-h,h'''-h'')<\gamma^{-9}}}\f{\#\{a,a',a'',a'''\in[M_1]:(a''-a)[h'-h]=(a'''-a')[h'''-h'']}{M_1^4},
\]
which is bounded above by $\f{1}{M_1^4}\cdot M_1^2\cdot \f{M_1}{\gamma^{18}M_2}=\gamma^{-18}\f{1}{M_1M_2}$, using the assumption $M_1\geq M_2$.

Also note that, for each $\xi\in\T$, the quantity $\left|\f{1}{N}\sum_{x\in\Z}\E_{|w|\leq 2M_1M_2}f(x)\overline{f(x+cw)}e(\xi w)\right|$ is bounded above by $1+2\left|\f{1}{N}\sum_{x\in\Z}\E_{w\in[2M_1M_2]}f(x)\overline{f(x+cw)}e(\xi w)\right|$ since $f$ is $1$-bounded and supported on $[N]$.

Putting our two observations together, splitting the average over $[2M_1M_2]$ up into averages over intervals of length $\gamma'M_1M_2$, and using the pigeonhole principle, we thus deduce that there exists a $w'\in[2/\gamma']$ for which
\[
\left|\f{1}{N}\sum_{x\in\Z}\E_{w\in[\gamma'M_1M_2]}f(x)\overline{T_{cw'\gamma'M_1M_2}f(x+cw)}e(\xi w)\right|\gg\gamma^{O(1)},
\]
assuming that $M_1M_2\gg\gamma^{-O(1)}$. Inserting extra averaging in the $x$ variable by shifting by elements of $c[\gamma'M_1M_2]$ and applying the triangle inequality, we deduce from the above that
\[
\f{1}{N}\sum_{x\in\Z}\left|\E_{z,w\in[\gamma'M_1M_2]}f(x+cz)\overline{T_{cw'\gamma'M_1M_2}f(x+cz+cw)}e(\xi w)\right|\gg\gamma^{O(1)}.
\]
It now follows from Lemma~\ref{lem2.2} that $\|T_{cw'\gamma'M_1M_2}f\|_{U^2_{c[\gamma'M_1M_2]}([N])}\gg\gamma^{O(1)}$. To conclude, we make the change of variables $x\mapsto x-cw'\gamma'M_1M_2$ in the definition of the Gowers box norm.
\end{proof}

The $s=2$ case of Lemma~\ref{lem5.1} is a generalization of Lemma~5.5 of~\cite{PelusePrendiville19} (with a slightly weaker conclusion, getting $U^5$-control instead of $U^4$-control), and thus its proof closely follows the corresponding proof from~\cite{PelusePrendiville19}.

\begin{lemma}[$s=2$ case of Lemma~\ref{lem5.1}]\label{lem5.5}
Let $N,M_1,M_2>0$ with $M_2\leq M_1$ and $M_1M_2\leq N/c$, $b_1,b_2\in\Z$, and $f:\Z\to\C$ be a $1$-bounded function supported on the interval $[N]$. If $\gcd(a+b_1,a+b_2)\leq 1/\gamma''$ and $|a+b_1|>\gamma''M_1$ for all but a $O(\gamma'')$ proportion of $a\in[M_1]$,
\[
\E_{a\in[M_1]}\|f\|_{\square^2_{c(a+b_1)[M_2],c(a+b_2)[M_2]}([N])}\geq\gamma,
\]
$\gamma'\ll(\gamma\gamma'')^{O(1)}$, and $\gamma''\ll\gamma^{O(1)}$, then
\[
\|f\|_{U^5_{c[\gamma'M_1M_2]}([N])}\gg\gamma^{O(1)},
\]
provided that $M_1M_2\gg(\gamma\gamma')^{-O(1)}$.
\end{lemma}
\begin{proof}
By splitting $\Z$ up into arithmetic progressions modulo $c$ and arguing as in the proof of Corollary~5.6 of~\cite{PelusePrendiville19}, it suffices to prove the result in the $c=1$ case. In the $c=1$ case, the proof of Lemma~5.5 of~\cite{PelusePrendiville19} goes through with a small number of changes. Since that proof is seven pages long, we will mostly just indicate the differences. These differences mainly arise from the fact that $M_1$ and $M_2$ can have very different sizes in this lemma, while in the corresponding lemma in~\cite{PelusePrendiville19}, $M_1=M_2=N^{1/2}$.

With a view towards applying Lemma~\ref{lem5.3}, let $U_{b_1,b_2}$ denote the set of all $a\in[M_1]$ such that $|a+b_1|>\gamma''M_2$ and $\gcd(a+b_1,a+b_2)\leq 1/\gamma''$, so that $|U_{b_1,b_2}|=(1-O(\gamma''))M_1$ by hypothesis. The set $U_{b_1,b_2}$ will play the same role as the set $U_b$ does in the proof in~\cite{PelusePrendiville19}.  By applying Lemma~\ref{lem5.3} with $c=a+b_1$, $d=a+b_2$, and $\gamma_1=\gamma_2=(\gamma'')^2$, we then get that
\begin{equation}\label{eq5.2}
\E_{a\in U_{b_1,b_2}}\f{1}{N}\sum_{x\in\Z}f(x)l_{a+b_2}(x)r_{a+b_1}(x)\gg\gamma^{O(1)},
\end{equation}
where
\[
\#\{x\in[N]:l_{a+b_2}(x)\neq l_{a+b_2}(x+(a+b_2)z)\text{ for some }z\in[\ve M_2/(\gamma'')^2]\}\ll\f{\ve}{(\gamma'')^2}N
\]
and
\[
\#\{x\in[N]:r_{a+b_1}(x)\neq l_{a+b_1}(x+(a+b_1)y)\text{ for some }y\in[\ve M_2/(\gamma'')^2]\}\ll\f{\ve}{(\gamma'')^2}N
\]
for every $0<\ve\leq(\gamma'')^2$. Since $f$ is supported on $[N]$, we may assume without loss of generality that $l_{a+b_2}$ and $r_{a+b_1}$ are supported on $[N]$ as well.

Inserting extra averaging in the $x$ variable in the left-hand side of~\eqref{eq5.2} by shifting by elements of $(a+b_1)[\gamma'M_2]$, taking advantage of the almost-invariance of $r_{a+b_1}$ under shifts from this progression, and then applying the Cauchy--Schwarz inequality once, we can assume that~\eqref{eq5.2} holds (with a worse implied constant in the exponent of $\gamma$ on the right-hand side) with $r_{a+b_1}$ replaced by the function $r_{a+b_1}'(x):=\E_{w}^{\mu_{\gamma'M_2}}f(x+(a+b_1)w)l_{a+b_2}(x+(a+b_1)w)$ for each $a\in U_{b_1,b_2}$. As in~\cite{PelusePrendiville19}, we then apply the Cauchy--Schwarz inequality to double the $a$ variable, take advantage of the almost-invariance of $l_{a+b_2}$, $l_{a'+b_2}$, and $r'_{a'+b_1}$ again to insert extra averaging by elements of $(a+b_2)[\gamma'M_2]$,  $(a'+b_2)[\gamma'M_2]$, and  $(a'+b_1)[\gamma'M_2]$, respectively, and then use Lemma~\ref{lem2.2} to get that
\[
\E_{a,a'\in U_{b_1,b_2}}\|r'_{a+b_1}\|^8_{\square^3_{(a+b_2)[\gamma' M_2],(a'+b_2)[\gamma' M_2],(a'+b_1)[\gamma' M_2]}([N])}\gg\gamma^{O(1)},
\]
assuming that $\gamma'\ll\gamma^{O(1)}$.

One can then continue to argue in an almost-identical manner as in~\cite{PelusePrendiville19}, with the only differences being that we use Lemma~\ref{lem2.2} in place of the version of the Gowers--Cauchy--Schwarz inequality used in~\cite{PelusePrendiville19} and, instead of the measures $\nu_{a,a',\gamma_i}$ (using the notation of that paper) being supported on an interval of length on the order of $N$, they are supported on an interval of length on the order of $M_1M_2$, to get that
\[
\E_{a\in[M_1]}\|fl_{a+b_2}\|_{U^3_{[\gamma'M_1M_2]}([N])}\gg\gamma^{O(1)}.
\]

Taking advantage of the almost-invariance of $l_{a+b_2}$ and applying the Cauchy--Schwarz inequality as in the end of the proof of Lemma~5.5 of~\cite{PelusePrendiville19}, the above inequality implies that
\[
\E_{h_1,h_1',h_2,h_2',h_3,h_3'\in[\gamma'M_1M_2]}\left[\E_{a\in[M_1]}\|\Delta_{(h_1,h_1'),(h_2,h_2'),(h_3,h_3')}'f\|_{\square^1_{(a+b_1)[\gamma'M_1M_2]}([N])}\right]\gg\gamma^{O(1)}.
\]
We can then apply Lemma~\ref{lem5.4} to the inner average to conclude.
\end{proof}

Now we can finally prove Lemma~\ref{lem5.1} in general.
\begin{proof}[Proof of Lemma~\ref{lem5.1}]
  The proof of the lemma proceeds by induction on $s$, with the $s=1$ and $s=2$ cases handled in Lemmas~\ref{lem5.4} and~\ref{lem5.5}, respectively. So suppose that the result holds for a general $s\geq 2$, and assume that $b_1,\dots,b_{s+1}\in\Z$ satisfy the hypotheses of the lemma. Let $f:\Z\to\C$ be a $1$-bounded function supported on $[N]$ such that $\E_{a\in[M_1]}\|f\|_{\square_{(c(a+b_i)[M_2])_{i=1}^{s+1}}^{s+1}([N])}\geq\gamma$.

For each $a\in[M_1]$ and $\ul{h},\ul{h}'\in[M_2]^{s-1}$, we define the function $g_{a,\ul{h},\ul{h}'}:\Z\to\C$ by
  \[
\Delta'_{(c(a+b_i)(h_i,h_i'))_{i=1}^{s-1}}f(x)=f\left(x+\sum_{i=1}^{s-1}c(a+b_i)h_i\right)g_{a,\ul{h},\ul{h}'}(x).
  \]
Note that $g_{a,\ul{h},\ul{h}'}$ is $1$-bounded since $f$ is $1$-bounded. Since $\gcd(a+b_{s},a+b_{s+1})<1/\gamma''$ for all but a $O_s(\gamma'')$-proportion of the $a\in[M_1]$, we can thus apply Lemma~\ref{lem5.3} deduce that
\begin{equation}\label{eq5.3}
\left|\E_{\substack{a\in[M_1] \\h_1,\dots,h_{s-1}\in[M_2] \\ h_1',\dots,h_{s-1}'\in[M_2]}}\f{1}{N}\sum_{x\in\Z}f(x+\sum_{i=1}^{s-1}c(a+b_i)h_i)g_{a,\ul{h},\ul{h}'}(x)l_{a,\ul{h},\ul{h}'}(x)r_{a,\ul{h},\ul{h}'}(x)\right|\gg\gamma^{O_s(1)},
\end{equation}
where, for all $a\in[M_1]$ and $\ul{h},\ul{h}'\in[M_2]^{s-1}$, we have
\[
 \#\{x\in[N]:r_{a,\ul{h},\ul{h}'}(x)\neq r_{a,\ul{h},\ul{h}'}(x+(a+b_{s+1})z)\text{ for some }y\in[\ve M_2/(\gamma'')^2]\}\ll_s\f{\ve}{(\gamma'')^2}N
\]
and
\[
 \#\{x\in[N]:l_{a,\ul{h},\ul{h}'}(x)\neq l_{a,\ul{h},\ul{h}'}(x+(a+b_{s})z)\text{ for some }z\in[\ve M_2/(\gamma'')^2]\}\ll_s\f{\ve}{(\gamma'')^2}N
\]
for all $0<\ve<(\gamma'')^2$. (For the $O(\gamma'')$ proportion of $a\in[M_1]$ not satisfying the size or greatest common divisor hypotheses, we can just take $r_{a,\ul{h},\ul{h}'}$ and $l_{a,\ul{h},\ul{h}'}$ to be identically zero.)

We rearrange the left-hand side of~\eqref{eq5.3} as
\[
\left|\f{1}{N}\sum_{x\in\Z}\E_{h_1,\dots,h_{s-1}\in[M_2]}f(x+\sum_{i=1}^{s-1}c(a+b_i)h_i)\left(\E_{\substack{a\in[M_1] \\ h_1',\dots,h_{s-1}'\in[M_2]}}g_{a,\ul{h},\ul{h}'}(x)l_{a,\ul{h},\ul{h}'}(x)r_{a,\ul{h},\ul{h}'}(x)\right)\right|,
\]
and then apply the Cauchy--Schwarz inequality to get that
\[
\E_{\substack{a,a'\in[M_1]\\h_1,\dots,h_{s-1}\in[M_2] \\ h_1',\dots,h_{s-1}'\in[M_2] \\ k_1',\dots,k_{s-1}'\in[M_2]}}\f{1}{N}\sum_{x\in\Z}g_{a,\ul{h},\ul{h}'}(x)\overline{g_{a',\ul{h},\ul{k}'}(x)}l_{a,\ul{h},\ul{h}'}(x)\overline{l_{a',\ul{h},\ul{k}'}(x)}r_{a,\ul{h},\ul{h}'}(x)\overline{r_{a',\ul{h},\ul{k}'}(x)}\gg_s\gamma^{O_s(1)},
\]
using that $f$ is $1$-bounded and supported on $[N]$. By the pigeonhole principle, there exists $\ul{h}\in[M_2]^{s-1}$ such that
\begin{equation}\label{eq5.4}
\E_{\substack{a,a'\in[M_1] \\ h_1',\dots,h_{s-1}'\in[M_2] \\ k_1',\dots,k_{s-1}'\in[M_2]}}\f{1}{N}\sum_{x\in\Z}g_{a,\ul{h},\ul{h}'}(x)\overline{g_{a',\ul{h},\ul{k}'}(x)}l_{a,\ul{h},\ul{h}'}(x)\overline{l_{a',\ul{h},\ul{k}'}(x)}r_{a,\ul{h},\ul{h}'}(x)\overline{r_{a',\ul{h},\ul{k}'}(x)}\gg_s\gamma^{O_s(1)}.
\end{equation}
Fix this $\ul{h}$.

Since the quantity inside of the averages on the left-hand side of~\eqref{eq5.4} is $\ll_s 1$ for all $a,a'\in[M_1]$ and $\ul{h}',\ul{k}'\in[M_2]^{s-1}$,  we have that this quantity is $\gg_s\gamma^{O_s(1)}$ for a $\gg_s\gamma^{O_s(1)}$ proportion of $a,a'\in[M_1]$ and $\ul{h}',\ul{k}'\in[M_2]^{s-1}$. For such $a,a',\ul{h}',\ul{k}'$, we have that
\begin{align*}
\gamma^{O_s(1)}\ll_s\f{1}{N}\sum_{x\in\Z}\E_{\ell_1,\dots,\ell_4\in[\gamma' M_2]}&(g_{a,\ul{h},\ul{h}'}\overline{g_{a',\ul{h},\ul{k}'})}(x+(a+b_s,a'+b_s,a+b_{s+1},a'+b_{s+1})\cdot\ul{\ell}) \\
  &l_{a,\ul{h},\ul{h}'}(x+(a'+b_s,a+b_{s+1},a'+b_{s+1})\cdot(\ell_2,\ell_3,\ell_4)) \\
  &\overline{l_{a',\ul{h},\ul{k}'}(x+(a+b_s,a+b_{s+1},a'+b_{s+1})\cdot(\ell_1,\ell_3,\ell_4))} \\
  &r_{a,\ul{h},\ul{h}'}(x+(a+b_s,a'+b_s,a'+b_{s+1})\cdot(\ell_1,\ell_2,\ell_4)) \\
  &\overline{r_{a',\ul{h},\ul{k}'}(x+(a+b_s,a'+b_s,a+b_{s+1})\cdot(\ell_1,\ell_2,\ell_3))},
\end{align*}
by almost-invariance of $l_{a,\ul{h},\ul{h}'}(x),l_{a',\ul{h},\ul{k}'}(x),r_{a,\ul{h},\ul{h}'}(x),$ and $r_{a',\ul{h},\ul{k}'}(x)$ under shifts by elements of their corresponding progressions, and then, using Lemma~\ref{lem2.2}, we thus deduce that
\[
\E_{\substack{\ell_1,\dots,\ell_4\in[\gamma' M_2] \\\ell_1',\dots,\ell_4'\in[\gamma' M_2] }}\f{1}{N}\sum_{x\in\Z}\Delta'_{(a+b_s)(\ell_1,\ell_1'),(a'+b_s)(\ell_2,\ell_2'),(a+b_{s+1})(\ell_3,\ell_3'),(a'+b_{s+1})(\ell_4,\ell_4')}(g_{a,\ul{h},\ul{h}'}\overline{g_{a',\ul{h},\ul{k}'}})(x)\gg\gamma^{O_s(1)},
\]
assuming that $\gamma'\ll_s\gamma^{O_s(1)}$.

Expanding the definition of $g_{a,\ul{h},\ul{h}'}$ and $g_{a',\ul{h},\ul{k}'}$ and using that the $\Delta'$ operator distributes over products of functions, it follows that the quantity
\begin{align*} 
  \E_{\substack{a,a'\in[M_1]\\ h_1',\dots,h_{s-1}'\in[M_2]\\k_1',\dots,k_{s-1}'\in[M_2]\\\ell_1,\dots,\ell_{4}\in[\gamma' M_2] \\ \ell_1',\dots,\ell_{4}'\in[\gamma' M_2]}}\f{1}{N}\sum_{x\in\Z}\prod_{\substack{\omega\in\{0,1\}^{s-1} \\ \omega\neq\ul{0}}}[&f_{a,a',\ul{h},\ul{\ell},\ul{\ell}',\omega}(x+(c(a+b_i)_{i=1}^{s-1}h_i')\cdot\omega)\cdot \\
  &\overline{f_{a,a',\ul{h},\ul{\ell},\ul{\ell}',\omega}'(x+(c(a'+b_i)_{i=1}^{s-1}k_i')\cdot\omega)}]
\end{align*}
is $\gg_s\gamma^{O_s(1)}$, where 
\[
f_{a,a',\ul{h},\ul{\ell},\ul{\ell}',\omega}(x):=\Delta'_{(a+b_s)(\ell_1,\ell_1'),(a'+b_s)(\ell_2,\ell_2'),(a+b_{s+1})(\ell_3,\ell_3'),(a'+b_{s+1})(\ell_4,\ell_4')}f(x+(c(a+b_i)h_i)_{i=1}^{s-1}\cdot(\ul{1}-\omega))
\]
and
\[
f_{a,a',\ul{h},\ul{\ell},\ul{\ell}',\omega}'(x):=\Delta'_{(a+b_s)(\ell_1,\ell_1'),(a'+b_s)(\ell_2,\ell_2'),(a+b_{s+1})(\ell_3,\ell_3'),(a'+b_{s+1})(\ell_4,\ell_4')}f(x+(c(a'+b_i)h_i)_{i=1}^{s-1}\cdot(\ul{1}-\omega)).
\]
Taking the averages over $h_2',\dots,h_{s-1}'\in[M_2]$ and $k_2',\dots,k_{s-1}'\in[M_2]$ inside, we can rewrite the average above as 
\begin{align*}
  \E_{\substack{a,a'\in[M_1]\\ h_1',k_1'\in[M_2]\\\ell_1,\dots,\ell_{4}\in[\gamma' M_2] \\ \ell_1',\dots,\ell_{4}'\in[\gamma' M_2]}}\f{1}{N}\sum_{x\in\Z}[f_{a,a',\ul{h},\ul{\ell},\ul{\ell}',\omega_0}(x+c(a+b_1)h_1')\overline{f_{a,a',\ul{h},\ul{\ell},\ul{\ell}',\omega_0}'(x+c(a'+b_1)k_1')}\cdot& \\
\D_{a,a',h_1',\ul{h},\ul{\ell},\ul{\ell}'}(x)\D_{a,a',k_1',\ul{h},\ul{\ell},\ul{\ell}'}'(x)&],
\end{align*}
where $\omega_0=(1,0,\dots,0)$ and $\D_{a,a',h_1',\ul{h},\ul{\ell},\ul{\ell}'}(x)$ and $\D_{a,a',k_1',\ul{h},\ul{\ell},\ul{\ell}'}'(x)$ equal
\[
\E_{h_2',\dots,h_{s-1}'\in[M_2]}\prod_{\substack{\omega\in\{0,1\}^{s-2}\\ \omega\neq\ul{0}}}(T_{c(a+b_1)h_1'}f_{a,a',\ul{h},\ul{\ell},\ul{\ell}',1\omega}\cdot f_{a,a',\ul{h},\ul{\ell},\ul{\ell}',0\omega})(x+(c(a+b_i)h_i')_{i=2}^{s-1}\cdot\omega)
\]
and
\[
\E_{k_2',\dots,k_{s-1}'\in[M_2]}\prod_{\substack{\omega\in\{0,1\}^{s-2}\\ \omega\neq\ul{0}}}(T_{c(a'+b_1)k_1'}f'_{a,a',\ul{h},\ul{\ell},\ul{\ell}',1\omega}\cdot f'_{a,a',\ul{h},\ul{\ell},\ul{\ell}',0\omega})(x+(c(a'+b_i)k_i')_{i=2}^{s-1}\cdot\omega),
\]
respectively.

Note that, by Lemma~\ref{lem2.2}, if $g:\Z\to\C$ is any function supported on the interval $[N]$ such that $\left|\f{1}{N}\sum_{x\in\Z}f(x)\D_{a,a',h_1',\ul{h},\ul{\ell},\ul{\ell}'}(x)\right|\geq\delta$, then $\|f\|_{\square^{s-2}_{c(a+b_2)[M_2],\dots,c(a+b_{s-1})[M_2]}([N])}\geq\delta$. In this situation, we say that $\D_{a,a',h_1',\ul{h},\ul{\ell}}$ is \textit{structured} for the norm $\|\cdot\|_{\square^{s-2}_{c(a+b_2)[M_2],\dots,c(a+b_{s-1})[M_2]}([N])}$. Similarly, $\D'_{a,a',k_1',\ul{h},\ul{\ell},\ul{\ell}'}$ is structured for the norm $\|\cdot\|_{\square^{s-2}_{c(a'+b_2)[M_2],\dots,c(a'+b_{s-1})[M_2]}([N])}$. Using that $\D_{a,a',h_1',\ul{h},\ul{\ell},\ul{\ell}'}$ is structured for $\|\cdot\|_{\square^{s-2}_{c(a+b_2)[M_2],\dots,c(a+b_{s-1})[M_2]}([N])}$ for every $a,a'\in[M_1]$, $h_1'\in[M_2]$, and $\ul{\ell},\ul{\ell}'\in[\gamma' M_2]^4$, we thus deduce that
\begin{align*}
\gamma^{O_s(1)}\ll_s\E_{\substack{a,a'\in[M_1] \\ h_1',k_1'\in[M_2] \\ \ell_1,\dots,\ell_4\in[\gamma' M_2] \\ \ell_1',\dots,\ell_4'\in[\gamma' M_2] \\ h''_2,\dots,h_{s-1}''\in[M_2] \\ h'''_2,\dots,h_{s-1}'''\in[M_2]}}\f{1}{N}\sum_{x\in\Z}\bigg[\Delta'_{(c(a+b_i)(h_i'',h_i'''))_{i=2}^{s-1}}f_{a,a',\ul{h},\ul{\ell},\omega_0}(x+c(a+b_1)h_1')&\\
\overline{\Delta'_{(c(a+b_i)(h_i'',h_i'''))_{i=2}^{s-1}}f_{a,a',\ul{h},\ul{\ell},\omega_0}'(x+c(a'+b_1)k_1')}&\\
\Delta'_{(c(a+b_i)(h_i'',h_i'''))_{i=2}^{s-1}}\D_{a,a',k_1',\ul{h},\ul{\ell},\ul{\ell}'}'(x)&\bigg].
\end{align*}

We now analyze, for each $a,a'\in[M_1]$, $k'_1\in[M_2]$, and $\ul{\ell},\ul{\ell}'\in[\gamma'M_2]^4$, the function $\Delta'_{(c(a+b_i)(h_i'',h_i'''))_{i=2}^{s-1}}\D_{a,a',k_1',\ul{h},\ul{\ell},\ul{\ell}'}'(x)$, which equals
\begin{equation}\label{eq5.5}
\E_{\substack{k_{2}^{\omega'},\dots,k_{s-1}^{\omega'}\in[M_2] \\ \omega'\in\{0,1\}^{s-2}}}\prod_{\substack{\omega,\omega'\in\{0,1\}^{s-2}\\ \omega\neq\ul{0}}}f'_{a,a',k_1',\ul{h},\ul{h}'',\ul{h}''',\ul{\ell},\ul{\ell}',\omega,\omega'}(x+(c(a'+b_i)_{i=2}^{s-1}k_i^{\omega'})\cdot\omega),
\end{equation}
where $f'_{a,a',k_1',\ul{h},\ul{h}'',\ul{h}''',\ul{\ell},\ul{\ell}',\omega,\omega'}(x)$ equals
\[
(T_{c(a'+b_1)k_1'}f'_{a,a',\ul{h},\ul{\ell},\ul{\ell}',1\omega}\cdot f'_{a,a',\ul{h},\ul{\ell},\ul{\ell}',0\omega})(x+(c(a+b_i)h_i))_{i=2}^{s-1}\cdot\omega'+(c(a+b_i)h_i''')_{i=2}^{s-1}\cdot(\ul{1}-\omega')).
\]
It is not hard to show that any function of the form~\eqref{eq5.5} can be approximated by an average of structured functions for the norm $\|\cdot\|_{\square^{s-2}_{c(a'+b_2)[\gamma' M_2],\dots,c(a'+b_{s-1})[\gamma' M_2]}([N])}$. More specifically, any function of the form
\[
\D(x):=\E_{\substack{k_{1}^{\omega'},\dots,k_{t}^{\omega'}\in[M_2] \\ \omega'\in\{0,1\}^{t}}}\prod_{\substack{\omega,\omega'\in\{0,1\}^{t}\\ \omega\neq\ul{0}}}f_{\omega,\omega'}(x+(c(a'+b_i)_{i=1}^tk_i^{\omega'})\cdot\omega)
\]
can be approximated by
\[
\cE(x):=\E_{\substack{k_{1}^{\omega'},\dots,k_{t}^{\omega'}\in[M_2] \\ \omega'\in\{0,1\}^{t}}}\E_{k_1^0,\dots,k_{1}^0\in[\gamma 'M_2]}\prod_{\substack{\omega,\omega'\in\{0,1\}^{t}\\ \omega\neq\ul{0}}}f_{\omega,\omega',\ul{k^{\omega'}}}'(x+(c(a'+b_i)_{i=1}^tk_i^{0})\cdot\omega),
\]
where $f_{\omega,\omega',\ul{k^{\omega'}}}'(x):=f_{\omega,\omega'}(x+(c(a'+b_i')k_1^{\omega'})\cdot\omega)$, assuming that $\gamma'\ll\gamma^{O_s(1)}$ and all of the $f_{\omega,\omega'}$'s are $1$-bounded and supported on an interval of length $\ll N$.

Indeed, to see that $\cE$ approximates $\D$, we make the change of variables $k_i^{\omega'}\mapsto k_i^{\omega'}+k_i^0$ for each $\omega'\in\{0,1\}^{t}$ and $i=1,\dots,t$ and average over $k_1^{0},\dots,k_t^{0}\in[\gamma' M_2]$ to get that $\D(x)$ equals
\[
\E_{k_1^{0},\dots,k_t^{0}\in[\gamma' M_2]}\sum_{\substack{k_{1}^{\omega'},\dots,k_{t}^{\omega'}\in\Z \\ \omega'\in\{0,1\}^{t}}}\prod_{\substack{\omega'\in\{0,1\}^t \\ i=1,\dots,t}}\f{1_{[M_2]}(k_i^{\omega'}+k_i^0)}{M_2}\prod_{\substack{\omega,\omega'\in\{0,1\}^{t}\\ \omega\neq\ul{0}}}f_{\omega,\omega'}(x+(c(a'+b_i)(k_i^{\omega'}+k_i^{0}))_{i=1}^t\cdot\omega).
\]
Note that, for every $x\in\Z$, one can replace each $1_{[M_2]}(k_{i}^{\omega'}+k_i^0)$ above with $1_{[M_2]}(k_i^{\omega'})$, at the cost of an error of size $O(\gamma')$, for the functions $1_{[M_2]}(\cdot)$ and $1_{[M_2]}(\cdot+k_i^0)$ are equal outside of a set of size $O(\gamma' M_2)$. Hence, $\cE(x)=\D(x)+O_{t}(\gamma')$ for all $x\in\Z$. Note too that $\cE(x)$ and $\D(x)$ are supported on intervals of size $\ll N$, so that they are in fact both equal to $0$ outside of a set of size $\ll N$. As a consequence, we have that $\|\D-\cE\|_{\ell^1}\ll_t\gamma'N$.

In the particular situation we care about, the above argument implies that there exists a finite set $W$ for which
\begin{align*}
\E_{\substack{a,a'\in[M_1] \\ h_1',k_1'\in[M_2] \\ \ell_1,\dots,\ell_4\in[\gamma' M_2] \\ \ell_1',\dots,\ell_4'\in[\gamma' M_2] \\ h''_2,\dots,h_{s-1}''\in[M_2] \\ h'''_2,\dots,h_{s-1}'''\in[M_2] \\ w\in W}}\f{1}{N}\sum_{x\in\Z}\bigg[&\Delta'_{(c(a+b_i)(h_i'',h_i'''))_{i=2}^{s-1}}f_{a,a',\ul{h},\ul{\ell},\ul{\ell}',\omega_0}(x+c(a+b_1)h_1')\\
&\overline{\Delta'_{(c(a+b_i)(h_i'',h_i'''))_{i=2}^{s-1}}f_{a,a',\ul{h},\ul{\ell},\ul{\ell}',\omega_0}'(x+c(a'+b_1)k_1')}\D_{a,a',k_1',\ul{h},\ul{h}'',\ul{h}''',\ul{\ell},\ul{\ell}',w}'(x)\bigg]
\end{align*}
is $\gg_s\gamma^{O_s(1)}$, where each $\D_{a,a',k_1',\ul{h},\ul{h}'',\ul{h}''',\ul{\ell},\ul{\ell}',w}'$ is structured for $\|\cdot\|_{\square^{s-2}_{c(a'+b_2)[\gamma'M_2],\dots,c(a'+b_{s-1})[\gamma'M_2]}([N])}$. As a consequence, we get that
\begin{align*}
\E_{\substack{a,a'\in[M_1] \\ h_1',k_1'\in[M_2] \\ \ell_1,\dots,\ell_4\in[\gamma' M_2] \\ \ell_1',\dots,\ell_4'\in[\gamma' M_2] \\ h''_2,\dots,h_{s-1}''\in[M_2] \\ h'''_2,\dots,h_{s-1}'''\in[M_2] \\ k''_2,\dots,k_{s-1}''\in[\gamma' M_2] \\ k'''_2,\dots,k_{s-1}'''\in[\gamma' M_2]}}\f{1}{N}\sum_{x\in\Z}\bigg[\Delta'_{(c(a+b_i)(h_i'',h_i'''))_{i=2}^{s-1},(c(a'+b_i)(k_i'',k_i'''))_{i=2}^{s-1}}f_{a,a',\ul{h},\ul{\ell},\ul{\ell}',\omega_0}(x+c(a+b_1)h_1')&\\
\overline{\Delta'_{(c(a+b_i)(h_i'',h_i'''))_{i=2}^{s-1},(c(a'+b_i)(k_i'',k_i'''))_{i=2}^{s-1}}f_{a,a',\ul{h},\ul{\ell},\ul{\ell}',\omega_0}'(x+c(a'+b_1)k_1')}&\bigg]
\end{align*}
is $\gg_s\gamma^{O_s(1)}$. Making the change of variables $x\mapsto x-c(a'+b_1)k_1'$, and arguing as in the proof of Lemma~\ref{lem5.4}, it follows that
\[
\E_{\substack{a,a'\in[M_1] \\ k_1,k_1',k_2,k_2'\in[\gamma' M_1M_2] \\ \ell_1,\dots,\ell_4\in[\gamma'M_2] \\ \ell_1',\dots,\ell_4'\in[\gamma' M_2] \\ h''_2,\dots,h_{s-1}''\in[M_2] \\ h'''_2,\dots,h_{s-1}'''\in[M_2] \\ k''_2,\dots,k_{s-1}''\in[\gamma' M_2] \\ k'''_2,\dots,k_{s-1}'''\in[\gamma' M_2]}}\f{1}{N}\sum_{x\in\Z}\bigg[\Delta'_{(c(a+b_i)(h_i'',h_i'''))_{i=2}^{s-1},(c(a'+b_i)(k_i'',k_i'''))_{i=2}^{s-1},c(k_1,k_1'),c(k_2,k_2')}f_{a,a',\ul{h},\ul{\ell},\ul{\ell}',\omega_0}(x)\bigg]
\]
is $\gg_{s}\gamma^{O_s(1)}$, provided that $M_1M_2\gg_{s}(\gamma\gamma')^{O_s(1)}$. Recalling the definition of $f_{a,a',\ul{h},\ul{\ell},\ul{\ell}',\omega_0}$, making the change of variables $x\mapsto x-(c(a+b_i)h_i)\cdot(0,1,\dots,1)$ in the above, using the pigeonhole principle to restrict the $h_i''$'s and $h_i'''$'s to lie in intervals of length $\gamma'M_2$, applying Lemma~\ref{lem2.3}, and making a change of variables in $x$ now yields
\begin{align*}
\gamma^{O_s(1)}&\ll_s\E_{\substack{a,a'\in[M_1] \\ k_1,k_1',k_2,k_2'\in[\gamma' M_1M_2] \\ h_1'',h''_2,\dots,h_{s}''\in[\gamma' M_2] \\ h'''_1,\dots,h_{s}'''\in[\gamma' M_2] \\ k''_1,\dots,k_{s}''\in[\gamma' M_2] \\ k'''_1,\dots,k_{s}'''\in[\gamma' M_2]}}\f{1}{N}\sum_{x\in\Z}\Delta'_{(c(a+b_i)(h_i'',h_i'''))_{i=1}^{s},(c(a'+b_i)(k_i'',k_i'''))_{i=1}^{s},c(k_1,k_1'),c(k_2,k_2')}f(x) \\
&=\E_{\substack{a'\in[M_1] \\ k_1,k_1',k_2,k_2'\in[\gamma' M_1M_2] \\ k''_1,\dots,k_{s}''\in[\gamma' M_2] \\ k'''_1,\dots,k_{s}'''\in[\gamma' M_2]}}\left[\E_{a\in[M_1]}\|\Delta'_{(c(a'+b_i)(k_i'',k_i'''))_{i=1}^{s},c(k_1,k_1'),c(k_2,k_2')}f\|_{\square^{s}_{(c(a+b_i)[\gamma' M_2])_{i=1}^{s}}([N])}^{2^s}\right].
\end{align*}
We conclude by applying the induction hypothesis twice.
\end{proof}

For the sake of convenience, we record next how to combine Lemmas~\ref{lem5.1} and~\ref{lem5.2} for use in the proof of Theorem~\ref{thm3.5}.
\begin{lemma}\label{lem5.6}
Let $N,M_1,M_2>0$ with $M_2\leq M_1$ and $M_1M_2\leq N/c$ and $f:\Z\to\C$ be a $1$-bounded function supported on the interval $[N]$. If
\[
\E_{\substack{h_1,\dots,h_s\in[M_1] \\ h_1',\dots,h_s'\in[M_1]}}\|f\|_{\square^{2^s-1}_{((c(\ul{h}-\ul{h}')\cdot\omega)[M_2])_{\ul{0}\neq\omega\in\{0,1\}^s}}([N])}\geq\gamma
\]
and $\gamma'\ll_s\gamma^{O_s(1)}$, then there exists an $s'\ll_s 1$ such that
\[
\E_{\substack{h_1,\dots,h_{s-1}\in[M_1] \\ h_1',\dots,h_{s-1}'\in[M_1] \\ \ell_1,\dots,\ell_{s'}\in[\gamma'M_1M_2] \\ \ell'_1,\dots,\ell'_{s'}\in[\gamma'M_1M_2]}}\|\Delta'_{(c(\ell_i,\ell_i'))_{i=1}^{s'}}f\|_{\square^{2^{s-1}-1}_{((c(\ul{h}-\ul{h}')\cdot\omega)[M_2])_{\ul{0}\neq\omega\in\{0,1\}^{s-1}}}([N])}\gg_s\gamma^{O_s(1)},
\]
provided that $M_1M_2\gg_s(\gamma\gamma')^{-O_s(1)}$.
\end{lemma}
\begin{proof}
 Using H\"older's inequality and expanding the definition of the Gowers box norm gives
  \[
\E_{\substack{h_1,\dots,h_s\in[M_1] \\ h_1',\dots,h_s'\in[M_1]}}\f{1}{N}\sum_{x\in\Z}\E_{\substack{k_\omega,k_\omega'\in[M_2] \\ \ul{0}\neq\omega\in\{0,1\}^s}}\Delta'_{((c(\ul{h}-\ul{h}')\cdot\omega)(k_\omega,k_\omega'))_{\ul{0}\neq\omega\in\{0,1\}^s}}f(x)\geq\gamma^{O_s(1)}.
  \]
For all but a $O_s(\gamma^{O_s(1)})$ proportion of $h_1,\dots,h_{s-1},h_1',\dots,h_{s}'$, we have $|h_s-h_{s}'+(h_1-h_1',\dots,h_{s-1}-h_{s-1}')\cdot\omega|>\gamma^{O_s(1)}M_1$ for every $\omega\in\{0,1\}^{s-1}$ for all but a $O_s(\gamma^{O_s(1)})$-proportion of $h_s\in[M_1]$ and, by Lemma~\ref{lem5.2}, we have
\[
\gcd(h_s-h_s'+(h_1-h_1',\dots,h_{s-1}-h_{s-1}')\cdot\omega,h_s-h_s'+(h_1-h_1',\dots,h_{s-1}-h_{s-1}')\cdot\omega')<\gamma^{-O_s(1)}
\]
for every pair of distinct $\omega,\omega'\in\{0,1\}^{s-1}$ for all but a $O_s(\gamma^{O_s(1)})$-proportion of $h_s\in[M_1]$. For such $h_1,\dots,h_{s-1},h_1',\dots,h_{s}'\in[M_1]$ we apply Lemma~\ref{lem5.1} with $h_s$ playing the role of $a$, $b_\omega=-h_s'+(h_1-h_1',\dots,h_{s-1}-h_{s-1}')\cdot\omega$ for each $\omega\in\{0,1\}^{s-1}$, and the function $\Delta_{((c(\ul{h}-\ul{h}')\cdot\omega'')(k_{\omega''0},k_{\omega''0}'))_{\ul{0}\neq\omega''\in\{0,1\}^{s-1}}}'f$ playing the role of $f$. This yields
\[
\E_{\substack{k_{\omega''0},k_{\omega''0}'\in[M_2] \\ \ul{0}\neq\omega''\in\{0,1\}^{s-1}}}\E_{\substack{h_1,\dots,h_{s-1}\in[M_1] \\ h_1',\dots,h_{s-1}'\in[M_1]}}\|\Delta_{((c(\ul{h}-\ul{h}')\cdot\omega'')(k_{\omega''0},k_{\omega''0}'))_{\ul{0}\neq\omega''\in\{0,1\}^{s-1}}}'f\|^{2^{s'}}_{U^{s'}_{c[\gamma'M_1M_2]}([N])}\gg_s\gamma^{O_s(1)}
\]
for some $s'\ll_s 1$ by the positivity of Gowers box norms. Expanding the definition of the $U^{s'}$-norm shows that the left-hand side above equals
\[
\E_{\substack{k_{\omega''0},k_{\omega''0}'\in[M_2] \\ \ul{0}\neq\omega''\in\{0,1\}^{s-1}}}\E_{\substack{h_1,\dots,h_{s-1}\in[M_1] \\ h_1',\dots,h_{s-1}'\in[M_1]\\ \ell_1,\dots,\ell_{s'}\in[\gamma'M_1M_2] \\ \ell'_1,\dots,\ell'_{s'}\in[\gamma'M_1M_2]}}\f{1}{N}\sum_{x\in\Z}\Delta'_{(c(\ell_i,\ell_i'))_{i=1}^{s'}}\Delta_{((c(\ul{h}-\ul{h}')\cdot\omega'')(k_{\omega''0},k_{\omega''0}'))_{\ul{0}\neq\omega''\in\{0,1\}^{s-1}}}'f(x),
\]
and then using that the operators $\Delta'_{(c(\ell_i,\ell_i'))_{i=1}^{s'}}$ and $\Delta_{((c(\ul{h}-\ul{h}')\cdot\omega'')(k_{\omega''0},k_{\omega''0}'))_{\ul{0}\neq\omega''\in\{0,1\}^{s-1}}}'$ commute gives the conclusion of the lemma.
\end{proof}

Now we can prove Theorem~\ref{thm3.5}.
\begin{proof}[Proof of Theorem~\ref{thm3.5}]
 For each pair of $s$-tuples $\ul{h},\ul{h}'\in[M_1]^{s}$, we associate linear polynomials $L_{\ul{h},\ul{h}',\omega}\in\Z[a]$ with $L_{\ul{h},\ul{h}',\omega}(a):=c(\ul{h}\cdot\omega+\ul{h}'\cdot(\ul{1}-\omega))a$ and $1$-bounded functions $f_{\ul{h},\ul{h}',\omega}:\Z\to\C$ with $f_{\ul{h},\ul{h}',\omega}:=T_{(b_1h_1,\dots,b_sh_s)\cdot\omega+(b_1h_1',\dots,b_sh_s')\cdot(\ul{1}-\omega)}f$ for each $\omega\in\{0,1\}^s$. Enumerate the polynomials $L_1,\dots,L_{2^s}$ in $\{L_{\ul{h},\ul{h}',\omega}:\omega\in \{0,1\}^s\}$ and corresponding functions $f_1,\dots,f_{2^s}$ in $\{f_{\ul{h},\ul{h}',\omega}:\omega\in \{0,1\}^s\}$ by picking any ordering such that $L_{2^s}=L_{\ul{h},\ul{h}',\ul{1}}$, so that the assumption~\eqref{eq3.2} implies that
  \[
    \E_{\substack{h_1,\dots,h_s\in[M_1] \\ h_1',\dots,h_s'\in[M_1]}}\Lambda^{O_s(CN),M_2}_{L_1,\dots,L_{2^s}}(1,f_1,\dots,f_{2^s})\geq\delta^{O_s(1)}.
  \]
Then, since $|c(\ul{h}\cdot\omega+\ul{h}'(\ul{1}-\omega))a|\ll_s N$ for all $a\in[M_2]$ and $\ul{h},\ul{h}'\in[M_1]$, we can apply Lemma~\ref{lem4.10} to deduce that
\[
\E_{h_1,\dots,h_s\in[M_1]}\|f\|_{\square^{2^s-1}_{((c(\ul{h}-\ul{h}')\cdot\omega)[\gamma' M_2])_{\omega\in I}}([N])}^{2^{2^s-1}}\gg_{C,s}\delta^{O_s(1)},
\]
provided $\delta'\ll_{C,s}\delta^{O_s(1)}$. The conclusion of the lemma now follows by $s$ applications of Lemma~\ref{lem5.6}.
\end{proof}

The following lemma shows how Theorem~\ref{thm3.5} can be used to control averages of Gowers box norms of the type appearing in Proposition~\ref{prop3.4} in terms of averages of Gowers box norms in which some of the differencing directions $p(\ul{a})$ are replaced by directions $p'(\ul{a})$ of smaller degree depending on fewer entries of $\ul{a}$. We will then prove Proposition~\ref{prop3.6} by applying this lemma many times.
\begin{lemma}\label{lem5.7}
Let $N,M_1,M_2>0$ with $M_2\leq M_1$ and $M_1M_2\leq N/|c|$, $I$ and $A\subset \Z^n$ be finite sets, $p_i\in\Z[a_1,\dots,a_n]$ for each $i\in I$, and $f_{\ul{a}}:\Z\to\C$ for each $\ul{a}\in A$ be $1$-bounded functions supported on the interval $[N]$. Let $k_i\in\N$ for each $i\in I$, set $t:=\sum_{i\in I}k_i$, define finite sets $A':=((-M_2,M_2)\cap\Z)^t$, $I':=\{0,1\}^{\{(i,r): i\in I,r\in[k_i]\}}\sm\{\ul{0}\}$, and $\A'\subset\Z[a_1,\dots,a_n][a_{i,r}:i\in I,r\in[k_i]]$ by
\[
\A':=\{(p_i(a_1,\dots,a_n)a_{i,r})_{i\in I,r\in[k_i]}\cdot\omega:\omega\in I'\},
\]
and set $p'_{\omega}(a_1,\dots,a_n,(a_{i,r})_{i\in I,r\in[k_i]}):=(p_i(a_1,\dots,a_n)a_{i,r})_{i\in I,r\in[k_i]}\cdot\omega$ for each $\omega\in I'$. Further assume that
\begin{equation}\label{eq5.6}
\max_{i\in I}\max_{\ul{a}\in A}|p_i(\ul{a})|M_1M_2\leq CN.
\end{equation}
Let $k_\omega\in\N$ for each $\omega\in I'$. If
\[
\E_{\ul{a}\in A}\E_{\ul{a}'\in A'}\|f_{\ul{a}}\|_{\square^{\sum_{\omega\in I'}k_\omega}_{(p'_{\omega}(\ul{a},\ul{a}')[M_1])_{\omega\in I',r'\in[k_\omega]}}([N])}\geq\gamma
\]
and $\gamma'\ll_{C,t,(k_\omega)_{\omega\in I'}}\gamma^{O_{t}(1)}$, then for every $(i_0,r_0)\in I\times[k_{i_0}]$, we have
\[
\E_{\ul{a}\in A}\E_{\ul{b}\in B}\|f_{\ul{a}}\|_{\square^{t'+\sum_{\omega\in J}k'_\omega}_{(p_{i_0}(\ul{a})[\gamma'M_1M_2])_{u\in[t']},(q_{\omega}(\ul{a},\ul{b})[M_1])_{\omega\in J,r'\in[k'_{\omega}]}}([N])}\gg_{C,t,(k_\omega)_{\omega\in I'}}\gamma^{O_{t,(k_\omega)_{\omega\in I'}}(1)},
\]
where
\begin{enumerate}
\item $B:=((-M_2,M_2)\cap\Z)^{t-1}$,
\item $J:=\{0,1\}^{\{(i,r):i\in I,r\in[k_i]\}\sm\{(i_0,r_0)\}}\sm\{\ul{0}\}$ for some $t'\ll_{t,(k_\omega)_{\omega\in I'}} 1$,
\item and, for $\omega\in J$, we have $q_\omega:=p'_{\omega'}$ and $k_\omega':=k_{\omega'}$, where 
\[
\omega'_{(i,r)}:=\begin{cases}\omega_{(i,r)} & (i,r)\neq(i_0,r_0)\\
0 & (i,r)=(i_0,r_0)\end{cases},
\]
\end{enumerate}
provided that $M_1M_2\gg_{C,t,(k_\omega)_{\omega\in I'}}(\gamma\gamma')^{-O_{t,(k_\omega)_{\omega\in I'}}(1)}$.
\end{lemma}
For example, Lemma~\ref{lem5.7} allows us to control the average
\[
\E_{a_1\in A}\E_{|a_{0,1}|,|a_{0,2}|<M_2}\|f_a\|_{\square^3_{a_1a_{0,1}[M_1],a_1a_{0,2}[M_1],(a_1a_{0,1}+a_1a_{0,2})[M_1]}([N])}
\]
in terms of an average of the form
\[
\E_{a_1\in A}\E_{\substack{\ell_1,\dots,\ell_{t'}\in[\gamma'M_1M_2] \\ \ell_1',\dots,\ell_{t'}'\in[\gamma'M_1M_2]}}\E_{|a_{0,2}|<M_2}\|\Delta'_{a_1(\ell_1,\ell_1'),\dots,a_1(\ell_{t'},\ell_{t'}')}f_a\|_{\square^1_{a_1a_{0,2}[M_1]}([N])}
\]
for some $t'\ll 1$.
\begin{proof}
Since $\|f_{\ul{a}}\|_{\square^{\sum_{\omega\in I'}k_\omega}_{(p'_{\omega}(\ul{a},\ul{a}')[M_1])_{\omega\in I',r'\in[k_\omega]}}([N])}\leq 1$ for all $\ul{a}\in A$ and $\ul{a}'\in A'$, it follows that for at least a $\gg\gamma$ proportion of $\ul{a}\in A$ and $(a_{i,r})_{i\in I,r\in[k_i],(i,r)\neq(i_0,r_0)}\in((-M_2,M_2)\cap\Z)^{t-1}$ we have
\[
\E_{|a_{i_0,r_0}|< M_2}\|f_{\ul{a}}\|_{\square^{\sum_{\omega\in I'}k_\omega}_{(p'_{\omega}(\ul{a},\ul{a}')[M_1])_{\omega\in I',r'\in[k_\omega]}}([N])}\gg\gamma.
\]
Expanding the definition of the Gowers box norm, we have that
\begin{equation}\label{eq5.7}
\E_{|a_{i_0,r_0}|< M_2}\f{1}{N}\sum_{x\in\Z}\E_{\substack{h_{\omega,r'},h_{\omega,r'}'\in[M_1] \\ \omega\in I',r'\in[k_\omega]}}\Delta'_{(p'_{\omega}(\ul{a},\ul{a}')(h_{\omega,r'} ,h'_{\omega,r'}))_{\omega\in I',r'\in[k_\omega]}}f_{\ul{a}}(x)\gg\gamma^{O_{t,(k_\omega)_{\omega\in I'}}(1)},
\end{equation}
which is of the form that Theorem~\ref{thm3.5} can be applied to. Indeed, the left-hand side of~\eqref{eq5.7} can be written as
\[
\E_{\substack{m_{\omega',r'},m_{\omega',r'}'\in[M_1] \\ \omega'\in I', \omega'_{(i_0,r_0)}=0\\ r'\in[k_{\omega'}]}}\E_{|a_{i_0,r_0}|\leq M_2}\f{1}{N}\sum_{x\in\Z}\E_{\substack{h_{\omega,r''},h_{\omega',r''}\in[M_1] \\ \omega\in I',\omega_{(i_0,r_0)}=1 \\ r''\in[k_\omega]}}\Delta'_{((p_{i_0}(\ul{a})a_{i_0,r_0}+b_{\ul{a},\omega})(h_{\omega,r''},h_{\omega,r''}'))_{\omega\in I',\omega_{(i_0,r_0)}=1,r''\in[k_\omega]}}g_{\ul{a},\ul{m}}(x),
\]
where
\[
b_{\ul{a},\omega}=(p_{i}(\ul{a})a_{i,r})_{i\in I,r\in[k_i]}\cdot\omega-p_{i_0}(\ul{a})a_{i_0,r_0}
\]
and
\[
  g_{\ul{a},\ul{m}}=\Delta'_{(((p_{i}(\ul{a})a_{i,r})_{i\in I,r\in[k_i]}\cdot\omega')(m_{\omega',r'},m_{\omega',r'}'))_{\omega'\in I',\omega'_{i_0,r_0}=0, r'\in[k_{\omega'}]}}f_{\ul{a}}.
\]
The conclusion of the lemma now follows from Theorem~\ref{thm3.5}.
\end{proof}

We can now finally prove Proposition~\ref{prop3.6}. As mentioned above, this will be done by applying Lemma~\ref{lem5.7} many times. To illustrate how Lemma~\ref{lem5.7} will be applied, we will show how to control  an average of norms of the form $\|\cdot\|_{\square^{|I_2|}_{p_i(\ul{a})[N^{1/3}]}([N])}$ by a global $U^s$-norm for some $s\ll 1$, where $I_2$ and $(p_i)_{i\in I_2}=\mathcal{A}_2(3,1;(1,2,1))$ are as in the example between Theorem~\ref{thm3.3} and Proposition~\ref{prop3.4}.

Assuming that $f:\Z\to\C$ is $1$-bounded and supported on the interval $[N]$ and that
\[
\E_{\substack{a_{0,1},a_{0,2},a_{(1,0),1}\leq N^{\f{1}{3}}\\ a_{(0,1),1},a_{(0,1),2},a_{(1,1),1}\leq N^{\f{1}{3}}}}\|f\|_{\square^{15}_{((6(a_{0,1}a_{(1,0),1},a_{0,2}a_{(0,1),1},a_{0,2}a_{(0,1),2},(a_{0,1}+a_{0,2})a_{(1,1),1})\cdot\omega)[N^{\f{1}{3}}])_{\omega\in\{0,1\}^4\setminus\{\ul{0}\}}}([N])}\geq\gamma,
\]
we apply Lemma~\ref{lem5.7} with $(i_0,r_0)=((1,0),1)$ to deduce that
\[
\E_{\substack{\ell_{1},\dots,\ell_{s_1}\leq\gamma'N^{\f{2}{3}} \\ \ell_{1}',\dots,\ell_{s_1}'\leq\gamma'N^{\f{2}{3}} \\a_{0,1},a_{0,2},a_{(0,1),1}\leq N^{\f{1}{3}}\\ a_{(0,1),2},a_{(1,1),1}\leq N^{\f{1}{3}}}}\|\Delta'_{6a_{0,1}(\ell_1,\ell_1'),\dots,6a_{0,1}(\ell_{s_1},\ell_{s_1}')}f\|_{\square^{7}_{((6(a_{0,2}a_{(0,1),1},a_{0,2}a_{(0,1),2},(a_{0,1}+a_{0,2})a_{(1,1),1})\cdot\omega)[N^{\f{1}{3}}])_{\omega\in\{0,1\}^3\setminus\{\ul{0}\}}}([N])}\gg\gamma^{O(1)}
\]
for some $s_1\ll 1$ when $\gamma'\ll\gamma^{O(1)}$. For each fixed $\ell_1,\dots,\ell_{s_1},\ell_1',\dots,\ell_{s_1}'$, we apply Lemma~\ref{lem5.7} with $(i_0,r_0)=((0,1),2)$ to get that 
\[
\E_{\substack{\ell_{1},\dots,\ell_{s_1}\leq\gamma'N^{\f{2}{3}} \\ \ell_{1}',\dots,\ell_{s_1}'\leq\gamma'N^{\f{2}{3}} \\m_{1},\dots,m_{s_2}\leq\gamma'N^{\f{2}{3}} \\ m_{1}',\dots,m_{s_2}'\leq\gamma'N^{\f{2}{3}} \\a_{0,1},a_{0,2}\leq N^{\f{1}{3}}\\ a_{(0,1),1},a_{(1,1),1}\leq N^{\f{1}{3}}}}\|\Delta'_{\substack{6a_{0,1}(\ell_1,\ell_1'),\dots,6a_{0,1}(\ell_{s_1},\ell_{s_1}'),\\6a_{0,2}(m_1,m_1'),\dots,6a_{0,2}(m_{s_2},m_{s_2}')}}f\|_{\square^{3}_{((6(a_{0,2}a_{(0,1),1},(a_{0,1}+a_{0,2})a_{(1,1),1})\cdot\omega)[N^{\f{1}{3}}])_{\omega\in\{0,1\}^2\setminus\{\ul{0}\}}}([N])}\gg\gamma^{O(1)}
\]
for some $s_2\ll 1$ when $\gamma'\ll\gamma^{O(1)}$ and argue similarly with $(i_0,r_0)=((0,1),1)$ and $(i_0,r_0)=((1,1),1)$ to deduce that
\[
\f{1}{N}\sum_{x}\E_{\substack{\ell_{1},\dots,\ell_{s_1}, \ell_{1}',\dots,\ell_{s_1}'\leq\gamma'N^{\f{2}{3}} \\m_{1},\dots,m_{s_2},m_{1}',\dots,m_{s_2}'\leq\gamma'N^{\f{2}{3}} \\u_{1},\dots,u_{s_3},u_{1}',\dots,u_{s_3}'\leq\gamma'N^{\f{2}{3}} \\v_{1},\dots,v_{s_4}, v_{1}',\dots,v_{s_4}'\leq\gamma'N^{\f{2}{3}} \\a_{0,1},a_{0,2}\leq N^{\f{1}{3}}\\}}\Delta'_{\substack{6a_{0,1}(\ell_1,\ell_1'),\dots,6a_{0,1}(\ell_{s_1},\ell_{s_1}'),\\6a_{0,2}(m_1,m_1'),\dots,6a_{0,2}(m_{s_2},m_{s_2}')\\6a_{0,2}(u_1,u_1'),\dots,6a_{0,2}(u_{s_3},u_{s_3}'),\\6(a_{0,1}+a_{0,2})(v_1,v_1'),\dots,6(a_{0,1}+a_{0,2})(v_{s_4},v_{s_4}')}}f(x)\gg\gamma^{O(1)}
\]
for some $s_3,s_4\ll 1$ and $\gamma'\ll\gamma^{O(1)}$. We write the above as
\[
\E_{a_{0,1},a_{0,2}\leq N^{\f{1}{3}}}\|f\|_{\square^{s_1+s_2+s_3+s_4}_{((6(a_{0,1},a_{0,2})\cdot\omega)[\gamma'N^{\f{2}{3}}])_{\omega\in\{0,1\}^2\setminus\{\ul{0}\},r'\in[k_\omega]}}([N])}\gg\gamma^{O(1)}
\]
with $k_{(1,0)}=s_1$, $k_{(0,1)}=s_2+s_3$, and $k_{(1,1)}=s_4$, and apply Lemma~\ref{lem5.7} twice more with $(i_0,r_0)=(0,2)$ and then $(i_0,r_0)=(0,1)$ to deduce that $\|f\|_{U^{s}_{[\gamma''N]}([N])}\gg\gamma^{O(1)}$ for $\gamma''\ll\gamma^{O(1)}$ and some $s\ll 1$.
\begin{proof}[Proof of Proposition~\ref{prop3.6}]
By applying Lemma~\ref{lem5.7} $\sum_{i\in I_{\ell-1}}k_i$ times, once with $(i_0,r_0)=(i,r)$ for each $i\in I_{\ell-2}$ and $r\leq k_i$, we get that
\begin{equation}\label{eq5.8}
\E_{\ul{a}\in A}\|f\|_{\square^{\sum_{i\in I_{\ell-2}}t_i}_{(p_i(\ul{a})[\delta' M^{2}])_{i\in I_{\ell-2},r'\in[t_{i}]}}([N])}\gg_{C,\ell}\delta^{O_\ell(1)},
\end{equation}
where $1\leq t_{i}\ll_\ell 1$ for each $i\in I_{\ell-2}$, assuming that $\delta'\ll_{C,\ell}\delta^{O_\ell(1)}$. More generally, whenever $j=1,\dots,\ell-1$, from
\[
\E_{\ul{a}\in A}\|f\|_{\square^{\sum_{i\in I_{\ell-j}}t_i}_{(p_i(\ul{a})[\delta'M^j])_{i\in I_{\ell-j},r'\in[t_i]}}([N])}\geq \gamma
\]
one can deduce
\[
\E_{\ul{a}\in A}\|f\|_{\square^{\sum_{i\in I_{\ell-(j+1)}}t_i}_{(p_i(\ul{a})[\delta'M^{(j+1)}])_{i\in I_{\ell-(j+1)},r'\in[t_i]}}([N])}\gg_{C,\ell} \gamma^{O_\ell(1)},
\]
where $1\leq t_i\ll_{\ell}1$ for each $i\in I_{\ell-(j+1)}$, by applying Lemma~\ref{lem5.7} once with $(i_0,r_0)=(i,r)$ for each $i\in I_{\ell-(j+1)}$ and $r\leq k_i$. Starting from~\eqref{eq5.8} and repeating this implication $\ell-2$ more times gives the conclusion of the proposition.
\end{proof}

\section{Control by uniformity norms}\label{sec6}

In this section, we combine the results of Sections~\ref{sec4} and~\ref{sec5} to control the general average $\Lambda_{P_1,\dots,P_m}^{N,M}(f_0,\dots,f_\ell;\psi_{\ell+1},\dots,\psi_m)$ in terms of $U^s$-norms of $f_\ell$ and $F_\ell$. We will also state and prove Theorem~\ref{thm6.2}, the control result for general polynomial progressions mentioned in the introduction.

Theorem~\ref{thm3.7} follows almost immediately from the results already proven.
\begin{proof}[Proof of Theorem~\ref{thm3.7}]
  Set $c':=(\deg{P_\ell})!c_\ell$. By making the change of variables $x\mapsto x+c'z$ in the definition of $\Lambda^{N,M}_{P_1,\dots,P_m}$ and averaging over $z\in[\delta'M^{\deg{P_\ell}}]$, we have that
  \[
    \left|\E_{y\in[M]}\psi_{\ell+1}(P_{\ell+1}(y))\cdots\psi_m(P_m(y))\f{1}{N}\sum_{x\in\Z}\left(\E_{z\in[\delta'M^\ell]}f_0(x+c'z)\cdots f_\ell(x+c'z+P_\ell(y))\right)\right|\geq\delta.
  \]
  By one application of the Cauchy--Schwarz inequality in the $x$ and $y$ variables, we thus get  \[
\left|\E_{z,z'\in[\delta' M^\ell]}\Lambda^{N,M}_{P_1,\dots,P_\ell}(\Delta'_{c'(z,z')}f_0,\dots,\Delta'_{c'(z,z')}f_\ell)\right|\gg_{C,\deg{P_\ell}}\delta^2,
\]
so it follows from Propositions~\ref{prop3.4} and~\ref{prop3.6} that
\[
\E_{z,z'\in[\delta'M^\ell]}\f{1}{N}\sum_{x\in\Z}\E_{\substack{h_i,h_i'\in[\delta'M^\ell] \\ i=1,\dots,s}}\Delta_{c'(h_1,h_1'),\dots,c'(h_s,h_{s'})}(\Delta_{c'(z,z')}f)(x)\gg_{C,\deg{P_\ell}}\delta^{O_{\deg{P_\ell}}(1)}
\]
for some $s\ll_\ell 1$, which gives the conclusion of the theorem.
\end{proof}

We now deduce control for $\Lambda_{P_1,\dots,P_m}^{N,M}(f_0,\dots,f_\ell;\psi_{\ell+1},\dots,\psi_m)$ in terms of $U^s$-norms of dual functions by using the Cauchy--Schwarz inequality once and then applying Theorem~\ref{thm3.7}.
\begin{proof}[Proof of Corollary~\ref{cor3.8}]
Note that $\Lambda_{P_1,\dots,P_m}^{N,M}(f_0,\dots,f_\ell;\psi_{\ell+1},\dots,\psi_m)=\f{1}{N}\sum_{x}f_{\ell}(x)F_\ell(x)$, so that an application of the Cauchy--Schwarz inequality gives
\[
\left|\Lambda_{P_1,\dots,P_m}^{N,M}(f_0,\dots,f_{\ell-1},F_{\ell};\psi_{\ell+1},\dots,\psi_m)\right|\geq\delta^2.
\]
Corollary~\ref{cor3.8} now follows from Theorem~\ref{thm3.7} with $F_\ell$ (which is a $1$-bounded function supported on an interval of the form $[O_{\deg{P_\ell}}(CN)]$) playing the role of $f_\ell$.
\end{proof}

\subsection{Control for general polynomial progressions} In this subsection, we prove the following result, whose proof largely follows the proofs of Propositions~\ref{prop3.4} and~\ref{prop3.6}.
\begin{theorem}\label{thm6.2}
Let $N,M>0$, $P_1,\dots,P_m\in\Z[y]$ be polynomials such that $\deg{P_1}\leq\dots\leq\deg{P_m}$ and each $P_i$ has leading coefficient $c_i$. There exists an $s\ll_{\deg{P_1},\dots,\deg{P_m}}1$ such that the following holds. If $m':=\#\{i\in[m-1]:\deg{P_i}=\deg{P_m}\}$, $1/C\leq |c_i|M^{\deg{P_m}}/N\leq C$ for each $m-m'\leq i\leq m$, all of the coefficients of $P_1,\dots,P_m$ have absolute value bounded by $C|c_m|$, $f_0,\dots,f_m:\Z\to\C$ are $1$-bounded functions supported on the interval $[N]$,
\[
|\Lambda_{P_1,\dots,P_m}(f_0,\dots,f_m)|\geq\delta,
\]
and $\delta'\ll_{C,\deg{P_1},\dots,\deg{P_m}}\delta^{O_{\deg{P_1},\dots,\deg{P_m}}(1)}$, then we have
\[
\|f_m\|_{\square^s_{Q_1,\dots,Q_s}([N])}\gg_{C,\deg{P_1},\dots,\deg{P_m}}\delta^{O_{\deg{P_1,\dots,\deg{P_m}}}(1)},
\]
where each $Q_i$ equals $(\deg{P_m})!c_m[\delta'M^{\deg{P_m}}]$ or $(\deg{P_m})!(c_m-c_j)[\delta'M^{\deg{P_m}}]$ for some $m-m'\leq j<m$, provided that $N\gg_{C,\deg{P_1},\dots,\deg{P_m}}(|c_m|/\delta\delta')^{O_{\deg{P_1},\dots,\deg{P_m}}(1)}$.
\end{theorem}
If $c_{m-(m'-1)},\dots,c_m$ are uniformly bounded, or, more generally, are of the form $c_i'q$ for bounded $c_i'$, then it follows easily from Theorem~\ref{thm6.2} that $\Lambda_{P_1,\dots,P_m}(f_0,\dots,f_m)$ is controlled by a $U^s$-norm of $f_m$. To prove Theorem~\ref{thm6.2}, all we need beyond the results of Sections~\ref{sec4} and~\ref{sec5} is a more general version of Lemma~\ref{lem4.7}, which we now prove.

\begin{lemma}\label{lem6.3}
Let $N,M>0$ and $P_1,\dots,P_m\in\Z[y]$ be polynomials such that $\deg{P_1}\leq\dots\leq\deg{P_m}$ and each $P_i$ has leading coefficient $c_i$. If $m':=\#\{i\in[m-1]:\deg{P_i}=\deg{P_m}\}$, $1/C\leq |c_i|M^{\deg{P_m}}/N\leq C$ for each $m-m'\leq i\leq m$, all of the coefficients of $P_1,\dots,P_m$ have absolute value bounded by $C|c_m|$, $f_0,\dots,f_m:\Z\to\C$ are $1$-bounded functions supported on the interval $[N]$,
  \[
|\Lambda_{P_1,\dots,P_m}(f_0,\dots,f_m)|\geq\gamma,
\]
and $\gamma'\ll_{C,\deg{P_1},\dots,\deg{P_m}}\gamma^{O_{\deg{P_1},\dots,\deg{P_m}}(1)}$, then we have
\[
\E_{\ul{a}\in A}^{\mu}\f{1}{N}\sum_{x\in\Z}\E_{y\in[M]}f_m(x)\prod_{i\in I}f'_i(x+Q_i(\ul{a},y))\gg_{C,\deg{P_1},\dots,\deg{P_m}}\gamma^{O_{\deg{P_1},\dots,\deg{P_m}}(1)},
\]
where
\begin{itemize}
\item $I=\{0,1\}^t\sm\{\ul{0}\}$ for some $t\ll_{\deg{P_1},\dots,\deg{P_{m}}}1$,
\item $A=((-\gamma' M,\gamma' M)\cap\Z)^t$,
\item $\mu(a_1,\dots,a_t)=\f{1_A(a_1,\dots,a_t)}{(2\lfloor\gamma' M\rfloor+1)^{t-1}}\mu_{\gamma'M}(a_t)$,
\item the collection $\mathcal{Q}:=(Q_i)_{i\in I}$ consists only of polynomials of degree $\deg{P_m}-1$, each of which has distinct leading coefficient, and the set of such leading coefficients is
  \[
    \{((\deg{P_m})d_1a_1,\dots,(\deg{P_m})d_ta_t)\cdot\omega:\omega\in I\},
  \]
  where each $d_i$ equals $c_m$ or $c_m-c_j$ for some $m-m'\leq j<m$,
\item we have
  \[
\max_{i\in I}\max_{\ul{a}\in A}\max_{y\in[M]}|Q_i|(\ul{a},y)\ll_{C,\deg{P_1},\dots,\deg{P_m}}N,
  \]
\item and $f_i'$ equals either $f_m$ or $\overline{f_m}$ for all $i\in I$.
\end{itemize}
\end{lemma}
\begin{proof}
Arguing as in the proof of Lemma~\ref{lem4.7}, we apply Lemma~\ref{lem4.4} $t_0\ll_{\deg{P_1},\dots,\deg{P_{m-m'-1}}}1$ times to deduce that
  \[
\E_{\ul{a}\in A_{0}}^{\mu_0}\f{1}{N}\sum_{x\in\Z}\E_{y\in[M]}f_{\ul{a},0}(x)\prod_{\substack{j\in J_0 \\ \deg{Q_j}\neq 0}}g_{j,0}(x+Q_j(a_1,\dots,a_{t_1},y))\gg_{t_0}\gamma^{O_{t_0}(1)},
\]
where $J_0\subset[m]\times\{0,1\}^{t_0}$, $A_0=((-\gamma'M,\gamma'M)\cap\Z)^{t_0}$, $\mu_1(a_1,\dots,a_{t_0})=\f{1_{A_0}(a_1,\dots,a_{t_0})}{(2\lfloor\gamma' M\rfloor+1)^{t_0-1}}\mu_{\gamma'M}(a_{t_{0}})$, $\mathcal{Q}_0:=(Q_0)_{j\in J_{0}}$ consists only of polynomials of degree $\deg{P_m}$ and constant (in $y$) polynomials, the leading coefficients of degree $\deg{P_m}$ polynomials in $\mathcal{Q}_0$ are $c_{m-m'},\dots,c_m$, there are $2^{t_0}$ polynomials of degree $\deg{P_m}$ in $\mathcal{Q}_0$ with leading coefficient equal to $c_i$ for each $m-m'\leq i\leq m$ with set of degree $\deg{P_m}-1$ coefficients equal to $\{(c_ia_1,\dots,c_ia_{t_0})\cdot\omega:\omega\in\{0,1\}^{t_0}\}$, $f_{\ul{a},0}$ is $1$-bounded for each $\ul{a}\in A_0$, and $g_{j,0}$ equals either $f_{j'}$ or $\overline{f_{j'}}$ if $Q_j$ has leading coefficient $c_{j'}$, provided that $\gamma'\ll_{C,\deg{P_1},\dots,\deg{P_{m-m'-1}}}\gamma^{O_{\deg{P_1},\dots,\deg{P_{m-m'-1}}}(1)}$, by arguing exactly as in the proof of Lemma~\ref{lem4.7}, except using the assumption that the coefficients of $P_1,\dots,P_m$ are all bounded in absolute value by $C|c_m|$ in place of the $(C,q)$-coefficients hypothesis.

The conclusion of the lemma now follows by arguing almost exactly as in the proof of Lemma~\ref{lem4.8}, with the only differences being that we start with more polynomials of degree $\deg{P_m}$ with each leading coefficient and we already have an ordering $c_{m-(m'-1)},\dots,c_m$ of these coefficients (and do not care whether they have any particular structure), by applying Lemma~\ref{lem4.5} after repeating the following $m'-1$ times: apply Lemma~\ref{lem4.6} once, and then Lemma~\ref{lem4.4} as many times as necessary until we can apply one of Lemmas~\ref{lem4.5} or~\ref{lem4.6}.
\end{proof}

The proof of Theorem~\ref{thm6.2} is exactly the same as the proof of Theorem~\ref{thm3.7}, except that one uses Lemma~\ref{lem6.3} in place of Lemma~\ref{lem4.7} and does not need to do the initial application of the Cauchy--Schwarz inequality done in the proof of Theorem~\ref{thm3.7}.
\begin{proof}[Proof of Theorem~\ref{thm6.2}]
Following the proof of Proposition~\ref{prop3.4}, we apply Lemma~\ref{lem6.3} once, Lemma~\ref{lem4.8} $(\deg{P_m}-2)$ times, Lemma~\ref{lem4.10} once, and then, following the proof of Proposition~\ref{prop3.6}, Lemma~\ref{lem5.7} $\ll_{\deg{P_1},\dots,\deg{P_m}}1$ times.
\end{proof}

\section{Lemmas for degree-lowering}\label{sec7}

In this section, we collect and prove various lemmas needed for the proofs of Lemmas~\ref{lem3.9} and~\ref{lem3.10}. The first two lemmas are standard results on Weyl sums that can be found, for example, in~\cite{Tao12} as Lemmas~1.1.16 and~1.1.14, respectively.
\begin{lemma}\label{lem7.1}
Let $N>0$ and $P\in\R[y]$ be a polynomial with $P(y)=a_my^m+\dots+a_0$. If
\[
\left|\sum_{n\in[N]}e(P(y))\right|\geq\gamma N,
\]
then there exists $q\in\N$ satisfying $q\ll\gamma^{-O_m(1)}$ such that
\[
\|q a_i\|\ll\f{\gamma^{-O_m(1)}}{N^i}
\]
for each $i=1,\dots,m$.
\end{lemma}
\begin{lemma}\label{lem7.2}
Let $N,\ve,\gamma>0$ with $\ve\ll 1$, $\gamma\gg\ve$, and $N\gg\gamma\1$. If $\|n\beta\|\leq\ve$ for at least a $\gamma$-proportion of $n\in[-N,N]\cap\Z$, then there exists a positive integer $q\ll\gamma\1$ such that $\|q\beta\|\leq \ve q/\gamma N$.
\end{lemma}
We also record, for the sake of convenience, the following result, which can be found in~\cite{PelusePrendiville19} as Lemma~6.5.
\begin{lemma}\label{lem7.3}
Let $\alpha\in \T$. If $a,b\in\N$ are such that
\[
\left|\alpha-\f{a}{b}\right|\leq\gamma,
\]
then, for any $D\geq 1$, there exists an integer $k$ with $|k|\leq D$ and a $\theta\in[-1,1]$ such that
\[
\alpha=\f{a}{b}+k\f{\gamma}{D}+\theta\f{\gamma}{D}.
\]
\end{lemma}

Before stating and proving the remaining lemmas in this section, we need one more piece of notation. For $s\in\N$ and $H\subset\Z^{2s}$, let $\square_{s}(H)$ denote the set of $3s$-tuples
\[
(k_{1}^{(1)},\dots,k_{s}^{(1)},k_{1}^{(2)},\dots,k_{s}^{(2)},k_{1}^{(3)},\dots,k_{s}^{(3)})\in\Z^{3s}
\]
such that $(k_{1}^{(1)},\dots,k_{s}^{(1)},k_{1}^{(\omega_1+2)},\dots,k_{s}^{(\omega_s+2)})\in H$ for all $\omega\in\{0,1\}^s$. Note that this is not the same definition of $\square_s(H)$ that appeared in~\cite{PelusePrendiville19}, where $\square_s(H)$ instead consisted of $2s$-tuples.

The following lemma will play a similar role in the proof of the degree-lowering result in this paper as Lemma~6.3 of~\cite{PelusePrendiville19} played in that paper, and its proof follows the same general strategy, with differences mainly arising from dealing with more general dual functions and from the use of different definitions of the $U^s$-norm.
\begin{lemma}\label{lem7.4}
Let $L,M>0$, $2\leq\ell\leq m$, $H\subset[\gamma'L]^{2s}$ with $|H|\geq\gamma L^{2s}$, $f_0,\dots,f_{\ell-1}:\Z\to\C$ be $1$-bounded functions supported on the interval $[L]$, and $\psi_{\ell+1},\dots,\psi_m:\Z\to S^1$ be characters. Let $F_\ell$ be defined as in Corollary~\ref{cor3.8}. If
\begin{equation}\label{eq7.1}
\E_{(\ul{h},\ul{h}')\in H}\left|\f{1}{L}\sum_{x\in\Z}\Delta'_{(h_i,h_i')_{i=1}^s}F_\ell(x)e(\phi(\ul{h},\ul{h}')x)\right|^2\geq\gamma
\end{equation}
for some $\phi:H\to\T$, then
\[
\E_{\ul{k}\in\square_s(H)}\left|\f{1}{L}\sum_{x\in\Z}G_{\ell,\ul{k}}(x)e(\psi(\ul{k})x)\right|^2\geq(\gamma\gamma')^{O_{s}(1)},
\]
where
\[
G_{\ell,\ul{k}}(x):=\E_{y\in[M]}\Delta'_{(k_{i}^{(2)},k_{i}^{(3)})_{i=1}^s}f_0(x-P_\ell(y))\cdots \Delta'_{(k_{i}^{(2)},k_{i}^{(3)})_{i=1}^s}f_{\ell-1}(x+P_{\ell-1}(y)-P_\ell(y))
\]
and
\[
\psi(\ul{k}):=\sum_{\omega\in\{0,1\}^s}(-1)^{|\omega|}\phi(k_{1}^{(1)},\dots,k_{s}^{(1)},k_{1}^{(\omega_1+2)},\dots,k_{s}^{(\omega_s+2)}).
\]
\end{lemma}
For example, when $s=2$, the function $\psi(\ul{k})$ equals
\[
\phi(k_1^{(1)},k_2^{(1)},k_1^{(2)},k_2^{(2)})-\phi(k_1^{(1)},k_2^{(1)},k_1^{(2)},k_2^{(3)})-\phi(k_1^{(1)},k_2^{(1)},k_1^{(3)},k_2^{(2)})+\phi(k_1^{(1)},k_2^{(1)},k_1^{(3)},k_2^{(3)}).
\]
\begin{proof}[Proof of Lemma~\ref{lem7.4}]
Define, for each $y\in[M]$, the function
\[
F_{\ell,y}(x):=f_0(x-P_\ell(y))\cdots f_{\ell-1}(x+P_{\ell-1}(y)-P_\ell(y))\psi_{\ell+1}(P_{\ell+1}(y))\cdots\psi_m(P_m(y)),
\]
so that $F_\ell(x)=\E_{y\in[M]}F_{\ell,y}(x)$. We can thus write the left-hand side of~\eqref{eq7.1} as
\begin{align*}
  \E_{\substack{y_{\omega 0},y_{\omega 1}\in[M] \\ \omega\in\{0,1\}^s}}\E_{(\ul{h},\ul{h}')\in H}\f{1}{L^2}\sum_{x,z\in\Z}e(\phi(\ul{h},\ul{h}')(x-z))\prod_{\omega\in\{0,1\}^s}[&F_{\ell,y_{\omega0}}(x+\ul{h}\cdot\omega+\ul{h}'\cdot(\ul{1}-\omega))\cdot\\
&\overline{F_{\ell,y_{\omega1}}(z+\ul{h}\cdot\omega+\ul{h}'\cdot(\ul{1}-\omega))}].
\end{align*}
Applying the Cauchy--Schwarz inequality to double the $h_1'$ variable gives the bound
\begin{align*}
(\gamma\gamma')^{O(1)}\leq \E_{\substack{y_{\omega 0},y_{\omega 1}\in[M] \\ \omega\in\{0,1\}^s}}\sum_{\substack{\ul{h},\ul{h}'\in [\gamma'L]^{2s} \\ h_1''\in[\gamma'L] }}&\f{1_H(\ul{h},\ul{h}')1_{H}(\ul{h},h_1'',h_2',\dots,h_s')}{L^{2s+1}}\cdot \\
&\bigg[\f{1}{L^2}\sum_{x,z\in\Z}\prod_{\substack{ \omega\in\{0,1\}^s\\ \omega_1=0}}\Delta_{h_1''-h_1}F_{\ell,y_{\omega0}}(x+\ul{h}\cdot\omega+\ul{h}'\cdot(\ul{1}-\omega))\\
&\quad \quad \quad \quad \quad \quad \ \ \overline{\Delta_{h_1''-h_1'}F_{\ell,y_{\omega1}}(z+\ul{h}\cdot\omega+\ul{h}'\cdot(\ul{1}-\omega))}\\
&\quad \quad \quad \quad \quad \quad \ \ e((\phi(\ul{h},\ul{h}')-\phi(\ul{h},h_1'',h_2',\dots,h_s'))(x-z))\bigg],
\end{align*}
by using the fact that $H\subset[\gamma'L]^{2s}$ and $|H|\geq\gamma L^{2s}$. Note that nothing inside of the above average depends on the variables $y_{\omega0},y_{\omega1}$ for any $\omega\in\{0,1\}^s$ with $\omega_1=1$, so we can restrict the first average to $y_{\omega0},y_{\omega1}\in[M]$ with $\omega_1=0$.

We apply the Cauchy--Schwarz inequality $s$ total times in this manner, doubling the $h_i'$ variable for each $i=1,\dots,s$, to get that
\[
\E_{y_0,y_1\in[M]}\E_{\ul{k}\in\square_s(H)}\f{1}{L^2}\sum_{x,z\in\Z}\Delta'_{(k_{i}^{(2)},k_{i}^{(3)})_{i=1}^s}F_{\ell,y_0}(x)\overline{\Delta'_{(k_{i}^{(2)},k_{i}^{(3)})_{i=1}^s}F_{\ell,y_1}(z)}e(\psi(\ul{k})(x-z))\geq(\gamma\gamma')^{O_s(1)},
\]
using the trivial upper bound $|\square_s(H)|\leq (\gamma'L)^{3s}$. Finally, note that the left-hand side of the above inequality equals
\[
\E_{\ul{k}\in\square_s(H)}\left|\f{1}{L}\sum_{x\in\Z}G_{\ell,\ul{k}}(x)e(\psi(\ul{k})x)\right|^2
\]
by recalling the definition of $F_{\ell,y}$ and using the fact that the $\Delta'$ operator distributes over the product of functions (the characters in $F_{\ell,y}$ cancel since $s\geq 1$).
\end{proof}

The final lemma of this section is a generalization of Lemma~6.4 of~\cite{PelusePrendiville19}, and its proof is essentially the same as the argument in~\cite{PelusePrendiville19}.

\begin{lemma}\label{lem7.5}
Let $L>0$ and, for each $i=1,\dots,s$, let $\phi_i:\Z^{2s}\to\T$ be a function not depending on the $(s+i)^{th}$ variable. If $0<\gamma'\leq 1$, $f:\Z\to\C$ is $1$-bounded and supported on the interval $[L]$, and
\begin{equation}\label{eq7.2}
\E_{\ul{h},\ul{h}'\in [\gamma'L]^s}\left|\f{1}{L}\sum_{x\in\Z}\Delta'_{(h_i,h_i')_{i=1}^s}f(x)e\left(\sum_{i=1}^s\phi_i(\ul{h},\ul{h}')x\right)\right|^2\geq\gamma,
\end{equation}
then $\|f\|^2_{U^{s+1}_{[\gamma'L]}([L])}\gg_s\gamma^{O_s(1)}$.
\end{lemma}
\begin{proof}
Expanding the square, the left-hand side of~\eqref{eq7.2} can be written as
\[
\f{1}{L^2}\sum_{x,z\in\Z}\E_{\ul{h},\ul{h}'\in [\gamma'L]^s}\Delta'_{(h_i,h_i')_{i=1}^s}f(x)\overline{\Delta'_{(h_i,h_i')_{i=1}^s}f(z)}e\left(\sum_{i=1}^s\phi_i(\ul{h},\ul{h}')[x-z]\right),
\]
so that applying Lemma~\ref{lem2.2} for each fixed $x,z\in\Z$ and $\ul{h}\in[\gamma'L]^s$ gives
\[
\f{1}{L^2}\sum_{x,z\in\Z}\E_{\ul{h}',\ul{h}''\in[\gamma'L]^s}\Delta'_{(h_i',h_i'')_{i=1}^s}f(x)\overline{\Delta'_{(h_i',h_i'')_{i=1}^s}f(z)}\geq\gamma^{O_s(1)}.
\]
By inserting extra averaging in the $x$ variable and using the pigeonhole principle to fix $z$ (which we may do since $f$ is supported on $[L]$ and $\gamma'\leq 1$), it follows that
\[
\f{1}{L}\sum_{x\in\Z}\E_{\ul{h}',\ul{h}''\in[\gamma'L]^s}\overline{\Delta'_{(h_i',h_i'')_{i=1}^s}f(x)}\E_{w\in[\gamma'L]}\Delta'_{(h_i',h_i'')_{i=1}^s}f(x+w)\gg_s\gamma^{O_s(1)}
\]
for some $z\in\Z$. To conclude, we apply the Cauchy--Schwarz inequality to double the $w$ variable, again using that $f$ is supported on $[L]$ and $\gamma'\leq 1$.
\end{proof}

\section{Degree-lowering}\label{sec8}

We begin by handling the base case of the inductive proof of Lemmas~\ref{lem3.9} and~\ref{lem3.10}.

\begin{lemma}\label{lem8.1}
Let $N,M>0$, $P_1,\dots,P_m\in\Z[y]$ be polynomials such that $P_1$ and $P_2$ have $(C,q)$-coefficients, $\deg{P_1}<\dots<\deg{P_m}$, and $P_i$ has leading coefficient $c_i$ for $i=1,\dots,m$, and $\psi_2,\dots,\psi_m:\Z\to S^1$ be characters such that $\psi_i(x)=e(\alpha_ix)$ with $\alpha_i\in\T$ for $i=2,\dots,m$. Assume further that $|c_1|M^{\deg{P_1}}/N\leq C$. If there exist $1$-bounded functions $f_0,f_1:\Z\to\C$ supported on the interval $[N]$ such that
\begin{equation}\label{eq8.1}
\left|\f{1}{N/c}\sum_{x\in\Z}F_{2}(cx)\psi_\ell(cx)\right|\geq\gamma,
\end{equation}
where $F_2$ is as in Corollary~\ref{cor3.8}, then there exists a positive integer $t\ll_{C,\deg{P_m}}\gamma^{-O_{\deg{P_m}}(1)}$ such that
\[
\|tc^{\deg{P_m}}c_m\alpha_m\|\ll_{C,\deg{P_m}}\f{\gamma^{-O_{\deg{P_m}}(1)}}{(M/|c|)^{\deg{P_m}}},
\]
provided that $N\gg_{C,\deg{P_m}}(q/\gamma)^{O_{\deg{P_m}}(1)}$.
\end{lemma}
Note that the hypothesis $c_1M^{\deg{P_1}}/N\leq C$ above actually follows from the slightly stronger condition $1/C\leq |c|M^{\deg{P_2}}/N\leq C$ in Lemma~\ref{lem3.10} and the assumptions that $P_1$ has $(C,q)$-coefficients, $\deg{P_2}>\deg{P_1}$, and $N\gg_{C,\deg{P_m}}(q/\gamma)^{O(1)}$. So, this lemma does indeed cover the $\ell=2$ case of Lemma~\ref{lem3.10}.
\begin{proof}
Inserting the definition of $F_2$, the inequality~\eqref{eq8.1} reads
\[
\left|\f{1}{N/c}\sum_{x\in\Z}\E_{y\in[M]}g_0(cx-P_2(y))g_1(cx+P_1(y)-P_2(y))\psi_2(cx)\psi_3(P_3(y))\cdots\psi_m(P_m(y))\right|\geq\gamma.
\]
We split the sum over $y\in[M]$ up into progressions modulo $c$ by writing $y=cz+h$ for $h=0,\dots,|c|-1$ and use the pigeonhole principle to fix an $h$ such that
\begin{align*}
\bigg|\f{1}{N/c}\sum_{x\in\Z}\E_{z\in[M/|c|]}g_0(cx-P_2(cz+h))g_1(cx+P_1(cz+h)-P_2(cz+h))& \\
\psi_2(cx)\psi_3(P_3(cz+h))\cdots\psi_m(P_m(cz+h))&\bigg|\gg\gamma,
\end{align*}
provided that $N\gg\gamma^{-O(1)}$. Note that $\f{P_2(cz+h)-P_2(h)}{c}\in\Z[y]$ has $(O_{\deg{P_2}}(C),cq)$-coefficients since $|h|\leq |c|$. We make the change of variables $x\mapsto x+\f{P_2(cz+h)-P_2(h)}{c}$ to get that
\[
\bigg|\f{1}{N/c}\sum_{x\in\Z}\E_{z\in[M/|c|]}g_0'(x)g_1'(x+P_1'(z))\psi_2(P_2(cz+h))\cdots\psi_m(P_m(cz+h))\bigg|\gg\gamma,
\]
where $g_0'(x):=T_{-P_2(h)}(g_0\psi_2)(cx)$, $g_1'(x):=T_{P_1(h)-P_2(h)}g_1(cx)$, and $P_1'(z):=\f{P_1(cz+h)-P_1(h)}{c}$, which also has $(O_{\deg{P_1}}(C),cq)$-coefficients. By the assumption $|c_1|M^{\deg{P_1}}/N\leq C$, we can apply Lemma~\ref{lem4.2} $d:=\deg{P_1}$ times and then the Cauchy--Schwarz inequality once to deduce from the above that
\[
\E_{|a_1|,\dots,|a_d|<\gamma'M/c}\left|\E_{z\in[M/|c|]}e(Q(\ul{a},z))\right|^2\gg_{C,d}\gamma^{O_d(1)}
\]
whenever $\gamma'\ll_{C,d}\gamma^{O_d(1)}$, where
\[
Q(\ul{a},z):=\sum_{i=2}^m\alpha_i\left[\sum_{\omega\in\{0,1\}^d}(-1)^{|\omega|}P_i(c(z+\ul{a}\cdot\omega)-h)\right].
\]
Thus,
\begin{equation}\label{eq8.2}
|\E_{z\in[M/|c|]}e(Q(\ul{a},z))|\gg_{C,d}\gamma^{O_d(1)}
\end{equation}
for a $\gg_{C,d}\gamma^{O_d(1)}$ proportion of integers $|a_1|,\dots,|a_d|<\gamma'M/|c|$.

Note that the leading term of $Q(\ul{a},z)$ equals $\f{(\deg{P_m})!}{(\deg{P_m}-d)!}c^{\deg{P_m}}a_1\cdots a_d c_m \alpha_mz^{\deg{P_m}-d}$. By Lemma~\ref{lem7.1}, there thus exists a $t_0\ll_{C,\deg{P_m}}\gamma^{-O_{\deg{P_m}}(1)}$ such that for each $d$-tuple of integers $\ul{a}=(a_1,\dots,a_d)$ with $|a_i|<\gamma' M/c$ for which~\eqref{eq8.2} holds, we have 
\[
\|t_0 c^{\deg{P_m}}a_1\cdots a_d c_m \alpha_m\|\ll_{C,d}\gamma^{-O_{\deg{P_m}}(1)}/(M/c)^{\deg{P_m}-d}.
\]
Fixing $\gamma'\asymp_{C,d}\gamma^{-O_d(1)}$, the conclusion of the lemma follows by applying Lemma~\ref{lem7.2} $d$ times, once for each $a_i$ appearing in the product $c^{\deg{P_m}}a_1\cdots a_d c_m \alpha_m$.
\end{proof}

Next, we show that Lemma~\ref{lem3.9} in the general $\ell\geq 2$ case follows from Lemma~\ref{lem3.10} in the $\ell$ case. The overall strategy of the following proof is the same as the proof of Proposition~6.6 in~\cite{PelusePrendiville19}, though several small changes need to be made due to the greater generality of Lemma~\ref{lem3.9} and the use of different definitions of the $U^s$-norm in the two papers. We now briefly sketch the structure of the argument. The proof starts by writing the $U^s$-norm of the dual function $F_\ell$ as an average of $U^2$-norms of differenced versions of $F_\ell$ (that is, $\Delta'_{(h_i,h_i')_{i=1}^{s-2}}F_\ell$ in the following proof and $\Delta_{h_1,\dots,h_{s-2}}F_\ell$ in~\cite{PelusePrendiville19}). By the inverse theorem for the $U^2$-norm, it follows that, on average, the differenced versions of $F_\ell$ have large correlation with some character $x\mapsto e(\phi(\ul{h},\ul{h}')x)$ depending on $(\ul{h},\ul{h}')$. One then uses Lemma~\ref{lem3.10} and the pigeonhole principle (along with Lemma~\ref{lem7.3}) to show that the function $\phi(\ul{h},\ul{h}')$ must be very close to a function of the form $\sum_{i=1}^{s-2}\phi_i(\ul{h},\ul{h}')$ appearing in Lemma~\ref{lem7.5} for many differencing parameters $(\ul{h},\ul{h}')$. The conclusion of the lemma then follows from Lemma~\ref{lem7.5}.
\begin{proof}[Proof of Lemma~\ref{lem3.9} for $\ell$ assuming Lemma~\ref{lem3.10} for $\ell$]
Note that, by splitting $\Z$ up into progressions modulo $|c|$, we have
\[
\|F_\ell\|^{2^s}_{U^s_{c[\delta'M^{\deg{P_\ell}}]}([CN])}=\E_{u=0,\dots,|c|-1}\E_{\substack{h_1,\dots,h_{s-2}\in[\delta' M^{\deg{P_\ell}}]\\ h_1',\dots,h_{s-2}'\in[\delta' M^{\deg{P_\ell}}]}}\|\Delta'_{c(h_i,h_i')_{i=1}^{s-2}}(T_uF_\ell)(c\cdot)\|_{U^2_{[\delta'M^{\deg{P_\ell}}]}([CN/c])}^4.
\]
Thus, since $M^{\deg{P_\ell}}\asymp_CN/c$, Lemma~\ref{lem2.4} tells us that
\[
\E_{u=0,\dots,|c|-1}\E_{\substack{h_1,\dots,h_{s-2}\in[\delta' M^{\deg{P_\ell}}]\\ h_1',\dots,h_{s-2}'\in[\delta' M^{\deg{P_\ell}}]}}\left|\f{1}{N/c}\sum_{x\in\Z}\Delta'_{c(h_i,h_i')_{i=1}^{s-2}}(T_uF_\ell)(cx)e(c\phi_u(\ul{h},\ul{h}')x)\right|^2\gg_C(\delta\delta')^{O(1)}
\]
for some $\phi_u:[\delta'M^{\deg{P_\ell}}]^{2(s-2)}\to\T$ for each $u=0,\dots,|c|-1$. By the pigeonhole principle, there exists an $H\subset[\delta' M^{\deg{P_\ell}}]^{2(s-2)}$ with $|H|\gg_C(\delta\delta')^{O(1)}(\delta' M^{\deg{P_\ell}})^{2(s-2)}$ and $U\subset\{0,\dots,|c|-1\}$ with $|U|\gg_{C}(\delta\delta')^{O(1)}|c|$ such that
\[
\left|\f{1}{N/c}\sum_{x\in\Z}\Delta'_{c(h_i,h_i')_{i=1}^{s-2}}(T_uF_\ell)(cx)e(c\phi_u(\ul{h},\ul{h}')x)\right|^2\gg_C(\delta\delta')^{O(1)}
\]
for every $(\ul{h},\ul{h}')\in H$ and $u\in U$.

Next, we apply Lemma~\ref{lem7.4} with $L=N/|c|$, which, since $M^{\deg{P_\ell}}\gg_CN/|c|$, yields
\[
\E_{\ul{k}\in\square_{s-2}(H)}\left|\f{1}{N/c}\sum_{x\in\Z}G_{\ell,\ul{k}}(c x)e(c\psi_u(\ul{k})x)\right|^2\gg_{C,s}(\delta\delta')^{O_s(1)},
\]
where, as in Lemma~\ref{lem7.4}, we have
\[
G_{\ell,\ul{k}}(x):=\E_{y\in[M]}\Delta'_{c(k_{i}^{(2)},k_{i}^{(3)})_{i=1}^{s-2}}T_uf_0(x-P_\ell(y))\cdots\Delta'_{c(k_{i}^{(2)},k_{i}^{(3)})_{i=1}^{s-2}}T_uf_{\ell-1}(x+P_{\ell-1}(y)-P_\ell(y))
\]
and
\[
\psi_u(\ul{k}):=\sum_{\omega\in\{0,1\}^{s-2}}(-1)^{|\omega|}\phi_u(k_{1}^{(1)},\dots,k_{s}^{(1)},k_{1}^{(\omega_1+2)},\dots,k_{s}^{(\omega_s+2)}).
\]
By the pigeonhole principle again, for each $u\in U$ there exists a set of $3(s-2)$-tuples $H_u'\subset\square_{s-2}(H)$ with $|H'_u|\gg_{C,s}(\delta\delta')^{O_s(1)}(\delta'M^{\deg{P_\ell}})^{3(s-2)}$ such that
\[
\left|\f{1}{N/c}\sum_{x\in\Z}G_{\ell,\ul{k}}(c x)e(c\psi_u(\ul{k})x)\right|^2\gg_{C,s}(\delta\delta')^{O_s(1)}
\]
for every $\ul{k}\in H_u'$. By applying Lemma~\ref{lem3.10} for $\ell$ with $m=\ell$, for each $\ul{k}\in H_u'$ there thus exist $c_u'\ll_{C}|cc_\ell|^{O_{\deg{P_\ell}}(1)}$ and  $t_u\ll_{C,\deg{P_\ell},s}(\delta\delta')^{-O_{s,\deg{P_\ell}}(1)}$ such that
\[
  \|t_u c_u'c_\ell\psi_u(\ul{k})\|\ll_{C,\deg{P_\ell},s}\f{(\delta\delta')^{-O_{\deg{P_\ell},s}(1)}}{M^{\deg{P_\ell}}/c_u'}.
\]
By applying Lemma~\ref{lem7.3} with $D\asymp_{C,\deg{P_\ell},s}(\delta\delta')^{-O_{\deg{P_\ell},s}(1)}$, it follows that for each $\ul{k}\in H_u'$, there exist integers $a_u(\ul{k})\ll_{C,\deg{P_\ell},s}(\delta\delta')^{-O_{\deg{P_\ell},s}(1)}$ and $|m_u(\ul{k})|\ll_{C,\deg{P_\ell},s} (\delta\delta')^{-O_{\deg{P_\ell},s}(1)}$ and $|\theta_u(\ul{k})|\leq 1$ such that
\[
c_\ell\psi_u(\ul{k})=\f{a_u(\ul{k})}{t_uc'_u}+\f{m_u(\ul{k})}{(\delta\delta')^{-O_{\deg{P_\ell},s}(1)}M^{\deg{P_\ell}}}+\f{\theta_u(\ul{k})}{(\delta\delta')^{-O_{\deg{P_\ell},s}(1)}M^{\deg{P_\ell}}}.
\]
By the pigeonhole principle yet again, for each $u\in U$ there exists a subset $H''_u\subset H_u'$ of size $|H''_u|\gg_{C,\deg{P_\ell},s}(\delta\delta')^{O_{\deg{P_\ell},s}(1)}|H'_u|$ for which there are $a_u\ll_{C,\deg{P_\ell},s}(\delta\delta')^{-O_{\deg{P_\ell},s}(1)}$ and $|m_u|\ll_{C,\deg{P_\ell},s}(\delta\delta')^{-O_{\deg{P_\ell},s}(1)}$ such that for any $\ul{k}\in H''_u$, we have
\[
c_\ell\psi_u(\ul{k})=\f{a_u}{t_uc'_u}+\f{m_u}{(\delta\delta')^{-O_{\deg{P_\ell},s}(1)}M^{\deg{P_\ell}}}+\f{\theta_u(\ul{k})}{(\delta\delta')^{-O_{\deg{P_\ell},s}(1)}M^{\deg{P_\ell}}}.
\]
Set
\begin{align*}
  \phi_{u,1}(\ul{k}):=(-1)^s&\sum_{\substack{\ul{0}\neq\omega\in\{0,1\}^{s-2} \\ \omega_1=0}} (-1)^{|\omega|}\phi_u(k_{1}^{(1)},\dots,k_{s}^{(1)},k_{1}^{(\omega_1+2)},\dots,k_{s}^{(\omega_s+2)})\\
  &+\f{a_u}{t_uc'_uc_{\ell}}+\f{m_u}{(\delta\delta')^{-O_{\deg{P_\ell},s}(1)}c_{\ell}M^{\deg{P_\ell}}}
\end{align*}
and, for $i=2,\dots,s-2$, set
\[
\phi_{u,i}(\ul{k}):=(-1)^s\sum_{\substack{\ul{0}\neq\omega\in\{0,1\}^{s-2} \\ \omega_1=\dots=\omega_{i-1}=1 \\ \omega_i=0}}(-1)^{|\omega|}\phi_u(k_{1}^{(1)},\dots,k_{s}^{(1)},k_{1}^{(\omega_1+2)},\dots,k_s^{(\omega_s+2)}).
\]
Note that $\phi_{x,i}$ does not depend on on $k_{i}^{(3)}$ and
\[
\psi_u(\ul{k})=\sum_{i=1}^{s-2}\phi_{u,i}(\ul{k})+\f{\theta_u(\ul{k})}{(\delta\delta')^{-O_{\deg{P_\ell},s}(1)}c_\ell M^{\deg{P_\ell}}}.
\]
For any $\ul{k}\in H''_u$, we thus have
\[
\left|c\psi_u(\ul{k})-c\sum_{i=1}^{s-2}\phi_{u,i}(\ul{k})\right|\ll_C\f{1}{(\delta\delta')^{-O_{\deg{P_\ell},s}(1)}M^{\deg{P_\ell}}},
\]
because $c\asymp_Cc_\ell$

By the pigeonhole principle again, for each $u\in U$ there exist $h_{u,1}',\dots,h_{u,s-2}'\in[\delta' M^{\deg{P_\ell}}]$ such that the fiber
\[
H'''_u:=\{(h_1,\dots,h_{s-2},h_1'',\dots,h_{s-2}'')\in H:(\ul{h},\ul{h}',\ul{h}'')\in H_u''\}
\]
has size $\gg_{C,\deg{P_\ell},s}(\delta\delta')^{O_{\deg{P_\ell},s}(1)}(\delta' M^{\deg{P_\ell}})^{2(s-2)}$. Fixing such $h'_{u,1},\dots,h'_{u,s-2}$, it follows that
\[
\E_{(\ul{h},\ul{h}'')\in H'''_u}\left|\f{1}{N/c}\sum_{x\in\Z}\Delta'_{(h_i,h_i'')_{i=1}^{s-2}}T_uF_\ell(c x)e\left(c\sum_{i=1}^{s-2}\phi_{u,i}(\ul{h},\ul{h}_u',\ul{h}'')x\right)\right|^2\gg_{C,\deg{P_\ell},s}(\delta\delta')^{O_{\deg{P_\ell},s}(1)},
\]
by the assumption $N/|c|\ll_CM^{\deg{P_\ell}}$. By positivity, for each $u\in U$ we can extend the average over $H'''_u$ to an average over all of $[\delta' M^{\deg{P_\ell}}]^{2(s-2)}$ using our lower bound on $|H'''_u|$ to get that
\[
\E_{\ul{h},\ul{h}''\in [\delta'M^{\deg{P_\ell}}]^{s-2}}\left|\f{1}{N/c}\sum_{x\in\Z}\Delta'_{(h_i,h_i'')_{i=1}^{s-2}}T_uF_\ell(c x)e\left(c\sum_{i=1}^{s-2}\phi_{u,i}(\ul{h},\ul{h}_u',\ul{h}'')x\right)\right|^2
\]
is $\gg_{C,\deg{P_\ell},s}(\delta\delta')^{O_{\deg{P_\ell},s}(1)}$. Applying Lemma~\ref{lem7.5} for each $u\in U$ and using positivity again, we deduce that
\[
\E_{u=0,\dots,c-1}\|T_uF_\ell(c\cdot)\|^{2^{s-1}}_{U^{s-1}_{[\delta'M^{\deg{P_\ell}}]}([CN/c])}\gg_{C,\deg{P_\ell},s}(\delta\delta')^{O_{\deg{P_\ell},s}(1)},
\]
from which we conclude the lemma by expanding the definition of the Gowers box norm.
\end{proof}

Now we show that Lemma~\ref{lem3.10} in the general $\ell\geq 3$ case follows from Lemmas~\ref{lem3.9} and~\ref{lem3.10} in the $\ell-1$ case.
\begin{proof}[Proof of Lemma~\ref{lem3.10} for $\ell$ assuming Lemmas~\ref{lem3.9} and~\ref{lem3.10} for $\ell-1$]
As in the proof of the base case, we insert the definition of $F_\ell$ and split the sum over $y\in[M]$ up into progressions modulo $|c|$ by writing $y=cz+h$ for $h=0,\dots,|c|-1$, and use the pigeonhole principle to fix an $h$ such that
\begin{align*}
\bigg|\f{1}{N/c}\sum_{x\in\Z}\E_{z\in[M/|c|]}f_0(cx-P_{\ell}(cz+h))\cdots f_{\ell-1}(cx+P_{\ell-1}(cz+h)-P_\ell(cz+h)) &\\
\psi_\ell(cx)\psi_{\ell+1}(P_{\ell+1}(cz+h))\cdots\psi_m(P_m(cz+h))&\bigg|\gg\delta,
\end{align*}
and then make the change of variables $x\mapsto x+\f{P_\ell(cz+h)-P_\ell(h)}{c}$ to deduce that
\begin{equation}\label{eq8.3}
  \left|\Lambda^{N/c,M/c}_{P_1',\dots,P_m'}(f_0',\dots,f_{\ell-1}';\psi_{\ell},\dots,\psi_m)\right|\gg\delta,
\end{equation}
where
\[
f'_i(x):=\begin{cases}
T_{-P_\ell(h)}(f_0\psi_\ell)(cx) & i=0 \\
T_{P_i(h)-P_\ell(h)}f_i(cx) & i=1,\dots,m
\end{cases}
\]
and
\[
P_i'(z):=\begin{cases}
\f{P_i(cz+h)-P_i(h)}{c} & i=1,\dots,\ell-1 \\
P_i(cz+h)-P_i(h) & i=\ell,\dots,m
\end{cases}.
\]
Note, as it will be relevant later, that the leading coefficient $c_i'$ of $P_i'$ equals $c^{\deg{P_i}-1}c_i$ when $i=1,\dots,\ell-1$ and equals $c^{\deg{P_i}}c_i$ when $i=\ell,\dots,m$, and the polynomials $P_1',\dots,P_{\ell-1}'\in\Z[z]$ all have $(O_{\deg{P_{\ell-1}}}(C),qc)$-coefficients.

Set $M':=M/|c|$ and $N':=(M')^{\deg{P_{\ell-1}}}(q|c|)^{\deg{P_{\ell-1}}-1}$. With a view towards applying Corollary~\ref{cor3.8}, we rewrite the left-hand side of~\eqref{eq8.3} as
\begin{align*}
  \bigg|\E_{\substack{0\leq w<(N/|c|)/C'N' \\ x\in[C'N']}}\E_{z\in[M']}T_{C'N'w}f_0'(x)T_{C'N'w}f'_1(x+P_1'(z))\cdots T_{C'N'w}f_{\ell-1}'(x+P_{\ell-1}'(z))&\\
                                                                                                                                                  \psi_\ell(P_\ell'(z))\cdots\psi_m(P_m'(z))&\bigg|\\
\end{align*}
for $C'\asymp_{C,\deg{P_{\ell-1}}}1$ and use the fact that $\max_{z\in[M']}|P_i'(z)|\ll_{C,\deg{P_{\ell-1}}} N'$ for each $i=1,\dots,\ell-1$ (which is a consequence of each $P_i'$ having $(O_{\deg{P_{\ell-1}}}(C),cq)$-coefficients) and the pigeonhole principle to deduce, for suitable $C'$, that
\[
\left|\Lambda^{C'N',M'}_{P_1',\dots,P_m'}(f_0'',\dots,f_{\ell-1}'';\psi_{\ell},\dots,\psi_m)\right|\geq\delta,
\]
where $f_i'':=T_{C'N'w}f_i'\cdot 1_{[C'N']}$ for some integer $0\leq w<(N/|c|)/C'N'$.

Now, since $(q|c|)^{\deg{P_{\ell-1}}-1}(M')^{\deg{P_{\ell-1}}}= N'$ and  $P_1',\dots,P_{\ell-1}'\in\Z[z]$ have $(O_{\deg{P_{\ell-1}}}(C),qc)$-coefficients, we may apply Corollary~\ref{cor3.8} to get that
\[
\|F_{\ell-1}'\|_{U^s_{(\deg{P_{\ell-1}})!c_{\ell-1}'[\delta'(M')^{\deg{P_{\ell-1}}}]}([O_{C,\deg{P_{\ell-1}}}(1)N'])}\gg_{C,\deg{P_{\ell-1}}}\delta^{O_{\deg{P_{\ell-1}}}(1)}
\]
for any $\delta'\ll_{C,\deg{P_{\ell-1}}}\delta^{O_{\deg{P_{\ell-1}}}(1)}$, where $s\ll_{\deg{P_{\ell-1}}}1$ and
\[
F_{\ell-1}'(x):=\E_{z\in[M']}f_0''(x-P_{\ell-1}'(z))\cdots f_{\ell-2}''(x+P_{\ell-2}'(z)-P_{\ell-1}'(z))\psi_{\ell}(P_\ell'(z))\cdots\psi_{m}(P_m'(z)).
\]
Fixing $\delta'\asymp_{C,\deg{P_{\ell-1}}}\delta^{O_{\deg{P_{\ell-1}}}(1)}$, it thus follows from repeated applications of Lemma~\ref{lem3.9} in the $\ell-1$ case that
\[
\|F_{\ell-1}'\|_{U^2_{(\deg{P_{\ell-1}})!c_{\ell-1}'[\delta'(M')^{\deg{P_{\ell-1}}}]}([O_{C,\deg{P_{\ell-1}}}(1)N'])}\gg_{C,\deg{P_{\ell-1}}}\delta^{O_{\deg{P_{\ell-1}}}(1)}
\]
Set $c':=(\deg{P_{\ell-1}})!c_{\ell-1}'$. By applying Lemma~\ref{lem2.4} in the same manner as in the previous proof and using the pigeonhole principle, we deduce that there exists a $u\in[c']$ such that
\[
\left|\f{1}{N'/c'}\sum_{x\in\Z}T_uF_{\ell-1}'(c'x)\psi_{\ell-1}(c'x)\right|\gg_{C,\deg{P_{\ell-1}}}\delta^{O_{\deg{P_{\ell-1}}}(1)}
\]
for some character $\psi_{\ell-1}:\Z\to S'$. We now apply Lemma~\ref{lem3.10} for $\ell-1$ to deduce that there exists a $c''\ll_C|c'c_\ell c_m|^{O_{\deg{P_m}}(1)}\ll_C|cc_m|^{O_{\deg{P_m}}(1)}$ and $t\ll_{C,\deg{P_m}}\delta^{-O_{\deg{P_m}}(1)}$ such that
\[
\|tc'' c^{\deg{P_m}}c_m\alpha_m\|\ll_{C,\deg{P_m}} \f{\delta^{-O_{\deg{P_m}}(1)}}{(M/c)^{\deg{P_m}}/c''},
\]
since the leading coefficient of $P_m'$ is $c^{\deg{P_m}}c_m$. This gives the conclusion of the lemma.
\end{proof}

Since we have shown that Lemma~\ref{lem3.10} holds in the $\ell=2$ case, Lemma~\ref{lem3.10} in the $\ell$ case implies Lemma~\ref{lem3.9} in the $\ell$ case, and Lemmas~\ref{lem3.9} and~\ref{lem3.10} in the $\ell-1$ case together imply Lemma~\ref{lem3.10} in the $\ell$ case, it now follows by induction that Lemmas~\ref{lem3.9} and~\ref{lem3.10} hold in general.

\section{Local $U^1$-control}\label{sec9}
As was mentioned in Section~\ref{sec3}, Theorem~\ref{thm3.3} will be proved using a combination of Corollary~\ref{cor3.8}, Lemma~\ref{lem3.9}, and Lemma~\ref{lem2.4}. For the sake of convenience, before proving Theorem~\ref{thm3.3} we first prove Lemma~\ref{lem3.11}, which gives the result of applying Corollary~\ref{cor3.8} once, Lemma~\ref{lem3.9} as many times as necessary, and then Lemma~\ref{lem2.4} once.

\begin{proof}[Proof of Lemma~\ref{lem3.11}]
We first apply Corollary~\ref{cor3.8}, which tells us that
\[
\|F_\ell\|_{U^s_{c'[\delta'M^{\deg{P_\ell}}]}([O_{\deg{P_\ell}}(CN)])}\gg_{C,\deg{P_\ell}}\delta^{O_{\deg{P_\ell}}(1)}
\]
for some $s\ll_{\deg{P_\ell}}1$ whenever $\delta'\ll_{C,\deg{P_\ell}}\delta^{O_{\deg{P_\ell}}(1)}$ and $N\gg_{\deg{P_\ell}}(q/\delta\delta')^{O_{\deg{P_\ell}}(1)}$. Fixing $\delta'\asymp_{C,\deg{P_\ell}}\delta^{O_{\deg{P_\ell}}(1)}$ and then applying Lemma~\ref{lem3.9} repeatedly (which we can do because $(\deg{P_\ell})!/C\leq |c'|M^{\deg{P_\ell}}/N\leq (\deg{P_\ell})!C^2$) thus yields
\[
\|F_\ell\|_{U^2_{c'[\delta'M^{\deg{P_\ell}}]}([O_{\deg{P_\ell}}(CN)])}\gg_{C,\deg{P_\ell}}\delta^{O_{\deg{P_\ell}}(1)}.
\]
We now expand the definition of the Gowers box norm and split the sum over $\Z$ up into progressions modulo $|c'|$ as in the proof of Lemmas~\ref{lem3.9} and~\ref{lem3.10} to write the above as
\[
\E_{u=0,\dots,|c'|-1}\|T_{-u}F_\ell(c'\cdot)\|_{U^2_{[\delta'M^{\deg{P_\ell}}]}([O_{\deg{P_\ell}}(CN/|c'|)])}\gg_{C,\deg{P_\ell}}\delta^{O_{\deg{P_\ell}}(1)},
\]
so that, by Lemma~\ref{lem2.4} and the inequality $(\deg{P_\ell})!/C\leq |c'|M^{\deg{P_\ell}}/N\leq (\deg{P_\ell})!C^2$ again, we have that
\[
\E_{u=0,\dots,|c'|-1}\left|\f{1}{N/c'}\sum_{x\in\Z}T_{-u}F_\ell(c'x)\psi_{\ell,u}(c'x)\right|\gg_{C,\deg{P_\ell}}\delta^{O_{\deg{P_\ell}}(1)}
\]
for some characters $\psi_{\ell,u}:\Z\to S^1$. Expanding the definition of $F_\ell$, the above inequality says that
\begin{align*}
  \E_{u=0,\dots,|c'|-1}\bigg|\f{1}{N/c'}\sum_{x\in\Z}\E_{y\in[M]}T_{-u}f_0(c'x-P_\ell(y))\cdots T_{-u}f_{\ell-1}(c'x+P_{\ell-1}(y)-P_\ell(y))&\\
  \psi_{\ell,u}(c'x)\psi_{\ell+1}(P_{\ell+1}(y))\cdots\psi_{m}(P_m(y))&\bigg|
\end{align*}
is $\gg_{C,\deg{P_\ell}}\delta^{O_{\deg{P_\ell}}(1)}$.

Next, as in the proofs of Lemmas~\ref{lem8.1} and~\ref{lem3.10}, we split the average over $y\in[M]$ above up into congruence classes modulo $|c'|$ by setting $y=c'z+h$ for $h=0,\dots,|c'|-1$ and make the change of variables $x\mapsto x+\f{P_\ell(c'z+h)-P_\ell(h)}{c'}$ to get, assuming $N\gg_{C,\deg{P_\ell}}(q/\delta)^{O_{\deg{P_\ell}}(1)}$, that
\[
  \E_{u,h=0,\dots,|c'|-1}\left|\Lambda_{P_1^{h},\dots,P_m^h}^{N/|c'|,M'}(f_0^{u,h},\dots,f_{\ell-1}^{u,h};\psi_{\ell,u},\psi_{\ell+1},\dots,\psi_m)\right|\gg_{C,\deg{P_\ell}}\delta^{O_{\deg{P_\ell}}(1)},
\]
where
\[
f_{i}^{u,h}(x):=\begin{cases}
T_{-P_\ell(h)}T_{-u}(f_0\psi_{\ell,u})(c'x) & i=0 \\
T_{P_i(h)-P_\ell(h)}T_{-u}f_i(c'x) & i=1,\dots,\ell-1
\end{cases}.
\]

To conclude, we argue as in the proof of Lemma~\ref{lem3.10}, using the fact that $\max_{z\in[M']}|P_i^h(z)|\leq C'N'/2$ for all $|h|\leq |c'|$ and $i=1,\dots,\ell-1$ whenever $N\gg_{C,\deg{P_\ell}}(q/\delta)^{O_{\deg{P_\ell}}(1)}$ to split the sum over $x\in\Z$ in $\Lambda_{P_1^{h},\dots,P_m^h}^{N/|c'|,M'}(f_0^{u,h},\dots,f_{\ell-1}^{u,h};\psi_{\ell,u},\psi_{\ell+1},\dots,\psi_m)$ up into intervals of length $C'N'$ and then applying the triangle inequality to get
\[
\E_{\substack{u,h=0,\dots,|c'|-1 \\ 0\leq w<(N/|c'|)/C'N'}}\left|\Lambda_{P_1^h,\dots,P_m^h}^{C'N',M'}(f_0^{u,h,w},\dots,f_{\ell-1}^{u,h,w};\psi_{\ell,u},\psi_{\ell+1},\dots,\psi_m)\right|\gg_{C,\deg{P_\ell}}\delta^{O_{\deg{P_\ell}}(1)}.
\]
\end{proof}

Now we can prove Theorem~\ref{thm3.3}.
\begin{proof}[Proof of Theorem~\ref{thm3.3}]
  We apply Lemma~\ref{lem3.11} $m-1$ times to get that
  \begin{equation}\label{eq9.1}
    \E_{\substack{u_i,h_i=0,\dots,|c_i|-1\\ 0\leq w_i<(C_{i+1}N_{i+1}/|c_i|)/C_iN_i \\ i=2,\dots,m}}\left|\Lambda_{P_1^{\ul{h}},\dots,P_m^{\ul{h}}}^{C_2N_2,M_2}(f_0^{\ul{u},\ul{h},\ul{w}},f_1^{\ul{u},\ul{h},\ul{w}};\psi_2^{\ul{u},\ul{h},\ul{w}},\dots,\psi_m^{\ul{u},\ul{h},\ul{w}})\right|\gg_{C,\deg{P_m}}\delta^{O_{\deg{P_m}}(1)},
  \end{equation}
  where $C_{m+1}=1$, $N_{m+1}=N$, $c_i=\tilde{c}_iq^{b_i}$ for $\tilde{c}_i\asymp_{C,\deg{P_m}}1$ and $b_i\ll_{\deg{P_m}}1$, $M_i:=M/\prod_{j=i}^m|c_i|$, $C_i\asymp_{C,\deg{P_m}}1$, and $N_i:=M_i^{\deg{P_{i-1}}}(q|c_i\cdots c_m|)^{\deg{P_{i-1}}-1}$ for each $i=2,\dots,m$, $f_0^{\ul{u},\ul{h},\ul{w}}$ is $1$-bounded and $f_1^{\ul{u},\ul{h},\ul{w}}(x)$ equals $1_{[C_2N_2]}(x)$ times
  \[
T_{\sum_{i=2}^m (c_{i+1}\cdots c_m)[w_ic_iC_iN_i-u_i+[P_1^{h_m,\dots,h_{i+1}}(h_{i})-P_i^{h_m,\dots,h_{i+1}}(h_{i})]]}f_1(c_2\cdots c_mx)
\]
for each $\ul{u},\ul{h}\in\prod_{i=2}^m\{0,\dots,|c_i|-1\}$ and $\ul{w}\in\prod_{i=2}^m([0,(C_{i+1}N_{i+1}/|c_i|)/C_iN_i)\cap\Z)$, where $P^{h_m,\dots,h_{i+1}}$ denotes the polynomial $((P^{h_m})^{h_{m-1}})\dots)^{h_{i+1}}$ using the notation from Lemma~\ref{lem3.11}, each $P_{i}^{\ul{h}}$ is a polynomial of degree $\deg{P_i}$ whose coefficients have magnitude $\ll_{C,\deg{P_m}}q^{O_{\deg{P_m}}(1)}$ and whose leading coefficient is independent of $h$, and $P_1^{\ul{h}}$ has leading coefficient of the form $C'(qc_2\cdots c_m)^{\deg{P_1}-1}$ for some $C'\ll_{C}1$ and satisfies $\max_{y\in[M_2]}|P_1^{\ul{h}}(y)|\ll_{C,\deg{P_m}}N_2$.

For each character $\psi_i^{\ul{u},\ul{h},\ul{w}}$, let $\beta_i^{\ul{u},\ul{h},\ul{w}}\in\T$ be such that $\psi_i^{\ul{u},\ul{h},\ul{w}}(x)=e(\beta_i^{\ul{u},\ul{h},\ul{w}}x)$. Next, we argue as in the proof of Lemma~\ref{lem8.1} and apply Lemma~\ref{lem4.2} $d:=\deg{P_1}$ times and the Cauchy--Schwarz inequality once to get that
\[
\E_{\substack{u_i,h_i=0,\dots,|c_i| -1\\ 0\leq w_i<(C_{i+1}N_{i+1}/|c_i|)/C_iN_i \\ i=2,\dots,m}}\E_{|a_1|,\dots,|a_d|<\delta'M_2}\left|\E_{y\in[M_2]}e(Q^{\ul{u},\ul{h},\ul{w}}(\ul{a},y))\right|^2\gg_{C,\deg{P_m}}\delta^{O_{\deg{P_m}}(1)}
\]
whenever $\delta'\ll_{C,\deg{P_m}}\delta^{O_{\deg{P_m}}(1)}$, where
\[
Q^{\ul{u},\ul{h},\ul{w}}(\ul{a},y):=\sum_{i=2}^m\beta_i^{\ul{u},\ul{h},\ul{w}}\left[\sum_{\omega\in\{0,1\}^d}(-1)^{|\omega|}P_i^{\ul{h}}(y+\ul{a}\cdot\omega)\right].
\]
As in the proof of Lemma~\ref{lem8.1}, we have that $\left|\E_{y\in[M_2]}e(Q^{\ul{u},\ul{h},\ul{w}}(\ul{a},y))\right|\gg_{C,\deg{P_m}}\delta^{O_{\deg{P_m}}(1)}$ for a $\gg_{C,\deg{P_m}}\delta^{O_{\deg{P_m}}(1)}$ proportion of tuples $\ul{u},\ul{h},$ and $\ul{w}$ and integers $|a_1|,\dots,|a_d|<\delta'M_2$.

Now set $d':=\deg{P_m}-\deg{P_1}$ and write
\[
Q^{\ul{u},\ul{h},\ul{w}}(\ul{a},y)=B_{d'}^{\ul{u},\ul{h},\ul{w}}(\ul{a})y^{d'}+\dots+B_1^{\ul{u},\ul{h},\ul{w}}(\ul{a})y+B_0^{\ul{u},\ul{h},\ul{w}}(\ul{a}),
\]
so that by Lemma~\ref{lem7.1} there exists a $t\ll_{C,\deg{P_m}}\delta^{-O_{\deg{P_m}}(1)}$ such that for a $\gg_{C,\deg{P_m}}\delta^{O_{\deg{P_m}}(1)}$ proportion of $\ul{a},\ul{u},\ul{h},$ and $\ul{w}$, we have $\|tB_i^{\ul{u},\ul{h},\ul{w}}(\ul{a})\|\ll_{C,\deg{P_m}}\delta^{-O_{\deg{P_m}}(1)}/M_2^i$ for $i=1,\dots,d'$. By expanding each $B_i^{\ul{u},\ul{h},\ul{w}}(\ul{a})$ in terms of $a_1,\dots,a_d$, it then follows from repeated applications of Lemma~\ref{lem7.2} and the triangle inequality that, if $\delta'\asymp_{C,\deg{P_m}}\delta^{O_{\deg{P_m}}(1)}$ is fixed suitably small, there must exist $t'\ll_{C,\deg{P_m}}\delta^{-O_{\deg{P_m}}(1)}$ and $b_i\ll_{\deg{P_m}}1$ such that $\|t'q^{b_i}\beta_i^{\ul{u},\ul{h},\ul{w}}\|\ll_{C,\deg{P_m}}\delta^{-O_{\deg{P_m}}(1)}/M_2^{\deg{P_i}}$ for all $i=2,\dots,m$.

Thus, by splitting $y\in[M_2]$ up into progressions of length $M_2'\asymp_{C,\deg{P_m}}(\delta/q)^{O_{\deg{P_m}}(1)}M_2$ modulo $t'q^{s}$ for some $s\ll_{\deg{P_m}}1$, it follows from~\eqref{eq9.1} that
\[
  \E_{\substack{u_i,h_i=0,\dots,|c_i|-1\\ 0\leq w_i<(C_{i+1}N_{i+1}/|c_i|)/C_iN_i \\ i=2,\dots,m \\k_{\ul{u},\ul{h},\ul{w}}\in[M_2/M_2'] \\ k_{\ul{u},\ul{h},\ul{w}}'\in[t'q^{s}]}}\left|\f{1}{C_2N_2}\sum_{x}\E_{z\in[M_2']}f_0^{\ul{u},\ul{h},\ul{w}}(x)f_1^{\ul{u},\ul{h},\ul{w}}(x+P_1^{\ul{h}}(t'q^{s}(z-M_2'k_{\ul{u},\ul{h},\ul{w}}')-k_{\ul{u},\ul{h},\ul{w}}))\right|
\]
is $\gg_{C,\deg{P_m}}\delta^{O_{\deg{P_m}}(1)}$. Applying Lemma~\ref{lem4.2} $d$ more times, we get from the above that
\[
  \E_{\substack{u_i,h_i=0,\dots,|c_i|-1\\ 0\leq w_i<(C_{i+1}N_{i+1}/|c_i|)/C_iN_i \\ i=2,\dots,m}}\f{1}{C_2N_2}\sum_{x}\E_{|a_1|,\dots,|a_d|<\delta''M_2'}f_1^{\ul{u},\ul{h},\ul{w}}(x)f_1^{\ul{u},\ul{h},\ul{w}}(x+C'd!(t'q^{s})^{d}(qc_2\cdots c_m)a_1\cdots a_d)
\]
is $\gg_{C,\deg{P_m}}\delta^{O_{\deg{P_m}}(1)}$ whenever $\delta''\ll_{C,\deg{P_m}}\delta^{O_{\deg{P_m}}(1)}$. Note that this can be written as
\[
  \E_{\substack{u_i,h_i=0,\dots,|c_i|-1\\ 0\leq w_i<(C_{i+1}N_{i+1}/|c_i|)/C_iN_i \\ i=2,\dots,m}}\f{1}{C_2 N}\sum_x\sum_{|y|\leq (\delta'' M_2')^d}f_1^{\ul{u},\ul{h},\ul{w}}(x)f_1^{\ul{u},\ul{h},\ul{w}}(x+C'd!(t'q^{s})^{d}(qc_2\cdots c_m)y)G(y),
\]
where $G(y):=\E_{|a_1|,\dots,|a_d|<\delta'' M_2'}1_{y=a_1\cdots a_d}$. Inserting the $\int_0^1\widehat{G}(\xi)e(\xi y)d\xi$ for $G(y)$ above, bounding the contribution of minor arcs using Lemma~\ref{lem7.1}, pigeonholing in the major arcs, and fixing $\delta''\ll_{C,\deg{P_m}}\delta^{O_{\deg{P_m}}(1)}$ sufficiantly small, we get that there exists a $t''\ll (\delta\delta'')^{-O_d(1)}$ and $0<a\leq t''$ relatively prime to $t''$ such that
\[
  \E_{\substack{u_i,h_i=0,\dots,|c_i|-1\\ 0\leq w_i<(C_{i+1}N_{i+1}/|c_i|)/C_iN_i \\ i=2,\dots,m}}\left|\f{1}{C_2 N}\sum_x\E_{y<(\delta'' M_2')^d}f_1^{\ul{u},\ul{h},\ul{w}}(x)f_1^{\ul{u},\ul{h},\ul{w}}(x+C'd!(t'q^{s})^{d}(qc_2\cdots c_m)y)e\left(\f{ay}{t''}\right)\right|
\]
is $\gg_{C,\deg{P_m}}\delta^{O_{\deg{P_m}}(1)}$. We now split the sum over $y<(\delta'M_2')^d$ into arithmetic progressions modulo $t''$ of length $M_2'':=\lfloor(\delta''M_2')^d/t''\rfloor$ and apply Lemma~\ref{lem4.2} once more and use that $f_1^{\ul{u},\ul{h},\ul{w}}$ is $1$-bounded to deduce that
\[
  \E_{\substack{u_i,h_i=0,\dots,|c_i|-1\\ 0\leq w_i<(C_{i+1}N_{i+1}/|c_i|)/C_iN_i \\ i=2,\dots,m}}\left|\f{1}{C_2 N}\sum_x\E_{z\in[M_2'']}f_1^{\ul{u},\ul{h},\ul{w}}(x+C'd!(t'q^{s})^{d}t''(qc_2\cdots c_m)z)\right|\gg_{C,\deg{P_m}}\delta^{O_{\deg{P_m}}(1)}.
\]
Set $Q(z):=C'd!(t'q^{s})^{d}t''(qc_2\cdots c_m)z$ for ease of notation throughout the remainder of the argument.

To complete the proof of the theorem, it remains to unravel the definition of $f_1^{\ul{u},\ul{h},\ul{w}}$. First, we apply the pigeonhole principle to fix an $\ul{h}\in\prod_{i=2}^m\{0,\dots,|c_i|-1\}$ such that
\[
  \E_{\substack{u_i=0,\dots,|c_i|-1\\ 0\leq w_i<(C_{i+1}N_{i+1}/|c_i|)/C_iN_i \\ i=2,\dots,m}}\E_{x\in[C_2N_2]}\left|\E_{z\in[M_2'']}f_1^{\ul{u},\ul{h},\ul{w}}(x+Q(z))\right|\gg_{C,\deg{P_m}}\delta^{O_{\deg{P_m}}(1)}.
\]
For some $r_{\ul{h}}\ll_{C,\deg{P_m}}q^{O_{\deg{P_m}}(1)}$, the left-hand side of the above can thus be written as
\[
  \E_{\substack{x\in[C_2N_2]\\u_i=0,\dots,|c_i|-1\\ 0\leq w_i<(C_{i+1}N_{i+1}/|c_i|)/C_iN_i \\ i=2,\dots,m}}\left|\E_{z\in[M_2'']}T_{r_{\ul{h}}+\sum_{i=2}^m (c_{i+1}\cdots c_m)[w_iC_iN_i-u_i]}f_1(c_2\cdots c_m(x+Q(z)))\right|.
\]
Since, as $x$, $u_i$, and $w_i$ for each $i=2,\dots,m$ range over $[C_2N_2]$, $\{0,\dots,|c_i|-1\}$, and $[0,(C_{i+1}N_{i+1}/|c_i|)/C_iN_i)\cap\Z$, respectively, the quantity
\[
c_2\cdots c_m x+\sum_{i=2}^m(c_{i+1}\cdots c_m)[w_ic_iC_iN_i-u_i]
\]
ranges over $\ll N$ distinct integers lying within the interval $[1,N+O_m(|c_2\cdots c_m|C_mN_m)]$, and $N_m\ll_{C,\deg{P_m}}qN^{1-\ve}$ for some $0<\ve<1$ satisfying $\ve\gg_{\deg{P_m}}1$, we have that
\[
\f{1}{N}\sum_{x\in\Z}\left|\E_{z\in[M_2'']}f_1(x+c_2\cdots c_mQ(z)+r_{\ul{h}})\right|\gg_{\deg{P_m},C}\delta^{O_{\deg{P_m}}(1)},
\]
provided $N\gg_{C,\deg{P_m}}(q/\delta)^{O_{\deg{P_m}}(1)}$. We conclude by making the change of variables $x\mapsto x-r_{\ul{h}}$ and noting that any progression of the form $x-a[L]$ with $a>0$ can be written as $x-a(L+1)+a[L]$.
\end{proof}

\section{Density increment}\label{sec10}
In this section, we prove Theorem~\ref{thm3.2}, which we then use to finally prove Theorem~\ref{thm1.1}.
\begin{proof}[Proof of Theorem~\ref{thm3.2}]
Set $f_A:=1_A-\alpha 1_{[N]}$ and $M:= (N/q^{\deg{P_m}-1})^{1/\deg{P_m}}$. Note that $\Lambda_{P_1,\dots,P_m}^{N,M}(1_A)=0$ since $A$ contains only trivial progressions. By the multilinearity of $\Lambda_{P_1,\dots,P_m}^{N,M}$ and the identity $1_{A}=f_A+\alpha 1_{[N]}$, we have that $\Lambda_{P_1,\dots,P_m}^{N,M}(1_A)$ also equals
\[
\Lambda_{P_1,\dots,P_m}^{N,M}(1_A,f_A,1_A,\dots,1_A)+\alpha\Lambda_{P_1,\dots,P_m}^{N,M}(1_A,1_{[N]},f_A,1_A,\dots,1_A)+\cdots+\alpha^{m+1}\Lambda_{P_1,\dots,P_m}^{N,M}(1_{[N]}).
\]
Since $\Lambda_{P_1,\dots,P_m}^{N,M}(1_{[N]})\gg_{C,\deg{P_m}} 1$, we must have that
\[
\left|\Lambda_{P_i,\dots,P_m}^{N,M}(1_A,f_A,1_A,\dots,1_A)\right|\gg_{C,\deg{P_m}}\alpha^{O_{m}(1)}
\]
for some $i=1,\dots,m$. Theorem~\ref{thm3.3} then tells us that there exists a $q'\ll_{C,\deg{P_m}}\alpha^{-O_{\deg{P_m}}(1)}$, $b\ll_{\deg{P_m}}1$, and an $N'$ satisfying $M\geq N'\gg_{C,\deg{P_m}} M(\alpha/q)^{O_{\deg{P_m}}(1)}$ such that
\[
\f{1}{N}\sum_{x\in\Z}\left|\E_{y\in[N']}f_A(x+q'q^{b}y)\right|\gg_{C,\deg{P_m}}\alpha^{O_{\deg{P_m}}(1)},
\]
provided that $N\gg_{C,\deg{P_m}}(q/\alpha)^{O_{\deg{P_m}}(1)}$.

Note that $f_A$ has mean zero, so $\f{1}{N}\sum_{x\in\Z}\E_{y\in[N']}f_A(x+q'q^{b}y)=0$, which we can add to both sides of the above to get that
\[
\f{1}{N}\sum_{x\in\Z}\max\left(0,\E_{y\in[N']}f_A(x+q'q^{b}y)\right)\gg_{C,\deg{P_m}}\alpha^{O_{\deg{P_m}}(1)}.
\]
The total contribution to the above coming from $x\in\Z$ such that $x+q'q^{b}[N']\not\subset[N]$ is $\ll q'q^{O_{\deg{P_m}}(1)}N^{-1+1/\deg{P_m}}$, so that as long as $N\gg_{C,\deg{P_m}}(q/\alpha)^{O_{\deg{P_m}}(1)}$, there exists an $a\in[N]$ such that $a+q'q^{b}[N']\subset[N]$ and 
\[
\E_{y\in[N']}1_A(a+q'q^{b}y)\geq\alpha+\Omega_{C,\deg{P_m}}(\alpha^{O_{\deg{P_m}}(1)}),
\]
which means that we have the desired density increment.
\end{proof}
\begin{proof}[Proof of Theorem~\ref{thm1.1}]

Suppose that $A\subset[N]$ has density $\alpha$ and contains no nontrivial progressions of the form $x,x+P_1(y),\dots,x+P_m(y)$. Set $A_0=A$, $N_0=N$, $\alpha_0=\alpha$, and $q_0=1$. By applying Theorem~\ref{thm3.2} repeatedly, we get a sequence of $A_i$'s, $N_i$'s, $\alpha_i$'s, and $q_i$'s such that
\begin{enumerate}
\item $A_i\subset[N_i]$ with $\alpha_i=|A_i|/N_i$ and $\alpha_i\geq\alpha_{i-1}+\Omega_{P_1,\dots,P_m}(\alpha_{i-1}^{O_{P_1,\dots,P_m}(1)})$,
\item $N_i\gg_{P_1,\dots,P_m} (\alpha_{i-1}/(q_0\cdots q_{i-1}))^{O_{P_1,\dots,P_m}(1)}N_{i-1}^{1/\deg{P_m}}$,
\item $q_i\ll_{P_1,\dots,P_m}(q_0\cdots q_{i-1}/\alpha_{i-1})^{O_{P_1,\dots,P_m}(1)}$, and
\item $A_i$ contains no nontrivial progressions of the form
\[
x,x+P_1^{(q_0\cdots q_i)}(y),\dots,x+P_m^{(q_0\cdots q_i)}(y),
\]
\end{enumerate}
provided that $N_{i-1}\gg_{P_1,\dots,P_m}(q_0\cdots q_{i-1}/\alpha)^{O_{P_1,\dots,P_m}(1)}$.

Since no set can have density greater than $1$, the bound $N_i\gg_{P_1,\dots,P_m}(q_0\cdots q_i/\alpha)^{O_{P_1,\dots,P_m}(1)}$ must fail to hold for some $i\ll_{P_1,\dots,P_m}\alpha^{-O_{P_1,\dots,P_m}(1)}$. Thus,
\[
  N_i\ll_{P_1,\dots,P_m}\left(\f{q_0\cdots q_i}{\alpha}\right)^{O_{P_1,\dots,P_m}(1)}\ll_{P_1,\dots,P_m}\alpha^{-O_{P_1,\dots,P_m}(\sigma_1^i)}
\]
for some $0<\sigma_1\ll_{P_1,\dots,P_m} 1$ by the upper bound on the $q_i$'s. On the other hand, we also have that $N_i\gg_{P_1,\dots,P_m}\alpha^{O_{P_1,\dots,P_m}(\sigma_2^i)}N^{1/(\deg{P_m})^i}$ for some $0<\sigma_2\ll_{P_1,\dots,P_m}1$, again by the upper bound on the $q_i$'s. Comparing the upper and lower bounds for $N_i$ thus gives $N\ll_{P_1,\dots,P_m}\alpha^{-O_{P_1,\dots,P_m}(\sigma^i)}$ for some $\sigma\ll_{P_1,\dots,P_m}1$. Since $i\ll_{P_1,\dots,P_m}\alpha^{-O_{P_1,\dots,P_m}(1)}$, we get that $N\ll_{P_1,\dots,P_m}\alpha^{-O_{P_1,\dots,P_m}(\sigma^{O_{P_1,\dots,P_m}(\alpha^{-O_{P_1,\dots,P_m}(1)})})}$, from which the conclusion of the theorem follows.
\end{proof}

\bibliographystyle{plain}
\bibliography{bib}

\begin{thebibliography}{10}

\bibitem{BalogPelikanPintzSzemeredi94}
A.~Balog, J.~Pelik\'an, J.~Pintz, and E.~Szemer\'edi.
\newblock Difference sets without {$\kappa$}th powers.
\newblock {\em Acta Math. Hungar.}, 65(2):165--187, 1994.

\bibitem{BergelsonLeibman96}
V.~Bergelson and A.~Leibman.
\newblock Polynomial extensions of van der {W}aerden's and {S}zemer\'edi's
  theorems.
\newblock {\em J. Amer. Math. Soc.}, 9(3):725--753, 1996.

\bibitem{Bloom16}
T.~F. Bloom.
\newblock A quantitative improvement for {R}oth's theorem on arithmetic
  progressions.
\newblock {\em J. Lond. Math. Soc. (2)}, 93(3):643--663, 2016.

\bibitem{BourgainChang17}
J.~Bourgain and M.-C. Chang.
\newblock Nonlinear {R}oth type theorems in finite fields.
\newblock {\em Israel J. Math.}, Jul 2017.

\bibitem{DongLiSawin17}
D.~Dong, X.~Li, and W.~Sawin.
\newblock Improved estimates for polynomial roth type theorems in finite
  fields.
\newblock {\em preprint}, 2017.
\newblock arXiv:1709.00080.

\bibitem{Gowers98}
W.~T. Gowers.
\newblock A new proof of {S}zemer\'edi's theorem for arithmetic progressions of
  length four.
\newblock {\em Geom. Funct. Anal.}, 8(3):529--551, 1998.

\bibitem{Gowers01S}
W.~T. Gowers.
\newblock Arithmetic progressions in sparse sets.
\newblock In {\em Current developments in mathematics, 2000}, pages 149--196.
  Int. Press, Somerville, MA, 2001.

\bibitem{Gowers01}
W.~T. Gowers.
\newblock A new proof of {S}zemer\'edi's theorem.
\newblock {\em Geom. Funct. Anal.}, 11(3):465--588, 2001.

\bibitem{GreenTao10}
B.~Green and T.~Tao.
\newblock Linear equations in primes.
\newblock {\em Ann. of Math. (2)}, 171(3):1753--1850, 2010.

\bibitem{GreenTao17}
B.~Green and T.~Tao.
\newblock New bounds for {S}zemer\'{e}di's theorem, {III}: a polylogarithmic
  bound for {$r_4(N)$}.
\newblock {\em Mathematika}, 63(3):944--1040, 2017.

\bibitem{Lucier06}
J.~Lucier.
\newblock Intersective sets given by a polynomial.
\newblock {\em Acta Arith.}, 123(1):57--95, 2006.

\bibitem{Montgomery94}
H.~L. Montgomery.
\newblock {\em Ten lectures on the interface between analytic number theory and
  harmonic analysis}, volume~84 of {\em CBMS Regional Conference Series in
  Mathematics}.
\newblock Amer. Math. Soc., Providence, RI, 1994.

\bibitem{Peluse18}
S.~Peluse.
\newblock Three-term polynomial progressions in subsets of finite fields.
\newblock {\em Israel J. Math.}, 228(1):379--405, 2018.

\bibitem{Peluse19}
S.~Peluse.
\newblock On the polynomial {S}zemer\'{e}di theorem in finite fields.
\newblock {\em Duke Math. J.}, 168(5):749--774, 2019.

\bibitem{PelusePrendiville19}
S.~Peluse and S.~Prendiville.
\newblock Quantitative bounds in the non-linear {R}oth theorem.
\newblock {\em preprint}, 2019.
\newblock arXiv:1903.02592.

\bibitem{Prendiville17}
S.~Prendiville.
\newblock Quantitative bounds in the polynomial {S}zemer\'edi theorem: the
  homogeneous case.
\newblock {\em Discrete Anal.}, (5), 2017.

\bibitem{Rice19}
A.~Rice.
\newblock A maximal extension of the best-known bounds for the
  {F}urstenberg-{S}\'{a}rk\"{o}zy theorem.
\newblock {\em Acta Arith.}, 187(1):1--41, 2019.

\bibitem{Sarkozy78I}
A.~S\'ark\"ozy.
\newblock On difference sets of sequences of integers. {I}.
\newblock {\em Acta Math. Acad. Sci. Hungar.}, 31(1--2):125--149, 1978.

\bibitem{Sarkozy78II}
A.~S\'ark\"ozy.
\newblock On difference sets of sequences of integers. {III}.
\newblock {\em Acta Math. Acad. Sci. Hungar.}, 31:355--386, 1978.

\bibitem{Slijepcevic03}
S.~Slijep\v{c}evi\'c.
\newblock A polynomial {S}\'ark\"ozy-{F}urstenberg theorem with upper bounds.
\newblock {\em Acta Math. Hungar.}, 98(1-2):111--128, 2003.

\bibitem{Szemeredi75}
E.~Szemer\'edi.
\newblock On sets of integers containing no {$k$} elements in arithmetic
  progression.
\newblock {\em Acta Arith.}, 27:199--245, 1975.
\newblock Collection of articles in memory of Juri\u\i Vladimirovi\v c Linnik.

\bibitem{Tao12}
T.~Tao.
\newblock {\em Higher order {F}ourier analysis}, volume 142 of {\em Graduate
  Studies in Mathematics}.
\newblock American Mathematical Society, Providence, RI, 2012.

\bibitem{TaoZiegler08}
T.~Tao and T.~Ziegler.
\newblock The primes contain arbitrarily long polynomial progressions.
\newblock {\em Acta Math.}, 201(2):213--305, 2008.

\bibitem{TaoZiegler16}
T.~Tao and T.~Ziegler.
\newblock Concatenation theorems for anti-{G}owers-uniform functions and
  {H}ost-{K}ra characteristic factors.
\newblock {\em Discrete Anal.}, pages Paper No. 13, 60, 2016.

\bibitem{TaoZiegler18}
T.~Tao and T.~Ziegler.
\newblock Polynomial patterns in the primes.
\newblock {\em Forum Math. Pi}, 6:e1, 60, 2018.

\end{thebibliography}

\end{document}